\documentclass[11pt]{amsart}

\usepackage{latexsym}
\usepackage{url,cite}
\usepackage[colorlinks,linkcolor=blue,citecolor=blue]{hyperref}

\usepackage{makecell,multirow,array,rotating,lscape}
\usepackage{threeparttable}
\usepackage{mathrsfs,enumerate,mathdots}
\usepackage{extarrows,tikz}
\usepackage{makecell,subeqnarray}
\usepackage{amsfonts,amssymb,amsmath,cases}
\usepackage{graphicx,color,float,nonfloat,epstopdf,bbding,subfigure}
\usepackage{tensor}
\usepackage{stmaryrd}
\usepackage{latexsym}
\usepackage{url,cite,geometry}
\usepackage{makecell,multirow,array,rotating,lscape}
\usepackage{threeparttable}
\usepackage{mathrsfs,enumerate,mathdots}
\usepackage{extarrows,tikz}
\usepackage{algorithm,algorithmic} 
\usepackage{makecell,subeqnarray}
\usepackage{amsfonts,amssymb,amsmath,cases}
\usepackage{graphicx,color,float,nonfloat,epstopdf,bbding,subfigure,booktabs}
\DeclareMathAlphabet\mathbfcal{OMS}{cmsy}{b}{n}

\usepackage{tensor,stmaryrd,bm}
\newtheorem{definition}{Definition}[section]

\newtheorem{theorem}[definition]{Theorem}
\newtheorem{lemma}[definition]{Lemma}

\newtheorem{remark}[definition]{Remark}

\newtheorem{assumption}{Assumption}[section]

\newcommand{\IR}{\mathbb{R}}

\newcommand{\X}{\mathcal{X}}

\newcommand{\IS}{\mathcal{S}}
\newcommand{\B}{\mathcal{B}}

\newcommand{\G}{\mathcal{E}}

\newcommand{\Z}{\mathcal{Z}}
\newcommand{\IX}{\mathbfcal{X}}
\newcommand{\IY}{\mathbfcal{Y}}
\newcommand{\IT}{\mathbfcal{T}}
\newcommand{\IIN}{\mathbfcal{N}}
\newcommand{\bfA}{{\bm{A}}}
\newcommand{\bfB}{{\bm{B}}}
\newcommand{\bfC}{{\bm{C}}}
\newcommand{\bfP}{{\bm{P}}}
\newcommand{\bfQ}{{\bm{Q}}}
\newcommand{\bfH}{{\bm{H}}}

\newcommand{\bfz}{{\bm{z}}}

\newcommand{\bfy}{{\bm{y}}}
\newcommand{\bfx}{{\bm{x}}}
\newcommand{\bfd}{{\bm{d}}}
\newcommand{\bfxi}{{\bm{\xi}}}
\newcommand{\bfzeta}{{\bm{\zeta}}}

\newcommand{\lr}{\llbracket}
\newcommand{\rr}{\rrbracket}
\newcommand{\argmin}{\hbox{argmin}}

\begin{document}
	\baselineskip 18pt
\title[Extended ASAP for nonconvex nonsmooth problems]{Extended alternating structure-adapted proximal gradient algorithm for nonconvex nonsmooth problems}

\author{Ying Gao}
\address{School of Mathematical Sciences, Beihang University, Beijing, 100191, China.}
\email{yinggao@buaa.edu.cn}

\author{Chunfeng Cui}
\address{LMIB of the Ministry of Education, School of Mathematical Sciences, Beihang University, Beijing, 100191, China.}
\email{chunfengcui@buaa.edu.cn}

\author{Wenxing zhang}
\address{School of Mathematical Sciences,  University of Electronic Science and Technology of China, Chengdu 611731, China.}
\email{zhangwx@uestc.edu.cn}

\author{Deren Han}
\address{LMIB of the Ministry of Education, School of Mathematical Sciences, Beihang University, Beijing, 100191, China.}
\email{handr@buaa.edu.cn}



\begin{abstract}
Alternating structure-adapted proximal (ASAP) gradient algorithm (M.
Nikolova and P. Tan, SIAM J Optim, 29:2053-2078, 2019) 
has drawn much attention due to its efficiency in solving nonconvex nonsmooth optimization problems. However, the multiblock nonseparable structure confines the performance of ASAP to far-reaching practical problems, e.g., coupled tensor decomposition. In this paper, we propose an extended ASAP (eASAP) algorithm for nonconvex nonsmooth optimization whose objective is the sum of two nonseperable functions and a coupling one. By exploiting the blockwise restricted prox-regularity, eASAP is capable of  minimizing the objective whose coupling function is multiblock nonseparable. 
Moreover, we analyze the  global convergence of eASAP by virtue of the Aubin property on partial subdifferential mapping and the Kurdyka-{\L}ojasiewicz property on the objective. Furthermore, the sublinear convergence rate of eASAP is built upon the proximal point algorithmic framework under some mild conditions. 
Numerical simulations on multimodal data fusion demonstrate the compelling performance of the proposed method.
\end{abstract}

\maketitle

\section{Introduction}
In this paper, we focus on the structured optimization problem
\begin{align}\label{geP}
	\min\limits_{(\bfx,\bfy)\in\IR^{n}\times\IR^{m} } J(\bfx,\bfy):=F(\bfx_1,\ldots,\bfx_s)+G(\bfy_1,\ldots,\bfy_t)+H(\bfx,\bfy),
\end{align}
where $\bfx=(\bfx_1,\ldots,\bfx_s)\in\IR^n$ and  $\bfy=(\bfy_1,\ldots,\bfy_t)\in\IR^m$ are variables with subvector $\bfx_i\in\IR^{n_i}$ ($i=1,\ldots,s$) and  $\bfy_j\in\IR^{m_j}$ ($j=1,\ldots,t$), respectively; 
$F:\IR^{n}\to\IR_\infty:=\IR\cup\{+\infty\}$,  $G:\IR^{m}\to\IR_\infty$ and $H:\IR^{n}\times\IR^{m}\to\IR_\infty$  are proper closed functions. Problem \eqref{geP} captures a plethora of  models in diversified applications, such as candidate extraction  in image colorization \cite{color2015,tan2019inertial}, blind source separation in signal processing \cite{beck2016alternating,chung2010efficient}, and coupled tensor decomposition in multimodal data fusion   \cite{chatzichristos2022early,roald2022admm}.

Over the past decades, we have witnessed a flurry of research activities on structured optimization problem \eqref{geP} with two blocks, i.e., $s=t=1$. A canonical method for solving \eqref{geP} is the block coordinate descent method (BCD), which minimizes $J$ alternatively over $\bfx$ and $\bfy$.  
The convergence of BCD can be guaranteed under some conditions \cite{bertsekas2015parallel},
e.g., strict convexity of the objective $J$ or boundedness of its level set. Particularly, when $J$ is continuously differentiable, Bertsekas  \cite{bertsekas1997nonlinear} stated that the subproblems of BCD admit unique optima and the sequence $\{(\bfx^k,\bfy^k)\}_{k=0}^\infty$ generated by BCD converges to the critical point of \eqref{geP}. To weaken the convergent requirements of BCD, Auslender \cite{auslender1992asymptotic} proposed the following proximal BCD
\begin{subnumcases}{\label{pbcd}}
	\bfx^{k+1}\in\underset{\bfx\in\IR^n}{\argmin} \Big\{ J(\bfx,\bfy^k)+\frac{\tau}{2}\|\bfx-\bfx^k\|_2^2\Big\},\\
	\bfy^{k+1}\in\underset{\bfy\in\IR^m}{\argmin} \Big\{J(\bfx^{k+1},\bfy)+\frac{\sigma}{2}\|\bfy-\bfy^k\|_2^2\Big\},
\end{subnumcases}
where $\tau>0$ and $\sigma>0$ are stepsizes. Therein, the  convergence of \eqref{pbcd} was analyzed when $J$ is  convex. By exploiting
the Kurdyka-{\L}ojasiewicz (K\L) property, Attouch et al. \cite{attouch2010proximal} studied 
the global convergence of \eqref{pbcd} when $J$ is nonconvex nonsmooth. 
Thereby, the rationale of  K{\L} property motivates a great deal of theoretical and algorithmic advances for solving nonconvex nonsmooth optimization (see e.g., \cite{bot2019proximal,boct2020proximal,li2018calculus,GuoPong2015,yang2017alternating} and references therein). 

Practically, the subproblems in \eqref{pbcd} are nontrivial to admit closed-form solutions. When the component function, e.g., $F$, $G$ or $H$ is continuously differentiable, the linearization on these differentiable component functions facilitate the development of efficient numerical algorithms. For instance, Bolte et al. \cite{bolte2014proximal} proposed the following proximal alternating linearized minimization (PALM) algorithm when $H$ is continuously differentiable
\begin{subnumcases}{\label{palm}}
	\bfx^{k+1}\in\underset{\bfx\in\IR^n}{\argmin}\Big\{  F(\bfx)+\langle \bfx,\nabla_{\bfx} H(\bfx^k,\bfy^k)\rangle+\frac{\tau^k}{2}\|\bfx-\bfx^k\|_2^2\Big\},\\
	\bfy^{k+1}\in\underset{\bfy\in\IR^m}{\argmin}\Big\{ G(\bfy)+\langle \bfy,\nabla_{\bfy} H(\bfx^{k+1},\bfy^k)\rangle+\frac{\sigma^k}{2}\|\bfy-\bfy^k\|_2^2\Big\},
\end{subnumcases}
where $\tau^k>0$ and $\sigma^k>0$ are the varying stepsizes determined by the Lipschitz constants of $\nabla H(\cdot,\bfy^k)$ and $\nabla H(\bfx^{k+1},\cdot)$, respectively. With the 
K{\L} property, it was verified that any bounded sequence $\{ (\bfx^k,\bfy^k)\}_{k=0}^\infty$ generated by \eqref{palm} can converge globally to the critical point of \eqref{geP}. Afterwards, we have witnessed numerous   ameliorations of PALM by deploying inexact approximation, inertial acceleration, and stochastic strategies (see, e.g.,\cite{hu2023convergence,pock2016inertial,li2015accelerated,driggs2021stochastic} and references therein). Alternatively, Nikolova and Tan \cite{nikolova2019alternating} devised an alternating structure-adapted proximal (ASAP) gradient descent algorithm when $F$, $G$ admit Lipschitz continuous gradients
\begin{subnumcases}{\label{asap}}
	\bfx^{k+1}\in\underset{\bfx\in\IR^n}{\argmin}\Big\{ H(\bfx,\bfy^k)+\langle \bfx,\nabla_{\bfx} F(\bfx^k) \rangle+\frac{\tau}{2}\|\bfx-\bfx^k\|_2^2\Big\},
	\\
	\bfy^{k+1}\in\underset{\bfy\in\IR^m}{\argmin}\Big\{  H(\bfx^{k+1},\bfy)+\langle \bfy,\nabla_{\bfy}G(\bfy^k) \rangle+\frac{\sigma}{2}\|\bfy-\bfy^k\|_2^2\Big\}.
\end{subnumcases}
In contrast to PALM \eqref{palm}, the stepsizes $(\tau,\sigma)$ in ASAP \eqref{asap} involve the Lipschitz constants of $\nabla F$ and $\nabla G$. The global convergence of ASAP was analyzed in \cite[Theorem 5.15]{nikolova2019alternating} under the assumptions: (C1) $H$ can be split as the sum of a nonsmooth function and a partially differential function in variable $\bfx$; and (C2) the objective $J$ admits the K{\L} property at the critical point. Thenceforth, variants of ASAP were developed by deploying inertial acceleration, extrapolation techniques, and stochastic strategies (see e.g.,  \cite{tan2019inertial,yang2022some,gao2023alternating,jia2024stochastic} and references therein). It is worth noting that assumptions (C1)-(C2) are indispensable for the convergence analysis of these ASAP variants. 

PALM and ASAP can be easily extended to solve \eqref{geP} with 
\begin{align}\label{exten2}
	F(\bfx_1,\ldots,\bfx_s)=\sum\limits_{i=1}^s F_i(\bfx_i),\quad G(\bfy_1,\ldots,\bfy_t)=\sum\limits_{j=1}^t G_j(\bfy_t).
\end{align}  
The interested reader can refer to, e.g.,  \cite{xu2013block,nikolova2019alternating}, for more discussions on the applications of PALM and ASAP to \eqref{geP} with separable $F$ and $G$ as in \eqref{exten2}. However, to the best of our knowledge, few literature were devoted to solving \eqref{geP} with   nonseparable $F$ and $G$, particularly, the case of $H$ is nonseparable (possibly nonconvex and nonsmooth) whilst $F$, $G$ are continuously differentiable.

In this paper, pursuing the track of ASAP in \cite{nikolova2019alternating}, we focus on an extended ASAP algorithm (``eASAP"  for short) to solve \eqref{geP} with the  multiblock nonseparable structure. The  Gauss-Seidel algorithmic framework is exploited to minimize $\bfx_i$- and $\bfy_j$-subproblems  cyclically. Besides, the aforementioned assumption (C1) for guaranteeing the global convergence of ASAP is weaken by deploying some novel analytical techniques. The contributions of this paper are summarized as follows: 
\begin{enumerate}

	\item[1.] Nonseparable extension. Compared to the traditional separable extension \eqref{exten2}, we generalize the ASAP to solve \eqref{geP} with multiblock nonseparable structure, which subsumes a variety of models in, e.g., coupled tensor decomposition. As a peculiarity, the proposed eASAP can reduce to ASAP \eqref{asap}. 
	
	\item[2.] Weaker assumption. The global convergence of ASAP \eqref{asap} requires the splitting property on the component function $H$. However, under the blockwise restricted prox-regularity (see Definition \ref{block_h})  on $H$ in \eqref{geP}, we verify that eASAP can be applicable to diversified nonconvex nonsmooth optimization, even if $H$ is nonseparable. Furthermore, we prove that any limit point of the sequence generated by eASAP is a critical point of \eqref{geP}.
	
	\item[3.] Novel analytical techniques. Essentially, there are two difficulties in the convergence analysis of ASAP \eqref{asap}: (i) closedness of the partial subdifferentials of $J$;  and   (ii) boundedness of subdifferentials. Under the Aubin property on partial subdifferential mapping and the K{\L}  property on objective, we establish the global convergence of eASAP. Moreover, we build upon the sublinear convergence rate of eASAP by virtue of the proximal point algorithmic framework.  
	
\end{enumerate} 

The applications of \eqref{geP} to coupled tensor decomposition are recently  gaining much attention in multimodal data fusion \cite{chatzichristos2022early,kanatsoulis2018hyperspectral,li2018fusing}. A large number of practical models fall into the form of \eqref{geP}. For instance, the Bayesian framework based coupled tensor decomposition model  in \cite{farias2016exploring} reads
\begin{align}\label{optcp}
	\min\limits_{\bfA,\bfB}\; \frac{1}{2}\Big\|\IY-\lr \bfA_1,\bfA_2,\bfA_3\rr\Big\|_F^2+\frac{1}{2}\Big\|\IY'-\lr \bfB_1,\bfB_2,\bfB_3\rr\Big\|_F^2+\mu H(\bfA,\bfB),
\end{align}
where $\IY\in\IR^{I_1\times I_2\times I_3}$ and $ \IY'\in\IR^{J_1\times J_2 \times J_3}$ are three-order  tensors; $\bfA_l\in\IR^{I_l\times r}$  ($l=1,2,3$) and  $\bfB_l\in\IR^{J_l\times r}$ ($l=1,2,3$) are factor matrices from the CANDECOMP/PARAFAC (CP) decomposition of tensors $\IY$ and $\IY'$, respectively; $\bfA=[\bfA_1;\bfA_2;\bfA_3]$
and $\bfB=[\bfB_1;\bfB_2;\bfB_3]$ denote the stacked matrices;  $H$ is a nonseparable function in variables $\bfA$, $\bfB$, which quantifies the statistical dependence between $\bfA$ and $\bfB$; and $\mu>0$ accounts for the trade-off among objectives;  Some preliminaries of tensor are elaborated in section \ref{sec:num}. Table \ref{form1} lists several  canonical coupling functions for $H$ in \eqref{optcp}. Therein, the ``vec" denotes the vectorization of a matrix.
$\bfP$ and $\bfQ$ are given structured matrix and positive definite matrix, respectively. Actually, \eqref{optcp} falls into the form of \eqref{geP} with nonseparable (possibly nonconvex nonsmooth) coupling function $H$. Accordingly, it is desirable to develop an efficient and provably convergent algorithm for \eqref{geP} by extending ASAP algorithm. 
\begin{table}[H]
	\caption{Several choices of  $H$ in \eqref{optcp}.}
	\renewcommand\arraystretch{1.8}
	\centering
	\scalebox{1}{
		\begin{tabular}{ccc}
			\toprule[1pt]
			Coupling distribution & $H(\bfA,\bfB)$ & Properties \\
			\midrule[0.8pt] 
			Hybrid Gauss &$\big\|\textnormal{vec}(\bfA- \bfP\bfB)\big\|_2^2$& convex and  smooth\\
			Joint Gauss &$\big\|\hbox{vec}([\bfA;\bfB])\big\|_\bfQ^2$& convex and  smooth \\
			Laplacian &$ \big\|\hbox{vec}(\bfA-\bfP\bfB)\big\|_1$& convex  and nonsmooth\\
			Uniform  & $\big\|\hbox{vec}(\bfA-\bfP \bfB)\big\|_\infty$ & convex and nonsmooth\\
			Cauchy-type  & $\big\|\hbox{vec}(\bfA-\bfP\bfB)\big\|_p^p$ $(0<p<1)$ & nonconvex and nonsmooth\\
			\bottomrule[1pt]
	\end{tabular}}
	\label{form1}
\end{table}

The rest of this paper is organized as follows. In section \ref{sec:p}, we summarize some preliminaries for the upcoming discussions. We present the iterative scheme of eASAP in section \ref{sec:al}, followed by its global convergence and convergence rate analysis under some prerequisites in  section \ref{sec:analysis}. In section \ref{sec:num}, we test a class of coupled tensor decomposition problems on both synthetic and real data to demonstrate the numerical performance of the proposed method. Finally, some concluding remarks are drawn in section \ref{sect6}.

\section{Preliminaries}\label{sec:p}
For $\bfx:=(x_1,x_2,\ldots,x_n)^\top\in\IR^n$, let $\|\bfx\|_p:=(\sum_{i=1}^n|x_i|^p)^{1/p}$ $(1\le p<\infty)$ and $\|\bfx\|_{\infty}:=\max_{i=1,\ldots,n}|x_i|$ denote the $\ell^p$- and $\ell^\infty$- norm of $\bfx$, respectively. Particularly, $\|\bfx\|:=\|\bfx\|_2$ for brevity. Let $\bm{I}$ (resp., $\bf{0}$) denote the identity matrix (resp., zero vector/matrix) whose dimension can be clear from the context. The distance from $\bfx\in\IR^n$ to a set $\Omega\subset\IR^n$ is defined by $\textnormal{dist}(\bfx,\Omega):=\inf_{\bfy\in \Omega}\|\bfx-\bfy\|$. $\B_p^r(\bfx):=\{\bfx'\in\IR^n\mid\|\bfx'-\bfx\|_p\leq r\}$
denotes the $\ell^p$-norm ball centered at $\bfx$ with radius $r>0$. For brevity, we abuse  $\B^r(\bfx)$ or $\B(\bfx)$ as the neighborhood of $\bfx$  regardless of its metric and/or radius.

The function $f:\IR^n\to\IR_\infty$ is proper if $\textnormal{dom}(f):=\{\bfx\in\IR^n\mid f(\bfx)<+\infty\}$ is nonempty, $f$ is closed if $\hbox{epi}(f):=\{(\bfx,r)\in\IR^n\times\IR\mid f(\bfx)\le r\}$ is closed, and $f$ is lower bounded if $\inf_{\bfx\in\textnormal{dom}(f)} f(\bfx)>-\infty$. The proximity of a proper closed function $f$, denoted by $\hbox{prox}_f$, is defined by 
\begin{align}\label{proximity}	\hbox{prox}_f(\bfx):=(\bm{I}+\partial f)^{-1}(\bfx)=\underset{\bfx'\in\IR^n}{\hbox{argmin}}\;\Big\{ f(\bfx')+\frac{1}{2}\|\bfx'-\bfx\|^2\Big\}\quad\forall\bfx\in\IR^n,
\end{align}
where $\partial f:\IR^n\to2^{\IR^n}$ is the subdifferential of $f$ as follows.

\subsection{Subdifferential and partial subdifferential}
The definitions of (partial) subdifferential 
can be referred to, e.g., \cite{Varia-Ana,nikolova2019alternating}.
\begin{definition}[\cite{Varia-Ana}]\label{def:subdif}
	Let $f:\IR^n\to\IR_\infty$ be proper closed.
	\begin{itemize}
		\item[(i)] 
		The Fr\'echet subdifferential of $f$ at $\bfx\in\textnormal{dom}(f)$  is defined by
		\begin{align*}
			\hat{\partial}f(\bfx)=\{\bfd\in\IR^n\mid  f(\bfy)\ge f(\bfx)+\langle \bfd,\bfy-\bfx \rangle+o(\|\bfy-\bfx\|)\}.
		\end{align*}
		\item[(ii)] The limiting subdifferential of $f$ at $\bfx\in\textnormal{dom}(f)$ is defined by 
		\begin{align*}
			\partial f(\bfx)=\{\bfd\in\IR^n\mid \exists\; \bfx^k\to \bfx,\;f(\bfx^k)\to f(\bfx),\; \bfd^k\in\hat{\partial} f(\bfx^k)\to\bfd,\;\textnormal{as}\; k\to\infty\}.
		\end{align*}
	\end{itemize}
\end{definition}
\noindent Notationally, $\hat{\partial}f(\bfx)=\partial f(\bfx)=\emptyset$ for all $\bfx\notin\textnormal{dom}(f)$. It follows from Definition \ref{def:subdif} that $\hat{\partial}f(\bfx)\subset\partial f(\bfx)$ for all    $\bfx\in\textnormal{dom}(f)$. Moreover, $\hat{\partial}f(\bfx)$ is closed  convex whilst $\partial f(\bfx)$ is merely closed. Particularly, if $f$ is convex, then
\begin{align*}
	\hat{\partial}f(\bfx)=\partial f(\bfx)=\{\bfd\in\IR^n\mid f(\bfy)\ge f(\bfx)+\langle \bfd,\bfy-\bfx\rangle\;\;\;\forall \bfy\in\IR^n\}.
\end{align*}

The partial subdifferential plays a crucial role in nonsmooth analysis. The partial subdifferential of a proper closed function $h:\IR^n\times\IR^m\to\IR_\infty$ with respect to $\bfx$ (or equivalently, the subdifferential of $h(\cdot,\bfy)$) is denoted by $\partial_\bfx h$. The parametric closedness of partial subdifferential is defined as follows. 	\begin{definition}[\cite{nikolova2019alternating}] 
	Let $h:\IR^n\times\IR^m\to\IR_\infty$ be proper closed and let  $\{(\bfx^k,\bfy^k)\}_{k=0}^\infty\subset\textnormal{dom}(h)$ be a sequence converging to $(\bar{\bfx},\bar{\bfy})$.  The partial subdifferential $\partial_\bfx h$ is said to be parametrically closed at $(\bar{\bfx},\bar{\bfy})$ with respect to $\{ (\bfx^k,\bfy^k)\}_{k=0}^\infty$ if, for any sequence $\{\bfd^k_\bfx\}_{k=0}^\infty$ satisfying $\bfd^k_\bfx\in\partial_\bfx h(\bfx^k,\bfy^k)$ and $\lim\limits_{k\to\infty}\bfd^k_\bfx=\bar{\bfd}_\bfx$, there is $\bar{\bfd}_\bfx\in\partial_\bfx h(\bar{\bfx},\bar{\bfy})$.

\end{definition}

The parametric closedness of $\partial_{\bfy}h$ can be defined accordingly. Note that the parametric closedness of partial subdifferential is possibly invalid for generic function, even if the subdifferential is closed (see, e.g., \cite[Example 4.5]{nikolova2019alternating} for counterexample). We now present a class of functions with parametrically closed partial subdifferential. 

\subsection{Restricted prox-regularity}\label{sec:subR}
Let $c>0$ and  $f:\IR^n\to\IR_\infty$ be proper   closed. We define the  exclusion set by 
\begin{align*}
	\mathcal{S}_{f}^c:=\{\bfx\in\textnormal{dom}(f)\mid \|\bfd\|>c\;\;\;\forall \bfd\in\partial f(\bfx)\},
\end{align*}
which contains the points in  $\hbox{dom}(f)$ that the norms of subgradients are greater than a threshold $c$. For all $c_1\geq c_2>0$, we have $\mathcal{S}^{c_1}_f\subseteq\mathcal{S}^{c_2}_f$. 
\begin{definition}[\cite{wang2019global}]\label{res-prox}
	A proper closed function $f:\IR^n\to\IR_\infty$ is said to be restricted prox-regular if,  for any $c>0$ and bounded set  $\Omega\subset\textnormal{dom}(f)$, there exists $\gamma>0$ such that
	\begin{align*}
		f(\bfy)\ge f(\bfx)+\langle \bfd,\bfy-\bfx  \rangle-\frac{\gamma}{2}\|\bfy-\bfx\|^2\;\;\;\forall\bfx\in \Omega\backslash \mathcal{S}^{c}_f,\;\bfy\in \Omega,\;\bfd\in\partial f(\bfx),\;\|\bfd\|\le c.
	\end{align*} 
\end{definition}
A large amount of functions are restricted prox-regular, e.g., convex function, semiconvex function, smooth function with Lipschitz continuous gradient,   $\ell^p$-quasi-norm function ($0<p<1$), Schatten-$p$ quasi-norm function
($0<p<1$), and indicator function of compact smooth manifold (see \cite{wang2019global} for more details). 

\begin{definition}[\cite{bolte2010characterizations}]\label{semi}
	A proper closed function $f:\IR^n\to\IR_\infty$ is said to be semiconvex with modulus $\gamma>0$ if  $f+\frac{\gamma}{2}\|\cdot\|^2$ is convex, i.e., 
	\begin{align}
		f(\bfy)\ge f(\bfx)+\langle\bfd,\bfy-\bfx\rangle-\frac{\gamma}{2}\|\bfy-\bfx\|^2\;\;\; \forall \bfx, \bfy\in\textnormal{dom}(f)\;\;\hbox{and}\;\;\bfd\in\partial f(\bfx).
	\end{align} 
\end{definition}
In contrast to Definition \ref{res-prox}, the semiconvexity is essentially a special case of restricted prox-regularity with $\IS_{f}^c=\emptyset$.

Let $C^{k,p}_L(\IR^n)$ denote the set of $k$ times continuously differentiable functions on $\IR^n$ whose $p$th order derivatives are $L$-Lipschitz continuous. The following lemma warrants a sufficient descent property of the functions in $C_L^{1,1}(\IR^n)$  (see, e.g.,  \cite{bertsekas1997nonlinear}).

\begin{lemma}\label{nablaf}
	Let $f\in C_{L_f}^{1,1}(\IR^n)$. Then
	\begin{align*}
		\big\lvert f(\bfy)-f(\bfx)-\langle \nabla f(\bfx),\bfy-\bfx\rangle\big\rvert\le \frac{L_f}{2}\|\bfy-\bfx\|^2\quad\forall \bfx,\bfy\in\IR^n.
	\end{align*}
\end{lemma}

\begin{lemma}\label{Sumlipshitz}
	For a bounded set $\Omega\subset\IR^n$, let $f_1:\IR^n\to\IR_{\infty}$ be restricted prox-regular on $\Omega$ and let $f_2\in C_{L}^{1,1}(\Omega)$. Then $f=f_1+f_2$ is also restricted prox-regular on $\Omega$.
\end{lemma}

\begin{proof}
	Since $f_2\in C_{L}^{1,1}(\Omega)$, it follows from Lemma \ref{nablaf} that
	\begin{align}\label{nablalip}
		f_2(\bfy)\ge f_2(\bfx)+\langle \nabla f_2(\bfx),\bfy-\bfx\rangle-\frac{L}{2}\|\bfy-\bfx\|^2\quad \forall \bfx,\bfy\in \Omega.
	\end{align} 
	Since $\Omega$ is bounded, there exists $c_2>0$ such that $\|\nabla f_2(\bfx)\|\le c_2$ for all $ \bfx\in\Omega$.  Since $f_1$ is restricted prox-regular on $\Omega$, then, for all $c_1>c_2$, there exist $\gamma_1>0$ and an exclusion set $\mathcal{S}^{c_1}_{f_1}$ such that
	\begin{align}\label{resregular}
		f_1(\bfy)\ge f_1(\bfx)+\langle \bfd_1,\bfy-\bfx\rangle-\frac{\gamma_1}{2}\|\bfx-\bfy\|^2\quad\forall\bfx\in \Omega\backslash \mathcal{S}^{c_1}_{f_1},\;\bfy\in \Omega, 
	\end{align}
	for all $\bfd_1\in\partial f_1(\bfx)$ with $ \|\bfd_1\|\le c_1$. For any $\bfx\in \Omega\backslash\IS^{c_3}_f=\{ \bfx\in \Omega\mid \|\bfd_1+\nabla f_2(\bfx)\|\le c_3 \}$, by letting  $c_3=c_1-c_2$, we derive
	\begin{align*}
		\|\bfd_1\|=\|\bfd_1+\nabla f_2(\bfx)-\nabla f_2(\bfx)\|&\le\|\bfd_1+\nabla f_2(\bfx)\|+\|\nabla f_2(\bfx)\|\le c_1,
	\end{align*}
	which indicates $\Omega\backslash \mathcal{S}^{c_3}_f\subseteq \Omega\backslash \mathcal{S}^{c_1}_{f_1}$. Thus, summing \eqref{nablalip}-\eqref{resregular}  yields
	\begin{align*}
		f(\bfy)\ge f(\bfx)+\langle \bfd_1+\nabla f_2(\bfx),\bfy-\bfx\rangle -(\frac{\gamma_1}{2}+L)\|\bfx-\bfy\|^2\;\;\;\forall \bfx\in \Omega\backslash \mathcal{S}^{c_3}_f,\; \bfy\in \Omega.
	\end{align*}
	Namely, $f=f_1+f_2$ is a restricted prox-regular function on $\Omega$.  
\end{proof}

We define blockwise restricted prox-regular and blockwise semiconvex  functions as follows.
\begin{definition}\label{block_h}
	A proper closed function $h:\IR^n\times\IR^m\to\IR_\infty$ is said to be blockwise restricted prox-regular if $h(\cdot,\bfy)$ and $h(\bfx,\cdot)$ are restricted prox-regular. Particularly, $h$ is said to be blockwise semiconvex if  $h(\cdot,\bfy)$ and $h(\bfx,\cdot)$ are semiconvex. 
\end{definition}

The following lemma states the parametric closedness of partial subdifferentials for blockwise restricted prox-regular functions. 
\begin{lemma}\label{subdfclosed}
	Let $h:\IR^n\times\IR^m\to\IR_\infty$ be blockwise restricted prox-regular and let $\{(\bfx^k,\bfy^k)\}_{k=0}^\infty$ be the sequence converging to $(\bar{\bfx},\bar{\bfy})\in\mathrm{dom}(h)$. For any $c>0$, $\bfd^k_\bfx\in \partial_{\bfx} h(\bfx^k,\bfy^k)$ and $\bfd_\bfy^k\in\partial_{\bfy} h(\bfx^k,\bfy^k)$, if
	\begin{align*}
		\lim_{k\to\infty}\bfd_\bfx^k=\bar{\bfd}_\bfx\;\;\hbox{and}\;\;\lim_{k\to\infty}\bfd_\bfy^k=\bar{\bfd}_\bfy
	\end{align*}
	with $\|\bar{\bfd}_\bfx\|<c$ and $\|\bar{\bfd}_\bfy\|< c$, then   $\bar{\bfd}_\bfx\in \partial_{\bfx}h(\bar{\bfx},\bar{\bfy})$ and  $\bar{\bfd}_\bfy\in\partial_{\bfy}h(\bar{\bfx},\bar{\bfy})$. Namely, $\partial_{\bfx}h$ and $\partial_{\bfy}h$ are parametrically closed at $(\bar{\bfx},\bar{\bfy})$ with respect to the sequence $\{(\bfx^k,\bfy^k)\}_{k=0}^\infty$.
\end{lemma}
\begin{proof}
	Recall that $h$ is blockwise restricted prox-regular. Accordingly, for any $c>0$, bounded set $\Omega\subset\textnormal{dom}(h(\cdot,\bfy))$ and exclusion set $\mathcal{S}^{c}_{h(\cdot,\bfy)}=\{\bfx\in\textnormal{dom}(h(\cdot,\bfy)) \mid \|\bfd\|>c\;\;\forall\bfd\in\partial_\bfx h(\bfx,\bfy)\}$, there exist $\gamma_1>0$ and a sequence $\{ (\bfx^{k_l},\bfy^{k_l})\}_{l=0}^\infty$ such that
	\begin{align}\label{restrict_prf}
		h(\bfx,\bfy^{k_l})\ge h(\bfx^{k_l},\bfy^{k_l})+\langle \bfd^{k_l}_\bfx,\bfx-\bfx^{k_l} \rangle-\frac{\gamma_1}{2}\|\bfx-\bfx^{k_l}\|^2\quad\forall \bfx^{k_l}\in\Omega\backslash \mathcal{S}^{c}_{h(\cdot,\bfy)},
	\end{align}
	where $\bfx\in\Omega$ and $\|\bfd^{k_l}_\bfx\|\le c$. The limit points $\bar{\bfd}_\bfx,\; \bar{\bfd}_\bfy$ are bounded by $c$. Hence, the sequence $\{(\bfx^{k_l},\bfy^{k_l})\}_{l=0}^\infty$ satisfying \eqref{restrict_prf} is a subsequence of $\{(\bfx^k,\bfy^k)\}_{k=0}^\infty$. By the fact that all subsequences of a convergent sequence admit identical limit point, we pass $l\to\infty$ in \eqref{restrict_prf} to yield
	\begin{align*}
		h(\bfx,\bar{\bfy})\ge h(\bar{\bfx},\bar{\bfy})+\langle \bar{\bfd}_\bfx,\bfx-\bar{\bfx} \rangle-\frac{\gamma_1}{2}\|\bfx-\bar{\bfx}\|^2.
	\end{align*}
	By Definition \ref{def:subdif}~(i), we have $\bar{\bfd}_\bfx\in\hat{\partial}_\bfx h(\bar{\bfx},\bar{\bfy})\subset\partial_\bfx h(\bar{\bfx},\bar{\bfy})$. Hence, $\partial_\bfx h$ is parametrically closed at  $(\bar{\bfx},\bar{\bfy})$ with respect to the sequence $\{(\bfx^k,\bfy^k)\}_{k=0}^\infty$.
	The above deduction can also be extended for the $\bfy$ block, which completes the proof.
\end{proof}

Note that although the blockwise restricted prox-regularity in Lemma \ref{subdfclosed} is described for $h$ with two block variables $\bfx$ and $\bfy$, the statement can be easily generalized to $h$ with multiblock variables.

\subsection{K{\L} property}
The K{\L} property is a powerful tool for analyzing nonconvex nonsmooth optimization (see, e.g. \cite{bolte2014proximal}). For any $-\infty<  c_1<c_2\le\infty$, a sublevel set of $f$ is defined by 
\begin{align*}
	[c_1<f<c_2]:=\{\bfx\in\IR^n\mid c_1<f(\bfx)<c_2\}.
\end{align*}
For any $\eta\in(0,\infty]$, let $\Phi_\eta$ denote the class of continuous concave functions $\varphi:[0,\eta)\to\IR_{+}$\footnote{$\IR_+$ denotes the positive scalar set.} satisfying
\begin{enumerate}[(i)]
	\item $\varphi$ is continuous at origin and $\varphi(0)=0$;
	\item $\varphi$ is continuously differentiable on $(0,\eta)$;
	\item $\varphi'(t)>0$ for all $t\in(0,\eta)$.
\end{enumerate}

\begin{definition}[\cite{attouch2010proximal}]\label{KL}
	A proper closed function $f:\IR^n\to\IR_{\infty}$ admits the K{\L}  property at $\bar{\bfx}\in\textnormal{dom}(\partial f):=\{\bfx\in\IR^n\mid \partial f(\bfx)\ne\emptyset\}$ if there exist $\eta\in(0,\infty]$, $\varphi\in\Phi_{\eta}$ and a neighbourhood of $\bar{\bfx}$ (denoted by $\mathcal{B}(\bar{\bfx})$) such that 
	\begin{align}\label{kl}
		\varphi'(f(\bfx)-f(\bar{\bfx}))\textnormal{dist}(0,\partial f(\bfx))\ge1
	\end{align}
	for all $\bfx\in\mathcal{B}(\bar{\bfx})\cap [f(\bar{\bfx})<f(\bfx)<f(\bar{\bfx})+\eta]$. 
	Moreover, $f$ is called a K{\L} function if it admits the K{\L} property for all  $\bar{\bfx}\in\textnormal{dom}(\partial f)$.	
\end{definition}

The following lemma states the uniformized K{\L}  property on the compact set \cite{bolte2014proximal}.

\begin{lemma}\label{uniKL}
	Let $\Omega\subset\IR^n$ be  compact and let $f:\IR^n\to\IR_\infty$ be proper closed with K{\L} property on  $\Omega$. If $f$ is a constant for all $\bfx\in\Omega$, then there exist $\epsilon>0$, $\eta>0$ and $\varphi\in\Phi_{\eta}$ such that \eqref{kl} holds for all $\bar{\bfx}\in   		\{\bfx\in\IR^n \mid \textnormal{dist}(\bfx,\Omega)<\epsilon \}\cap \{f(\bar{\bfx})<f(\bfx)<f(\bar{\bfx})+\eta \}$.
\end{lemma}

\subsection{Aubin property}
Let $T:\IR^n\to2^{\IR^m}$ be a set-valued function. The graph of $T$ is defined by   $\mathrm{gph}(T):=\{(\bfx,\bfy)\in\IR^n\times\IR^m\mid \bfy\in T(\bfx)\}$. 
\begin{definition}[\cite{Varia-Ana}]
	For any $\X\subset\IR^n$ and $\mathcal{Y}\subset\IR^m$,
	a set-valued function $T:\IR^n\to2^{\IR^m}$  is said to be {\it Lipschitz-like} on $\X$ relative to $\mathcal{Y}$ if  there exists  $L_T>0$ such that
	\begin{align}\label{LipLik:a}
		T(\bfx)\cap\mathcal{Y}\subset T(\bfx')+L_T\|\bfx-\bfx'\|\B^1\quad\forall \bfx',\bfx\in\X.
	\end{align}
	For any $(\bar{\bfx},\bar{\bfy})\in\mathrm{gph}(T)$,  $T$ admits {\it Aubin property} (or {\it locally Lipschitz-like})
	around  $(\bar{\bfx},\bar{\bfy})$ if \eqref{LipLik:a} holds with  $\X=\B(\bar{\bfx})$ and $\mathcal{Y}=\B(\bar{\bfy})$, whilst  $T$ is said to be $L_T$-Lipschitz continuous on $\X$ if \eqref{LipLik:a} holds with $\mathcal{Y}=\IR^m$. Particularly, if $T$ is a single-valued function, the  $L_T$-Lipschitz continuity on $\X$ reduces to 
	\begin{align*}
		\|T(\bfx)-T(\bfx')\|\leq L_T\|\bfx-\bfx'\|\quad\forall \bfx,\bfx'\in\IR^n.
	\end{align*}
\end{definition}
\noindent For the clarity of description, the above definitions can be presented under the distance metric.  Concretely, $T$ admits  Aubin property around $(\bar{\bfx},\bar{\bfy})\in\hbox{gph}(T)$ if there exist $\B(\bar{\bfx})$, $\B(\bar{\bfy})$ and  $L_T>0$ such that
\begin{align*}
	\textnormal{dist}(\bfy,T(\bfx))\leq L_T\textnormal{dist}(\bfx,\bfx')\;\;\;\; \forall \bfx\in\B(\bar{\bfx}),\bfy\in\B(\bar{\bfy})\cap T(\bfx').
\end{align*}

\section{Extended ASAP algorithm}\label{sec:al}


We now present the extended ASAP (eASAP) for solving \eqref{geP}. For notational convenience, we denote $\bfz:=(\bfx,\bfy)\in\IR^{n+m}$. Furthermore, let the symbols, e.g.,
\begin{align}\label{abbrxsubvector}
	\bfx_{\leq i}:=(\bfx_1,\ldots,\bfx_i), \;\; \bfx_{>i}:=(\bfx_{i+1},\ldots,\bfx_s)\;\;\hbox{and}\;\;\bfx_{-i}:=(\bfx_{<i},\bfx_{>i})
\end{align}
be the truncations of $\bfx$ by index $i=1,\ldots,s$. Extremely, we denote  $\bfx_{<1}=\bfx_{>s}=\emptyset$. Given $\bfx_{-i}\in\IR^{n-n_i}$ and $\bfy_{-j}\in\IR^{m-m_j}$, let  $F_i:\IR^{n_i}\to\IR_\infty$ ($i=1,\ldots,s$) and  $G_j:\IR^{m_j}\to\IR_\infty$ ($j=1,\ldots,t$) be the ``partial'' functions of $F$ and $G$, which are defined by
\begin{align}\label{notation:allB}
	F_i(\cdot):=F(\bfx_{<i},\cdot,\bfx_{>i}) \;\;\;\hbox{and}\;\;\;
	G_j(\cdot):=G(\bfy_{<j},\cdot,\bfy_{>j}).
\end{align}
Throughout, we make the following assumptions on the proposed eASAP method.
\begin{assumption}\label{assum1}
	\begin{itemize}
		\item[]
		\item[(i)]  $F_i\in C_{\mu_i}^{1,1}(\IR^{n_i})$ for all $i=1,\ldots,s$, and $G_j\in C_{\nu_j}^{1,1}(\IR^{m_j})$ for all $j=1,\ldots,t$.
		\item[(ii)] $H:\IR^n\times \IR^m\to\IR_\infty$ is proper closed and lower bounded.
		\item[(iii)] $J:\IR^n\times\IR^{m}\to\IR_{\infty}$ is lower bounded. 
	\end{itemize}    
\end{assumption} 
\begin{remark}
	By the definitions of $F_i$ and $G_j$ in  \eqref{notation:allB}, the Lipschitz constants $\mu_i$ and $\nu_j$ in Assumption \ref{assum1} (i) should be relevant to $\bfx_{-i}$ and $\bfy_{-j}$, respectively. For instance, given  $\bfx_{-i}\in\IR^{n-n_i}$,    $F_i\in C_{\mu_i}^{1,1}(\IR^{n_i})$ implies 
	\begin{align*}
		\|\nabla F_i(\bfx_i)-\nabla F_i(\bfx'_i)\|\le \mu_i(\bfx_{-i})\|\bfx_i-\bfx'_i\|\quad \forall \bfx_i,\bfx'_i\in\IR^{n_i}.
	\end{align*}
	Hereafter, for notational convenience, we denote 
	\begin{subequations}\label{Ldef}
		\begin{align}
			&\bar{\mu}_i:=\sup\{ \mu_i(\bfx_{-i})\;\;\forall\bfx_{-i}\in\IR^{n-n_i}\},\quad \underline{\mu}_i:=\inf\{ \mu_i(\bfx_{-i})\;\;\;\; \forall\bfx_{-i}\in\IR^{n-n_i}\},\label{LdefA}\\
			&\bar{\nu}_j:=\sup\{ \nu_j(\bfy_{-j})\;\;\forall\bfy_{-j}\in\IR^{m-m_j}\},\quad \underline{\nu}_j:=\inf\{ \nu_j(\bfy_{-j})\;\;\;\; \forall\bfy_{-j}\in\IR^{m-m_j}\},\label{LdefB}
		\end{align}
	\end{subequations} 
\end{remark}\noindent
By deploying the algorithmic framework of Gauss-Seidel, we present the pseudo codes of eASAP for solving \eqref{geP} in Algorithm \ref{algo1}. Some remarks on eASAP are provided as follows.
\begin{algorithm}[t]
	\renewcommand{\algorithmicrequire}{\textbf{Input:}}
	\renewcommand{\algorithmicensure}{\textbf{Output:}}
	\caption{eASAP for solving \eqref{geP} with multiblock nonseparable structure.}\label{algo1}
	\begin{algorithmic}[1]
		\REQUIRE Choose the stepsizes 
		and the initial points $\bfz^0=(\bfx^0,\bfy^0)\in\IR^{n+m}$. Set $\varepsilon>0$.
		\REPEAT
		\FOR{$i=1,\ldots,s$}   \STATE $\bfx^{k+1}_i\in\underset{\bfx_i\in\IR^{n_i}}{\hbox{argmin}}\Big\{H(\bfx_{<i}^{k+1},\bfx_i,\bfx_{>i}^{k},\bfy^k)+\langle \bfx_i-\bfx_i^k, \nabla_{\bfx_i}F(\bfx_{<i}^{k+1},\bfx_{\geq i}^{k})\rangle +\frac{\tau_i^k}{2}\|\bfx_i-\bfx^k_i\|^2\Big\}\label{res}$.
		\ENDFOR
		\FOR{$j=1,\ldots,t$}   \STATE$\bfy  ^{k+1}_j\in\underset{\bfy_j\in\IR^{m_j}}{\hbox{argmin}}\Big\{H(\bfx^{k+1},\bfy_{<j}^{k+1},\bfy_j,\bfy_{>j}^{k})+\langle \bfy_j-\bfy_j^k,\nabla_{\bfy_j} G(\bfy_{<j}^{k+1},\bfy_{\geq j}^{k})\rangle+\frac{\sigma_j^k}{2}\|\bfy_j-\bfy^k_j\|^2\Big\}\label{rew}$.
		\ENDFOR
		\UNTIL{$\|\bfz^{k+1}-\bfz^{k}\|\le\varepsilon$.} 
	\end{algorithmic}  
\end{algorithm}


\begin{remark}\label{algore}
	\begin{itemize}
		\item[]
		\item[(i)]
		The $\bfx_i$- and  $\bfy_j$-subproblems in Algorithm \ref{algo1} are involved in the proximity of
		$H(\bfx_{<i},\cdot,\bfx_{>i},\bfy)$ and $H(\bfx,\bfy_{<j},\cdot,\bfy_{>j})$, respectively.   Accordingly, by the definition of proximity in \eqref{proximity}, the $\bfx_i$- and $\bfy_j$-subproblems can be reformulated as
		\begin{subequations}
			\begin{align}
				&\bfx^{k+1}_i=\Big[\bm{I}+\frac{1}{\tau_i^k}\partial_{\bfx_i} H(\bfx_{<i}^{k+1},\cdot,\bfx_{>i}^{k},\bfy^k)\Big]^{-1}\left(\bfx_i^k-\frac{1}{\tau_i^k}\nabla_{\bfx_i} F(\bfx_{<i}^{k+1},\bfx_{\ge i}^{k})\right), \\
				&\bfy^{k+1}_j=\Big[\bm{I}+\frac{1}{\sigma_j^k}\partial_{\bfy_j} H(\bfx^{k+1},\bfy_{<j}^{k+1},\cdot,\bfy_{>j}^k)\Big]^{-1}\left(\bfy_j^k-\frac{1}{\sigma_j^k}\nabla_{\bfy_j} G(\bfy_{<j}^{k+1},\bfy_{\ge j}^{k})\right).
			\end{align}
		\end{subequations} 
		In practice, the proximity of $H(\cdot,\bfy)$ and $H(\bfx,\cdot)$ may be nontrival to attainable. However, the blockwise proximity of $H(\bfx_{<i},\cdot,\bfx_{>i},\bfy)$ and $H(\bfx,\bfy_{<j},\cdot,\bfy_{>j})$ may admit closed closed-form solutions or can be solved efficiently by some subroutines. 
		\item[(ii)] As an ad hoc instance, the eASAP (i.e., Algorithm \ref{algo1}) reduced to ASAP \eqref{asap}  with varying stepsizes as $s=t=1$. 
		\item[(iii)] By the optimality conditions of $\bfx_i$- and $\bfy_j$-subproblems in Algorithm \ref{algo1}, we have 
		\begin{subequations}
			\begin{align}
				&-\tau_i^k(\bfx^{k+1}_i-\bfx_i^k)-\nabla_{\bfx_i} F(\bfx_{<i}^{k+1},\bfx_{\ge i}^{k})\in\partial_{\bfx_i} H(\bfx_{\le i}^{k+1},\bfx_{>i}^{k},\bfy^k),\label{optx}\\
				&-\sigma_j^k(\bfy^{k+1}_j-\bfy_j^k)-\nabla_{\bfy_j}G(\bfy_{<j}^{k+1},\bfy_{\ge j}^{k})\in\partial_{\bfy_j} H(\bfx^{k+1},\bfy_{\le j}^{k+1},\bfy_{>j}^{k}).\label{opty}
			\end{align}
		\end{subequations} 
	\end{itemize}
\end{remark}

\section{Convergence analysis}\label{sec:analysis}
We now analyze the convergence of eASAP under some mild conditions. Firstly, the descent property of the sequence generated by Algorithm \ref{algo1} will be discussed in subsection \ref{subsec4.1}. Secondly, by the blockwise restricted prox-regularity of $H$, we shall prove that any limit point of the sequence generated by eASAP is a critical point of \eqref{geP}. 
Furthermore, with the Aubin property of partial subdifferential mappings $\partial_{\bfx_i}J$ and $\partial_{\bfy_j}J$, we shall prove the global convergence of Algorithm \ref{algo1} with the K{\L} property in subsection \ref{subsec4.3}. Finally, in subsection \ref{subsec4.4}, we shall analyze the sublinear convergence rate of Algorithm \ref{algo1} by the error function (see \eqref{pgdmap} for the definition) under some peculiar conditions on stepsizes.

\subsection{Descent property of Algorithm \ref{algo1}}\label{subsec4.1}

To facilitate the upcoming analysis, we present some properties involving the proximal gradient descent method. For the separable optimization problem 
\begin{align}\label{gep}
\min\limits_{\bfx\in\IR^n}\; \varphi(\bfx)+\psi(\bfx),
\end{align}
where $\varphi\in C_{L_\varphi}^{1,1}(\IR^n)$ and $\psi:\IR^n\to\IR_\infty$ is proper closed, the recursion of proximal gradient descent method for solving \eqref{gep} reads 
\begin{align}\label{pgd}
\bar{\bfx}\in\hbox{prox}_{\psi/\tau}(\bfx-\frac{1}{\tau}\nabla\varphi(\bfx)),\quad\tau>0,
\end{align}
where $\bar{\bfx}$ is a new iterate point from previous point $\bfx\in\IR^n$. 
The following    lemma can be referred to, e.g., \cite{bolte2014proximal}. 

\begin{lemma}
\label{despgd}
Let $\varphi\in C_{L_\varphi}^{1,1}(\IR^n)$, $\psi:\IR^n\to\IR_\infty$ be proper closed, and $\bar{\bfx}$ be produced by \eqref{pgd} with some $\tau>0$. Then
\begin{align*}
	\psi(\bar{\bfx})+\varphi(\bar{\bfx})\le \psi(\bfx)+\varphi(\bfx)-\frac{\tau-L_{\varphi}}{2}\|\bar{\bfx}-\bfx\|^2\quad\forall\bfx\in\IR^n.
\end{align*}
\end{lemma}

\begin{lemma}\label{desH}
Suppose that  Assumption \ref{assum1} holds. Then for all $i=1,\ldots,s$ and $j=1,\ldots,t$,  
the sequence $\{\bfz^k:=(\bfx^k,\bfy^k)\}_{k=0}^\infty$ generated by Algorithm \ref{algo1} satisfies
\begin{subequations}
	\begin{align}
		H(\bfx_{\le i}^{k+1},\bfx_{>i}^k,\bfy^k)+F(\bfx_{\le i}^{k+1},\bfx_{>i}^k)\le& H(\bfx_{<i}^{k+1},\bfx_{\ge i}^k,\bfy^{k})+F(\bfx_{<i}^{k+1},\bfx_{\ge i}^k)\nonumber\\&
		-\frac{1}{2}(\tau_i^k-\mu_i^k)\|\bfx_{i}^{k+1}-\bfx_i^k\|^2,\label{desxi}\\
		H(\bfx^{k+1},\bfy_{\le j}^{k+1},\bfy_{>j}^k)+G(\bfy_{\le j}^{k+1},\bfy_{>j}^k)\le& H(\bfx^{k+1},\bfy_{<j}^{k+1},\bfy_{\ge j}^k)+G(\bfy_{<j}^{k+1},\bfy_{\ge j}^k)\nonumber\\
		&-\frac{1}{2}(\sigma_j^k-\nu_j^k)\|\bfy_j^{k+1}-\bfy_j^k\|^2,\label{desyj}
	\end{align}
\end{subequations}
where $\mu_i^k$ and $\nu_j^k$ are the  Lipschitz constants of $\nabla_{\bfx_i}F(\bfx_{<i}^{k+1},\bfx_i,\bfx_{>i}^{k})$ and  $\nabla_{\bfy_j} G(\bfy_{<i}^{k+1},\bfy_j,\bfy_{>j}^{k})$, respectively.
\end{lemma}

\begin{proof}
Essentially, the $\bfx_i$-subproblem  in Algorithm \ref{algo1} is a special case of \eqref{pgd} with $\tau:=\tau_i^k$,  $\bfx:=\bfx_i^k$, $\psi(\cdot):=H(\bfx_{<i}^{k+1},\cdot,\bfx_{>i}^{k},\bfy^k)$ and  $\varphi(\cdot):=F(\bfx_{<i}^{k+1},\cdot,\bfx_{>i}^{k})$. Accordingly,  \eqref{desxi} can be derived tractably from Lemma \ref{despgd}. Analogously, by setting $\tau:=\sigma_j^k$, $\bfx:=\bfy_j^k$, $\psi(\cdot):=H(\bfx^{k+1},\bfy_{<j}^{k+1},\cdot,\bfy_{>j}^{k})$ and $\varphi(\cdot):=G(\bfy_{<j}^{k+1},\cdot,\bfy_{>j}^{k})$ in \eqref{pgd}, we  can derive \eqref{desyj} from  Lemma \ref{despgd}.
\end{proof}

The descent properties of the sequence generated by Algorithm \ref{algo1} can be readily attainable from  Lemma \ref{desH}. 
\begin{lemma}\label{su_des}
Suppose that Assumption \ref{assum1} holds. Let $\{\bfz^k=(\bfx^k,\bfy^k)\}_{k=0}^\infty$ be the sequence produced by Algorithm \ref{algo1} with the stepsizes $(\tau_i^k,\sigma_i^k)$ satisfying 
\begin{subequations}\label{step}
	\begin{align}
		&\tau_i^k=\gamma_i \mu_i^k \;\;\hbox{with} \;\; \gamma_i>1,\quad i=1,\ldots,s,\label{stepA}\\
		&\sigma_j^k=\gamma'_j \nu_j^k \;\;\hbox{with} \;\;\gamma'_j>1,\quad j=1,\ldots,t,\label{stepB}
	\end{align}
\end{subequations}
where $\mu_i^k$ and $\nu_j^k$ are as in Lemma \ref{desH}. 	Then, the following statements hold.
\begin{itemize}
	\item[(i)] The sequence $\{J(\bfz^k)\}_{k=0}^\infty$ is nonincreasing and there exists $c>0$ such that
	\begin{align}\label{monotoneJfun}
		J(\bfz^{k+1})\le J(\bfz^k)-c\|\bfz^{k+1}-\bfz^k\|^2\;\quad \forall k\ge0.
	\end{align} 
	\item[(ii)] $\sum_{k=0}^\infty\|\bfz^{k+1}-\bfz^k\|^2<+\infty$, hence  $\lim_{k\to\infty}\|\bfz^{k+1}-\bfz^k\|=0$.
\end{itemize}
\end{lemma}

\begin{proof}
By summing \eqref{desxi} over $i=1,\ldots,s$ and recalling the notations in  \eqref{abbrxsubvector}, we have 
\begin{align}\label{eqx}  	H(\bfx^{k+1},\bfy^k)+F(\bfx^{k+1})\le H(\bfx^{k},\bfy^k)+F(\bfx^k)- \frac{1}{2}\sum\limits_{i=1}^s(\tau_i^k-\mu_i^k)\|\bfx_{i}^{k+1}-\bfx_i^k\|^2.  	
\end{align}
Analogously, by summing \eqref{desyj} over $j=1,\ldots,t$, we have
\begin{align}\label{eqy}	H(\bfx^{k+1},\bfy^{k+1})+G(\bfy^{k+1})\le H(\bfx^{k+1},\bfy^k)+G(\bfy^k)- \frac{1}{2}\sum\limits_{j=1}^t (\sigma_j^k-\nu_j^k)\|\bfy_j^{k+1}-\bfy_j^k\|^2.
\end{align}
Furthermore, by summing  \eqref{eqx}-\eqref{eqy} and using the definition of $J$ in \eqref{geP}, we derive
\begin{align}\label{Jinequality}
	J(\bfz^{k+1})&\le J(\bfz^k)-\frac{1}{2}\sum\limits_{i=1}^s(\tau_i^k-\mu_i^k)\|\bfx_{i}^{k+1}-\bfx_i^k\|^2-\frac{1}{2}\sum\limits_{j=1}^t (\sigma_j^k-\nu_j^k)\|\bfy_j^{k+1}-\bfy_j^k\|^2\nonumber\\
	&\le J(\bfz^k)-\frac{1}{2}\sum\limits_{i=1}^s (\gamma_i-1)\underline{\mu}_i\|\bfx_i^{k+1}-\bfx_i^k\|^2-\frac{1}{2}\sum\limits_{j=1}^t(\gamma'_j-1)\underline{\nu}_j\|\bfy_j^{k+1}-\bfy_j^k\|^2,
\end{align}
where the last inequality is due to \eqref{Ldef} and \eqref{step}. By defining 
$c:=\frac{1}{2}\min\{(\gamma_i-1)\underline{\mu}_i,(\gamma'_j-1)\underline{\nu}_j\mid i=1,\ldots,s,\;\; j=1,\ldots,t\}$, we derive \eqref{monotoneJfun} from  \eqref{Jinequality}. Furthermore, it follows from \eqref{step} that $c>0$, which implies the statement  (i) holds. With the lower boundedness of $J$ in Assumption \ref{assum1} (iii), the sequence $\{J(\bfz^k)\}_{k=0}^\infty$ converges to some finite value, denoted by $J^*$.  

By summing \eqref{monotoneJfun} over $k=0,1,\cdots,\infty$, we have 
\begin{align*}
	\sum\limits_{k=1}^{\infty} \|\bfz^{k+1}-\bfz^k\|^2\leq\frac{1}{c}[J(\bfz^0)-J^*]<+\infty,
\end{align*}
which indicates the statement  (ii). \end{proof}


\subsection{Subsequetial convergence}\label{subsec4.2}
Let $\hat{\bfz}:=(\hat{\bfx},\hat{\bfy})\in\IR^{n+m}$ be a critical point of \eqref{geP}, i.e., ${\bf 0}\in\partial J(\hat{\bfz})$. The critical point set of \eqref{geP} is denoted by \hbox{crit}($J$).  It follows from the objective of \eqref{geP} that    
\begin{align}\label{optC}
\begin{cases}
	\partial_{\bfx_i}J(\bfx,\bfy)=\nabla_{\bfx_i} F(\bfx)+\partial_{\bfx_i} H(\bfx,\bfy),\quad i= 1,\ldots,s,\\
	\partial_{\bfy_j}J(\bfx,\bfy)=\nabla_{\bfy_j} G(\bfy)+\partial_{\bfy_j} H(\bfx,\bfy),\quad j=1,\ldots,t.
\end{cases}
\end{align}	
We shall verify that all the cluster points of the sequence generated by Algorithm \ref{algo1} fall into $\mathrm{crit}(J)$. Let us start with a preliminary lemma on the boundedness of partial subdifferentials.


\begin{lemma}{\label{subg1}}
Suppose that Assumption \ref{assum1} holds. Let $\{(\bfx^k,\bfy^k)\}_{k=0}^\infty$ be the sequence generated by Algorithm \ref{algo1} and denote  
\begin{subequations}\label{subdiff}
	\begin{align}
		&\bfxi_{\bfx_i}^{k+1}:=\nabla_{\bfx_i} F(\bfx_{\leq i}^{k+1},\bfx_{>i}^{k})-\nabla_{\bfx_i} F(\bfx_{<i}^{k+1},\bfx_{\ge i}^{k})-\tau_i^k(\bfx^{k+1}_i-\bfx_i^k), \;\; i=1,\ldots,s, \label{Ax}\\
		&\bfxi_{\bfy_j}^{k+1}:=\nabla_{\bfy_j} G(\bfy_{\leq j}^{k+1},\bfy_{>j}^{k})-\nabla_{\bfy_j} G(\bfy_{<j}^{k+1},\bfy_{\leq j}^{k})-\sigma_j^k(\bfy^{k+1}_j-\bfy_j^k),\;\; j=1,\ldots,t.\label{Ay}
	\end{align}
\end{subequations}
Then $\bfxi_{\bfx_i}^{k+1}\in\partial_{\bfx_i}J(\bfx^{k+1}_{\le i},\bfx^{k}_{>i},\bfy^k)$ and $\bfxi_{\bfy_j}^{k+1}\in\partial_{\bfy_j}J(\bfx^{k+1},\bfy^{k+1}_{\le j},\bfy^{k}_{>j})$. Particularly, for the constants $\bar{\mu}_i$,  $\bar{\nu}_j$ in \eqref{Ldef}, and $\gamma_i$, $\gamma'_j$ in \eqref{step}, the following inequalities hold.
\begin{subequations}\label{xiinequality}
	\begin{align}
		&\|\bfxi_{\bfx_i}^{k+1}\|\le (\gamma_i+1)\bar{\mu}_i \|\bfx^{k+1}_i-\bfx_i^{k}\|, \label{xiinequalityA}\\
		&\|\bfxi_{\bfy_j}^{k+1}\|\le(\gamma'_j+1)\bar{\nu}_j \|\bfy^{k+1}_j-\bfy_j^k\|.\label{xiinequalityB}
	\end{align}
\end{subequations}
\end{lemma}
\begin{proof}	
We limit our discussion to  \eqref{xiinequalityA}. The \eqref{xiinequalityB} can be derived analogously. By summing $\nabla_{\bfx_i} F(\bfx_{\leq i}^{k+1},\bfx_{>i}^{k})$ to both sides of \eqref{optx}, we obtain
\begin{align*}
	\nabla_{\bfx_i} F(\bfx_{\leq i}^{k+1},\bfx_{>i}^{k})-\nabla_{\bfx_i} F(\bfx_{<i}^{k+1},\bfx_{\ge i}^{k})-\tau_i^k(\bfx^{k+1}_i-\bfx_i^k)\in \partial_{\bfx_i} H(\bfx_{\le i}^{k+1},\bfx_{>i}^{k},\bfy^k)+\nabla_{\bfx_i} F(\bfx_{\leq i}^{k+1},\bfx_{>i}^{k}).
\end{align*}
By \eqref{optC} and the definition of $\bfxi_{\bfx_i}^{k+1}$ in \eqref{Ax}, the inclusion reduces to $\bfxi_{\bfx_i}^{k+1}\in \partial_{\bfx_i} J(\bfx_{\le i}^{k+1},\bfx_{>i}^{k},\bfy^k)$.

Furthermore, it follows from \eqref{Ax} and the Lipschitz continuity of $\nabla F_i$ in Assumption \ref{assum1} (i) that 

\begin{align*}
	\|\bfxi_{\bfx_i}^{k+1}\|\le (\mu_i^k+\tau_i^k)\|\bfx^{k+1}_i-\bfx_i^k\|\le   (\gamma_i+1)\bar{\mu}_i \|\bfx^{k+1}_i-\bfx_i^k\|,		
\end{align*}
where the last inequality can be deduced by the $\bar{\mu}_i$ in \eqref{Ldef} and the stepsizes $\tau_i^k$ in \eqref{stepA}. 
\end{proof}

We now present some assumptions on $H$ in \eqref{geP} for the  convergence analysis of Algorithm \ref{algo1} (see also \cite{nikolova2019alternating}).      
\begin{assumption}\label{assum2}
\begin{itemize}
	\item[]
	\item[(i)] The $\mathrm{dom}(H)$ is closed and the partial subdifferential of $H$ satisfies
	\begin{align*}		\otimes_{i=1}^s\partial_{\bfx_i} H(\bfx,\bfy)\times \otimes_{j=1}^t\partial_{\bfy_j} H(\bfx,\bfy)\subset\partial H(\bfx,\bfy)\;\;\quad \forall (\bfx,\bfy)\in\textnormal{dom}(H),
	\end{align*}
	where $\otimes$ denotes the Cartesian product of sets.
	\item[(ii)] $H$ is blockwise restricted prox-regular on any bounded set. More precisely, for any $(\bfx,\bfy)$ in a bounded set, the partial functions  $H(\bfx_{<i},\cdot,\bfx_{>i},\bfy)$  ($i=1,\ldots,s$) and $H(\bfx,\bfy_{<j},\cdot,\bfy_{>j})$   ($j=1,\ldots,t$) are  restricted prox-regular. 
\end{itemize}
\end{assumption}  
\begin{remark}\label{rem:J}
Since $F_i\in C_{\mu_i}^{1,1}(\IR^{n_i})$  and $G_j\in C_{\nu_j}^{1,1}(\IR^{m_j})$ (see Assumption \ref{assum1} (i)), Assumption \ref{assum2}\;(i) implies that 
\begin{align}\label{subsetJ}
	\otimes_{i=1}^s\partial_{\bfx_i} J(\bfx,\bfy)\times \otimes_{j=1}^t\partial_{\bfy_j} J(\bfx,\bfy)\subset\partial J(\bfx,\bfy)\quad \forall (\bfx,\bfy)\in\mathrm{dom}(J).
\end{align}
Therefore, $\Z^*:=\{\hat{\bfz}\mid 0\in\partial_{\bfx_i}J(\hat{\bfz}),\; 0\in\partial_{\bfy_j} J(\hat{\bfz}),\; i=1,\ldots,s,\; j=1,\ldots,t\}$ is the subset of $\mathrm{crit}(J)$. 
\end{remark}    

For a sequence $\{\bfz^k\}_{k=0}^\infty$ generated by Algorithm \ref{algo1} with some initial point $\bfz^0\in\IR^{n+m}$, let $\IS(\bfz^0)$ be the set of cluster points of  $\{\bfz^k\}_{k=0}^\infty$, i.e., 
\begin{align}\label{Sz0limpoints}
\IS(\bfz^0):=\big\{\bfz^*\in\IR^{n+m}\mid \exists\; \hbox{a subsequence}\; \{\bfz^{k_l}\}_{l=0}^\infty\subset \{\bfz^k\}_{k=0}^\infty\;\;\hbox{satisfying}\;\;\lim_{l\to\infty}\bfz^{k_l}=\bfz^*\big\}.
\end{align}
Moreover, as stated in \cite{bolte2014proximal}, $\IS(\bfz^0)$ is a nonempty compact set when $\{\bfz^k\}_{k=0}^\infty$ is bounded. By the boundedess of partial subgradients in Lemma \ref{subg1}, we now verify that $\IS(\bfz^0)\subset\mathrm{crit}(J)$.

\begin{lemma}\label{subconvergence}
Suppose that Assumptions \ref{assum1}-\ref{assum2} hold, and the stepsizes $(\tau_i^k,\sigma_i^k)$ satisfy \eqref{step}. If the sequence   $\{\bfz^k\}_{k=0}^\infty$ generated by Algorithm \ref{algo1} is bounded, then the following statements hold.
\begin{itemize}
	\item[(i)] $J$ is  finite and constant on $\IS(\bfz^0)$.
	\item[(ii)] $\IS(\bfz^0)\subset\textnormal{crit}(J)$.
	\item[(iii)] $\lim_{k\to\infty}\textnormal{dist}(\bfz^k,\IS(\bfz^0))=0$.
\end{itemize}
\end{lemma}

\begin{proof}
(i) For any $\bfz^*\in \IS(\bfz^0)$, there exists a subsequence $\{\bfz^{k_l}\}_{l=0}^\infty$ such that $\lim_{l\to\infty} \bfz^{k_l}=\bfz^*$. By the closedness of $J$, we have $   	\lim_{l\to\infty} J(\bfz^{k_l})\ge J(\bfz^*)$.

(ii) We have from Lemma \ref{su_des} (ii) that $\lim_{l\to\infty}\|\bfx^{k_l+1}-\bfx^{k_l}\|=0$ and $\lim_{l\to\infty}\|\bfy^{k_l+1}-\bfy^{k_l}\|=0$.   It follows from \eqref{xiinequality} that 
\begin{align*}
	\lim\limits_{l\to\infty} \bfxi_{_{\bfx_i}}^{k_l}\to 0,\quad i=1,\ldots,s,\quad \textnormal{and}\quad 
	\lim\limits_{l\to\infty} \bfxi_{_{\bfy_j}} ^{k_l}\to 0,\quad j=1,\ldots,t.
\end{align*}
By the blockwise restricted prox-regularity on $H$ (see Assumption \ref{assum2} (ii)) and gradient Lipschitz continuity of  $F_i$, $G_j$ (see Assumption \ref{assum1} (i)), we have from Lemma \ref{Sumlipshitz} that $J$ is blockwise restricted prox-regular on any bounded set. 
Furthermore, 
it follows from Lemma \ref{subdfclosed} that the partial subdifferentials of $J$ are parametrically closed at $\bfz^*$ with respect to the subsequence $\{\bfz^{k_l}\}_{l=0}^\infty$. Accordingly, we obtain 
\begin{align*}
	\bm{0}\in\partial_{\bfx_i}J(\bfz^*)\quad\hbox{and}\quad  \mathbf{0}\in\partial_{\bfy_j}J(\bfz^*).
\end{align*}   
Finally, we have from  \eqref{subsetJ} that  $\mathbf{0}\in\partial J(\bfz^*)$, which implies $\bfz^*\in\hbox{crit}(J)$ and $\IS(\bfz^0)\subset\hbox{crit}(J)$.


(iii) Suppose for contradiction that $\lim_{k\to\infty}\textnormal{dist}(\bfz^k,\IS(\bfz^0))\ne0$. Then there exist the strictly increasing
sequence $\{\bfz^{k_m}\}_{m=0}^\infty$ and a constant $\tilde{c}>0$ such that
\begin{align}\label{contra}
	\|\bfz^{k_m}-\bfz^*\|\ge\textnormal{dist}(\bfz^{k_m},\IS(\bfz^0))>\tilde{c}\quad \forall\bfz^*\in \IS(\bfz^0).
\end{align}
Since $\{\bfz^{k_m}\}_{m=0}^\infty$ is a subsequence of the bounded sequence $\{\bfz^k\}_{k=0}^\infty$, then it has a convergent subsequence $\{\bfz^{k_{m_l}}\}_{l=0}^\infty$ with the limit point  in $\IS(\bfz^0)$. Hence, $\lim_{l\to\infty}\textnormal{dist}(\bfz^{k_{m_l}},\IS(\bfz^0))=0$, which is a contradiction to \eqref{contra}. This completes the proof. 
\end{proof}

\subsection{Convergence of the iterative sequence
}\label{subsec4.3} We now present several stronger convergence results of Algorithm \ref{algo1} under the following extra assumption.

\begin{assumption}\label{assum:Aub}
For all $i=1,\ldots,s$, $j=1,\ldots,t$, and $\hat{\bfz}\in \mathcal{Z}^*$, the set-valued functions $\partial_{\bfx_i} J$ and $\partial_{\bfy_j} J$ admit Aubin property around $(\hat{\bfz},\mathbf{0})$ with moduli $\eta_i>0$ and $\eta'_j>0$, respectively.  
\end{assumption}

\begin{lemma}\label{disP}
Suppose that Assumptions \ref{assum1}, \ref{assum2} and \ref{assum:Aub} hold, and the stepsizes $(\tau_i^k,\sigma_j^k)$ satisfy \eqref{step}. If the sequence  $\{\bfz^k\}_{k=0}^\infty$ generated by Algorithm \ref{algo1} is bounded, then there exists  $k'>0$ such that 
\begin{align*}
	\textnormal{dist}(\mathbf{0},\partial J(\bfz^{k }))\le \varrho\|\bfz^{k}-\bfz^{k-1}\|\;\;\; \forall k>k' \;\;\hbox{and}\;\;\;	\lim\limits_{k\to\infty} \textnormal{dist}(\mathbf{0},\partial J(\bfz^{k}))=0,
\end{align*}  
where 
\begin{align*}
	\varrho=\sqrt{2(s+t)\left(\max\{\eta_i^2,(\gamma_i+1)^2\bar{\mu}_i^2\mid i=1,\ldots,s\}+\max\{(\eta'_j)^2,(\gamma'_j+1)^2\bar{\nu}_j^2\mid j=1,\ldots,t\}\right)}.
\end{align*}
\end{lemma}

\begin{proof}
By the definitions of  $(\bfxi_{\bfx_i}^k,\bfxi_{\bfy_j}^k)$ in \eqref{subdiff}, we denote  	\begin{align}\label{subgradients}	\bar{\bfzeta}_{\bfx_i}^{k}:=\underset{\bfzeta_{\bfx_i}\in \partial_{\bfx_i} J (\bfz^{k})}{\hbox{argmin}}\|\bfzeta_{\bfx_i}-\bfxi_{\bfx_i}^k\|\;\;\hbox{and}\;\;\bar{\bfzeta}_{\bfy_j}:=\underset{\bfzeta_{\bfy_j}^{k}\in \partial_{\bfy_j} J(\bfz^{k})}{\hbox{argmin}} \|\bfzeta_{\bfy_j}-\bfxi_{\bfy_j}^k\|.
\end{align}
We have from Lemma \ref{subg1} that $\bfxi_{\bfx_i}^k\in \partial_{\bfx_i} J(\bfx^k_{\le i},\bfx^{k-1}_{>i},\bfy^{k-1})$ and $\bfxi_{\bfy_j}^k\in \partial_{\bfy_j} J(\bfx^k,\bfy^k_{\le j},\bfy^{k-1}_{>j})$. Since $\{\bfz^k\}_{k=0}^\infty$ is bounded, it follows from Lemma \ref{subconvergence} that, for the critical point $\hat{\bfz}\in\IS(\bfz^0)\subset \textnormal{crit}(J)$, there exists  $k'>0$ such that  $\bfz^k\in\B(\hat{\bfz})$,  $\bfxi_{\bfx_i}^k\in\B^{\delta_1}(\mathbf{0})$, $\bfxi_{\bfy_j}^k\in\B^{\delta_2}(\mathbf{0})$ for all $k>k'$. By the Aubin property of $\partial_{\bfx_i} J$ and $\partial_{\bfy_j} J$ in Assumption \ref{assum:Aub},  we have  
\begin{subequations}\label{distxi}
	\begin{align}
		\textnormal{dist}(\bfxi_{\bfx_i}^k,\partial_{\bfx_i} J(\bfz^k))=\|\bar{\bfzeta}_{\bfx_i}^{k}-\bfxi_{\bfx_i}^k\|&\le \eta_{i}\|(\bfx^k,\bfy^k)-(\bfx^k_{\le i},\bfx^{k-1}_{>i},\bfy^{k-1})\|\nonumber\\
		&= \eta_{i}\|(\bfx_{>i}^k,\bfy^k)-(\bfx^{k-1}_{>i},\bfy^{k-1})\|,\label{distxiA}\\
		\textnormal{dist}(\bfxi_{\bfy_j}^{k},\partial_{\bfy_j} J(\bfz^k))=\|\bar{\bfzeta}_{\bfy_j}^{k}-\bfxi_{\bfy_j}^k\|&\le \eta'_{j}\|(\bfx^k,\bfy^k)-(\bfx^{k},\bfy_{\le j}^{k},\bfy_{>j}^{k-1})\|= \eta'_{j}\|\bfy_{>j}^k-\bfy_{>j}^{k-1}\|.\label{distxiB}
	\end{align}
\end{subequations}
We have from \eqref{xiinequalityA} and  \eqref{distxiA} that
\begin{align}\label{xixi}
	\|\bar{\bfzeta}_{\bfx_i}^{k}\|^2&\le2\|\bfxi_{\bfx_i}^k\|^2+2\|\bar{\bfzeta}_{\bfx_i}^{k}-\bfxi_{\bfx_i}^k\|^2\nonumber\\
	&\le 2(\tau_i^{k-1}+\mu_i^{k-1})^2\|\bfx^{k}_i-\bfx_i^{k-1}\|^2+2\eta_{i}^2\|(\bfx_{>i}^k,\bfy^k)-(\bfx^{k-1}_{>i},\bfy^{k-1})\|^2\nonumber\\
	&\leq2(\gamma_i+1)^2\bar{\mu}_i^2\|\bfx^{k}_i-\bfx_i^{k-1}\|^2+2\eta_{i}^2\|(\bfx_{>i}^k,\bfy^k)-(\bfx^{k-1}_{>i},\bfy^{k-1})\|^2\nonumber\\
	&\le2\bar{\vartheta}^2 \|(\bfx_{\geq i}^k,\bfy^k)-(\bfx^{k-1}_{\geq i},\bfy^{k-1})\|^2,
\end{align}
where $\bar{\vartheta}:=\max\big\{\eta_i, (\gamma_i+1)\bar{\mu}_i\mid i=1,\ldots,s\big\}$. 
Analogously, we have from \eqref{xiinequalityB} and  \eqref{distxiB} that
\begin{align}\label{xiyj}
	\|\bar{\bfzeta}_{\bfy_j}^{k}\|^2&\le2\|\bfxi_{\bfy_j}^k\|^2+2 \|\bar{\bfzeta}_{\bfy_j}^{k}-\bfxi_{\bfy_j}^k\|^2\nonumber\\
	&\le 2(\sigma_j^{k-1}+\nu_j^{k-1})^2\|\bfy^{k}_j-\bfy_j^{k-1}\|^2+2(\eta'_j)^2\|\bfy^{k}_{>j}-\bfy^{k-1}_{>j}\|^2\nonumber\\
	&\le2(\gamma'_j+1)^2\bar{\nu}_j^2\|\bfy^{k}_j-\bfy_j^{k-1}\|^2+2(\eta'_j)^2\|\bfy^{k}_{>j}-\bfy^{k-1}_{>j}\|^2\nonumber\\
	&\le2(\bar{\vartheta}')^2 \|\bfy_{\geq j}^k-\bfy_{\geq j}^{k-1}\|,
\end{align}
where   $\bar{\vartheta}':=\max\big\{\eta_j',  (\gamma'_j+1)\bar{\nu}_j\mid j=1,\ldots,t\big\}$. Therefore, 
\begin{align*}
	\left(\textnormal{dist}(\mathbf{0},\partial J(\bfz^k))\right)^2&\le\sum_{i=1}^s\|\bar{\bfzeta}_{\bfx_i}^{k}\|^2+\sum_{j=1}^t\|\bar{\bfzeta}_{\bfy_j}^{k}\|^2 \le 2(s+t)\left(\bar{\vartheta}^2+(\bar{\vartheta}')^2\right)\|\bfz^{k}-\bfz^{k-1}\|^2=:\varrho^2 \|\bfz^{k}-\bfz^{k-1}\|^2.
\end{align*}
Furthermore, it follows from Lemma \ref{su_des} (ii) that $\lim_{k\to\infty} \textnormal{dist}(\mathbf{0},\partial J(\bfz^k))=0$. 
\end{proof}
\begin{remark}
For Assumption \ref{assum:Aub}, it is typically difficult to check the Aubin property of set-valued functions in Assumption \ref{assum:Aub} because $\hat{\bfz}$ is unknown. However, as discussed in \cite[section 5]{levy2000stability}, the prox-regularity can guarantee the Aubin property of the subdifferential mappings holds. Actually, it follows from Lemma \ref{subconvergence} that any subsequence of the bounded sequence $\{\bfz^k\}_{k=0}^\infty$ converges to a critical point $\hat{\bfz}\in \Z^*$. Hence, given $\hat{\bfz}\in \Z^*$,  there exists $k'\ge 1$ such that  $\bfz^k\in\B(\hat{\bfz})$, $\bfxi_{_{\bfx_i}}^k\in\B^{\delta_1}({\bf 0})$ and  $\bfxi_{_{\bfy_j}}^k\in\B^{\delta_2}({\bf 0})$ for all $k\ge k'$. It implies that $J$ admits the blockwise prox-regularity around  $\hat{\bfz}$. Namely,  Assumption \ref{assum:Aub} holds logically from Assumptions \ref{assum1}-\ref{assum2}.
The interested readers can refer to, e.g., \cite{levy1997characterizing,mordukhovich2005subgradient,jourani2012c1}, for more details.
\end{remark}

\begin{theorem}
Suppose that $J$ is a K{\L}  function, Assumptions \ref{assum1}, \ref{assum2} and \ref{assum:Aub} hold, and the stepsizes satisfy \eqref{step}. If the sequence $\{\bfz^k\}_{k=0}^\infty$  generated by Algorithm \ref{algo1} is bounded, then the following statements  hold.
\begin{itemize}
	\item[(i)] The sequence $\{\bfz^k\}_{k=0}^\infty$ has finite length, i.e., $\sum_{k=1}^\infty \|\bfz^{k+1}-\bfz^k\|<\infty$. 
	\item[(ii)] The sequence $\{\bfz^k\}_{k=0}^\infty$ converges to a critical point  of $J$.
\end{itemize}
\end{theorem}

\begin{proof}
By the boundedness of $\{\bfz^k\}_{k=0}^\infty$, there exists a subsequence $\{\bfz^{k_l}\}_{l=0}^\infty$ such that $\lim_{l\to\infty}\bfz^{k_l}=\bfz^*$. Since $\{ J (\bfz^k)\}_{k=0}^\infty$ is nonincreasing   and $\lim_{k\to\infty} J(\bfz^k)\ge J(\bfz^*)=J^*$,  there exists  integer $k_0>0$ such that 
\begin{align*}
	J(\bfz^*)<J(\bfz^{k_0})<J(\bfz^*)+d\quad\;\forall d>0,
\end{align*}
i.e.,  $\bfz^k\in[J(\bfz^*)<J(\bfz^{k_0})<J(\bfz^*)+d]$ for all $k>k_0$. On the other hand, it follows from Lemma \ref{subconvergence}\;(iii) that  $\lim_{k\to\infty}\textnormal{dist}(\bfz^k,\IS(\bfz^0))=0$. Therefore, for any $\epsilon>0$, there exists  integer $k_1>0$ such that $\textnormal{dist}(\bfz^k,\IS(\bfz^0))<\epsilon$ for all $k>k_1$. Additionally, for the $k'$ in Lemma \ref{disP}, we deduce that $\bfz^k$ belongs to the intersection for all $k>l:=\max\{ k_0,k_1,k'\}+1$.

(i) By Lemma \ref{uniKL}, there exists a concave function $\varphi\in\Phi_{\eta}$ such that
\begin{align*}
	\varphi'(J(\bfz^k)-J(\bfz^*))\textnormal{dist}(0,\partial J(\bfz^k))\ge 1\quad \forall k\ge l.
\end{align*} 
Furthermore, it follows from Lemma \ref{disP} that 
\begin{align}\label{conca1}
	\varphi'(J(\bfz^k)-J(\bfz^*))\ge \left(\textnormal{dist}(0,\partial J(\bfz^k))\right)^{-1}\ge \left(\varrho\|\bfz^k-\bfz^{k-1}\|\right)^{-1}.
\end{align}
By the concavity of $\varphi$, we have 
\begin{align}\label{conca}
	\varphi(J(\bfz^{k})-J(\bfz^*))-\varphi(J(\bfz^{k+1})-J(\bfz^*))&\ge \varphi'(J(\bfz^{k})-J(\bfz^*))(J(\bfz^k)-J(\bfz^{k+1}))\nonumber\\
	&\ge \frac{c\|\bfz^{k+1}-\bfz^k\|^2}{\varrho\|\bfz^k-\bfz^{k-1}\|},
\end{align}
where the last inequality is due to Lemma \ref{su_des} (i) and \eqref{conca1}. By defining  $r_k:=\varphi(J(\bfz^k)-J(\bfz^*))$, we have   $\{r_k\}_{k=0}^\infty$ is nonincreasing since  $\varphi'(t)>0$. Let $\underline{r}:=\inf_{k\geq0}r_k$ and $c'=\varrho/c$. Then, \eqref{conca} can be rewritten as  
\begin{align*}
	\|\bfz^{k+1}-\bfz^k\|^2\le \left(c'(r_k-r_{k+1})\|\bfz^k-\bfz^{k-1}\|\right)^{1/2}.
\end{align*}
By the identity $2\sqrt{ab}\le a+b$ for all $a$, $b\ge 0$, we deduce  
\begin{align}\label{sqrt}
	2\|\bfz^{k+1}-\bfz^k\|\le c'(r_k-r_{k+1})+\|\bfz^k-\bfz^{k-1}\|.
\end{align}
By summing \eqref{sqrt} over $k=l+1,\ldots,K$, it yields 
\begin{align*}  	\sum\limits_{k=l+1}^K\|\bfz^{k+1}-\bfz^k\|&\le c'(r_{l+1}-r_{K+1})+\|\bfz^{K+1}-\bfz^K\|\\
	&\le c'(r_{l+1}-\underline{r})+\|\bfz^{K+1}-\bfz^K\|.
\end{align*}
By letting $K\to\infty$, we have  from Lemma \ref{su_des} that $\sum_{k=0}^\infty\|\bfz^{k+1}-\bfz^k\|<\infty$. This implies that
\begin{align*}
	\|\bfz^{k_m}-\bfz^{k_n}\|\le \sum\limits_{k=k_n}^{k_m-1}\|\bfz^{k+1}-\bfz^k\|<\sum\limits_{k=k_n}^\infty\|\bfz^{k+1}-\bfz^k\|\quad\;\; \forall k_m>k_n\ge l.
\end{align*}
By taking $k_n\to\infty$, we obtain $\|\bfz^{k_m}-\bfz^{k_n}\|\to0$ from Lemma \ref{su_des}\;(ii). It implies that $\{\bfz^{k_n}\}_{n=0}^\infty$ is a Cauchy sequence, and hence a convergent sequence.  It follows from Lemma \ref{subconvergence} that $\{\bfz^k\}_{k=0}^\infty$ converges to a $\bfz^*\in\mathrm{crit}(J)$.
\end{proof}

\begin{theorem}
Supposed that Assumptions \ref{assum1}, \ref{assum2} and \ref{assum:Aub} hold. Let $\{\bfz^k\}_{k=0}^\infty$ be a sequence generated by Algorithm \ref{algo1} with  initial point $\bfz^0$.  Assume that $J$ admits the K{\L} property and   $\varphi(t)=\frac{c}{\theta}t^{\theta}$ with $\theta\in(0,1]$ and $c>0$. Then the following  statements hold.
\begin{enumerate}[(i)]
	\item If $\theta=1$, Algorithm \ref{algo1} terminates in a finite number of iterations.
	\item If $\theta\in[\frac{1}{2},1)$, then there exist $k_2\ge 1$, $w>0$ and $q\in[0,1)$ such that 
	\begin{align*}
		J(\bfz^k)-J^*\le w q^{k-k_2}\quad\forall k>k_2,
	\end{align*}
	where $J^*=J(\bfz^*)$ for all $\bfz^*\in\IS(\bfz^0)$.
	\item If $\theta\in(0,\frac{1}{2})$, then there exists an integer $k_3>0$ and a constant $w>0$ such that
	\begin{align*}
		J(\bfz^k)-J^*\le \left(\frac{w}{(k-k_3)(1-2\theta)}\right)^{\frac{1}{1-2\theta}}\;\;\quad\forall k>k_3.
	\end{align*}
\end{enumerate}	
\end{theorem}
\begin{proof}
The above theorem is an immediately consequence of \cite[Remark 6]{bolte2014proximal}. Herein, we omit the proof for brevity.	
\end{proof}

\subsection{Convergence rate of the iterative sequence}\label{subsec4.4} Now we analyze the sublinear convergence rate of Algorithm \ref{algo1} under the mild conditions on stepsizes. Compared to the $\bfx_i$- and $\bfy_j$-subproblems in Algorithm \ref{algo1}, we define the auxiliary iterate $\bar{\bfz}^k:=(\bar{\bfx}^k,\bar{\bfy}^k)$ by 
\begin{subnumcases}{\label{res2}}
\bar{\bfx}^k_i\in\underset{\bfx_i\in\IR^{n_i}}{\hbox{argmin}}\Big\{H(\bfx_{<i}^k,\bfx_i,\bfx_{>i}^k,\bfy^k)+\langle \bfx_i-\bfx_i^k,\nabla_{\bfx_i}F(\bfx^k)\rangle+\frac{\tilde{\tau}_i^k}{2}\|\bfx_i-\bfx^k_i\|^2\Big\},\label{prox}\\
\bar{\bfy}^k_j\in\underset{\bfy_j\in\IR^{m_j}}{\hbox{argmin}}\Big\{H(\bfx^k,\bfy_{<j}^k,\bfy_j,\bfy_{>j}^k)+\langle \bfy_j-\bfy_j^k,\nabla_{\bfy_j}G(\bfy^k)\rangle+\frac{\tilde{\sigma}_j^k}{2}\|\bfy_j-\bfy^k_j\|^2\Big\},\label{proy}
\end{subnumcases}
where $\tilde{\tau}_i^k>0$ and $\tilde{\sigma}_j^k>0$. 
Note that the auxiliary iterate $\bar{\bfz}^k$ is essentially futile for numerical computation but is merely used to establish the convergence rate.

Given  $\bfx_{-i}=(\bfx_{<i},\bfx_{>i})\in\IR^{n-n_i}$ and $\bfy_{-j}=(\bfy_{<j},\bfy_{>j})\in\IR^{m-m_j}$, by using the proximity of $H(\bfx_{<i},\cdot,\bfx_{>i},\bfy)$ and  $H(\bfx,\bfy_{<j},\cdot,\bfy_{>j})$, we define the error function by
\begin{subequations}
\begin{align*} 
	&\G_{\bfx_i}(\bfz):=\bfx_i-({\bm I}+\frac{1}{\tau_i}\partial_{\bfx_i}H)^{-1}\left(\bfx_i-\frac{1}{\tau_i}\nabla_{\bfx_i}F(\bfx)\right),\quad i=1,\ldots,s,\\
	&\G_{\bfy_j}(\bfz):=\bfy_j-({\bm I}+\frac{1}{\sigma_j}\partial_{\bfy_j}H)^{-1}\left(\bfy_j-\frac{1}{\sigma_j}\nabla_{\bfy_j}G(\bfy)\right),\quad j=1,\ldots,t.
\end{align*}
\end{subequations}

Accordingly, by the optimality conditions of \eqref{res2}, we have  
\begin{align*}
\begin{cases}
	\bar{\bfx}^k_i=\bfx_i^k-\G_{\bfx_i}(\bfz^k),\\
	\bar{\bfy}^k_j=\bfy_j^k-\G_{\bfy_j}(\bfz^k).
\end{cases}
\end{align*}  
Furthermore, we denote 
\begin{align}\label{pgdmap}
\G(\bfz):=\otimes_{i=1}^s\G_{\bfx_i}(\bfz)\times \otimes_{j=1}^t\G_{\bfy_j}(\bfz).
\end{align}
It follows from \eqref{res2} that  $\G(\bfz^k)=\bfz^k-\bar{\bfz}^{k}$.   Let $\hat{\bfz}$ be a critical point of \eqref{geP}. By the optimality condition, we have $\mathbf{0}\in\G(\hat{\bfz})$. Hence, $\G(\bfz^k)$ essentially quantifies the violation of stationary of $\bfz^k$.

\begin{lemma}\label{ratedsc}
Suppose that Assumption \ref{assum1} holds. Then, the sequence $\{\bar{\bfz}^{k}=(\bar{\bfx}^k,\bar{\bfy}^k)\}_{k=0}^\infty$ defined by \eqref{res2} satisfies
\begin{subequations}
	\begin{align}
		&H(\bfx_{<i}^{k},\bar{\bfx}_i^k,\bfx_{>i}^{k},\bfy^k)+F(\bfx_{<i}^{k},\bar{\bfx}_i^k,\bfx_{>i}^{k}) \le  H(\bfx^k,\bfy^k)+F(\bfx^k)-\frac{1}{2}(\tilde{\tau}_i^k-\tilde{\mu}_i^k)\|\bar{\bfx}_i^k-\bfx_i^k\|^2,\label{Hx}\\
		&H(\bfx^{k},\bfy_{<j}^{k},\bar{\bfy}_j^k,\bfy_{>j}^{k})+G(\bfy_{<j}^{k},\bar{\bfy}_j^{k},\bfy_{>j}^{k}) \le H(\bfx^{k},\bfy^{k})+G(\bfy^k)-\frac{1}{2}(\tilde{\sigma}_j^k-\tilde{\nu}_j^k)\|\bar{\bfy}_j^k-\bfy_j^k\|^2 \label{Hy}
	\end{align}
\end{subequations}
for all $i=1,\ldots, s$ and $j=1,\ldots,t$, where $\tilde{\mu}_i^k$ and $\tilde{\nu}_j^k$ are the Lipschitz constants of $\nabla_{\bfx_i}F(\bfx_{<i}^{k},\cdot,\bfx_{>i}^{k})$ and  $\nabla_{\bfy_j} G(\bfy_{<i}^{k},\cdot,\bfy_{>j}^{k})$, respectively.
\end{lemma}
\begin{proof}
Indeed, \eqref{prox} is a special case of \eqref{pgd} with $\tau:=\tilde{\tau}_i^k$, $\bfx:=\bfx_i^{k}$, $\psi(\cdot):=H(\bfx_{<i}^{k},\cdot,\bfx_{>i}^{k},\bfy^k)$ and  $\varphi(\cdot):=F(\bfx_{<i}^{k},\cdot,\bfx_{>i}^{k})$.     Accordingly, \eqref{Hx} can be deduced by Lemma \ref{despgd}. Analogously, by setting $\tau:=\tilde{\sigma}_j^k$, $\bfx:=\bfy_j^{k}$, $\psi(\cdot):=H(\bfx^{k},\bfy_{<j}^{k},\cdot,\bfy_{>j}^{k})$ and $\varphi(\cdot):=G(\bfy_{<j}^{k},\cdot,\bfy_{>j}^{k})$ in \eqref{pgd}, we  can derive \eqref{Hy} by  Lemma \ref{despgd}.    
\end{proof}

We now define a Lyapunov sequence by 
\begin{align}\label{Laydef}
R^{k}:=J(\bfz^{k})+\frac{1}{2}\sum\limits_{i=1}^s (\tau_i^{k-1}-\mu_i^{k-1})\|\bfx_i^{k}-\bfx_i^{k-1}\|^2+\frac{1}{2}\sum\limits_{j=1}^t(\sigma_j^{k-1}-\nu_j^{k-1})\|\bfy_j^{k}-\bfy_j^{k-1}\|^2.
\end{align}
Besides, 
we denote 
\begin{align}\label{upperB}
\bar{\eta}:=\max\{\eta_i\mid i=1,\ldots,s\},\quad \bar{\eta}':=\max\{\eta'_j\mid j=1,\ldots,t\},
\end{align}
where $\eta_i$ and $\eta_j'$ are defined in Assumption \ref{assum:Aub}, and let 
\begin{align}\label{equ:deilta_rho}
\delta_i^k:=\frac{1}{2}(10\tau_i^k-2\tilde{\mu}_i^k-\alpha_i-20\mu_i^{k-1}),\quad
\rho_j^k:=\frac{1}{2}(10\sigma_j^k-2\tilde{\nu}_j^k-\beta_j-20\nu_j^{k-1}),
\end{align}
where $\mu_i^k$ and $\nu_j^k$ are as in Lemma \ref{desH}, $\tilde{\mu}_i^k$ and $\tilde{\nu}_j^k$ are as in Lemma \ref{ratedsc}, $\alpha_i$ and $\beta_j$ are blockwise semiconvex constants of $H$ (more details can be seen in the following lemma).

%

\begin{lemma}\label{Laydes}
Suppose that Assumptions \ref{assum1}, \ref{assum2} (i) and \ref{assum:Aub} hold. For all $i=1,\ldots,s$ and $j=1,\ldots,t$, assume
$H(\bfx_{<i},\cdot,\bfx_{>i},\bfy)$ and $H(\bfx,\bfy_{<j},\cdot,\bfy_{>j})$ are semiconvex  with  moduli $\alpha_i$ and $\beta_j$, respectively. Let $\{\bfz^k\}_{k=0}^\infty$ be the bounded sequence generated by Algorithm \ref{algo1} and $\{ \bar{\bfz}^k\}_{k=0}^\infty$ be the sequence defined in \eqref{res2}.  
If the stepsizes $(\tau_i^k,\sigma_j^k)$ of Algorithm \ref{algo1} satisfy 
\begin{subequations}\label{Nstep}
	\begin{align}
		&4\mu_i^k>\tau_i^k>\max\big\{\bar{\eta}-\mu_i^k,\frac{2\tilde{\mu}_i^k+\alpha_i}{10}+2\mu_i^{k-1}\big\},\label{set1}\\
		&4\nu_j^k>\sigma_j^k>\max\big\{\bar{\eta}'-\nu_j^k,\frac{2\tilde{\nu}_j^k+\beta_j}{10}+2\nu_j^{k-1} \big\},
	\end{align}
\end{subequations}
	%
Then for $\delta_i^k>0$ and $ \rho_j^k>0$, 
we have 
\begin{align*}
	R^{k+1}\le R^k-\delta^k_i \|\bar{\bfx}^k_i-\bfx_i^k\|^2-\rho^k_j\|\bar{\bfy}^k_j-\bfy_j^k\|^2.
\end{align*} 
\end{lemma}
\begin{proof}
By summing \eqref{Hx}-\eqref{Hy} and using the definition of $J$ in \eqref{geP}, we have
\begin{align}
	&H(\bfx_{<i}^{k},\bar{\bfx}_i^k,\bfx_{>i}^{k},\bfy^k)+F(\bfx_{<i}^{k},\bar{\bfx}_i^k,\bfx_{>i}^{k})+H(\bfx^{k},\bfy_{<j}^{k},\bar{\bfy}_j^k,\bfy_{>j}^{k})+G(\bfy_{<j}^{k},\bar{\bfy}_j^{k},\bfy_{>j}^{k})\nonumber\\
	\le & J(\bfz^k)+H(\bfx^k,\bfy^k)-\frac{1}{2}(\tilde{\tau}_i^k-\tilde{\mu}_i^k)\|\bar{\bfx}_i^k-\bfx_i^k\|^2-\frac{1}{2}(\tilde{\sigma}^k_j-\tilde{\nu}_j^k)\|\bar{\bfy}_j^k-\bfy_j^k\|^2.\label{inequ1}
\end{align}
Since $H$ is blockwise semiconvex with with moduli $\alpha_i (i = 1,\ldots,s)$  and $\beta_j (j = 1,\ldots,t)$, it follows from
Assumption \ref{assum1}\;(i) that $J$ is also blockwise semiconvex. Hence, \begin{subequations}
	\begin{align}        H(\bfx_{<i}^{k},\bar{\bfx}_i^k,\bfx_{>i}^{k},\bfy^k)+F(\bfx_{<i}^{k},\bar{\bfx}_i^k,\bfx_{>i}^{k})\ge& H(\bfx^k,\bfy^k)+F(\bfx^k)
		+\langle \bar{\bfzeta}_{_{\bfx_i}}^k, \bar{\bfx}_i^k-\bfx_i^{k} \rangle\nonumber\\&-\frac{\alpha_i+\tilde{\mu}_i^k}{2}\|\bar{\bfx}_i^k-\bfx_i^k\|^2,\label{sum1}\\
		H(\bfx^{k},\bfy_{<j}^{k},\bar{\bfy}_j^k,\bfy_{>j}^{k})+G(\bfy_{<j}^{k}\bar{\bfy}_j^{k},\bfy_{>j}^{k}) \ge & H(\bfx^k,\bfy^k)+G(\bfy^k)+\langle \bar{\bfzeta}_{_{\bfy_j}}^k, \bar{\bfy}_j^{k}-\bfy_j^{k} \rangle\nonumber\\
		&-\frac{\beta_j+\tilde{\nu}_j^k}{2}\|\bar{\bfy}_j^{k}-\bfy_j^k\|^2,\label{sum2}
	\end{align}
\end{subequations}
where $\bar{\bfzeta}_{_{x_i}}^k\in\partial_{\bfx_i} J(\bfz^k)$ and $\bar{\bfzeta}_{_{y_j}}^k\in\partial_{\bfy_j} J(\bfz^k)$. Summing \eqref{sum1} and \eqref{sum2}, we get
\begin{align}                &H(\bfx_{<i}^{k},\bar{\bfx}_i^k,\bfx_{>i}^{k},\bfy^k)+F(\bfx_{<i}^{k},\bar{\bfx}_i^k,\bfx_{>i}^{k})+H(\bfx^{k},\bfy_{<j}^{k},\bar{\bfy}_j^k,\bfy_{>j}^{k})+G(\bfy_{<j}^{k},\bar{\bfy}_j^{k},\bfy_{>j}^{k})\nonumber\\
	\ge& J(\bfz^{k})+H(\bfx^k,\bfy^k)+\langle \bar{\bfzeta}_{_{\bfx_i}}^k, \bar{\bfx}_i^k-\bfx_i^{k} \rangle+\langle \bar{\bfzeta}_{_{\bfy_j}}^k, \bar{\bfy}_j^{k}-\bfy_j^{k} \rangle-\frac{\alpha_i+\tilde{\mu}_i^k}{2}\|\bar{\bfx}_i^k-\bfx_i^k\|^2 -\frac{\beta_j+\tilde{\nu}_j^k}{2}\|\bar{\bfy}_j^{k}-\bfy_j^k\|^2\nonumber\\
	\ge&  J(\bfz^{k+1})+\frac{1}{2}\sum\limits_{i=1}^s (\tau_i^k-\mu_i^k)\|\bfx_i^{k+1}-\bfx_i^k\|^2+\frac{1}{2}\sum\limits_{j=1}^t(\sigma_j^k-\nu_j^k)\|\bfy_j^{k+1}-\bfy_j^k\|^2+H(\bfx^k,\bfy^k)\nonumber\\
	&+\langle \bar{\bfzeta}_{_{\bfx_i}}^k, \bar{\bfx}_i^k-\bfx_i^{k} \rangle+\langle \bar{\bfzeta}_{_{\bfy_j}}^k, \bar{\bfy}_j^{k}-\bfy_j^{k} \rangle -\frac{\alpha_i+\tilde{\mu}_i^k}{2}\|\bar{\bfx}_i^k-\bfx_i^k\|^2 -\frac{\beta_j+\tilde{\nu}_j^k}{2}\|\bar{\bfy}_j^{k}-\bfy_j^k\|^2,\label{inequ2}
\end{align}
where the last inequality is obtained by  \eqref{Jinequality}. Combining \eqref{inequ1} and \eqref{inequ2} yields
\begin{align}
	& J(\bfz^{k+1})+\frac{1}{2}\sum\limits_{i=1}^s (\tau_i^k-\mu_i^k)\|\bfx_i^{k+1}-\bfx_i^k\|^2+\frac{1}{2}\sum\limits_{j=1}^t(\sigma_j^k-\nu_j^k)\|\bfy_j^{k+1}-\bfy_j^k\|^2\nonumber\\
	\le & J(\bfz^k)-\frac{1}{2}(\tilde{\tau}_i^k-2\tilde{\mu}_i^k-\alpha_i)\|\bar{\bfx}_i^k-\bfx_i^k\|^2-\frac{1}{2}(\tilde{\sigma}_j^k-2\tilde{\nu}_j^k-\beta_j)\|\bar{\bfy}_j^{k}-\bfy_j^{k}\|^2\nonumber\\
	&-\langle \bar{\bfzeta}_{_{\bfx_i}}^k, \bar{\bfx}_i^k-\bfx_i^{k} \rangle-\langle \bar{\bfzeta}_{_{\bfy_j}}^k , \bar{\bfy}_j^{k}-\bfy_j^{k} \rangle.\label{Jine}
\end{align}
Furthermore, by Cauchy-Schwarz inequality, for all $a^{k-1}>0$, $b^{k-1}>0$, we have
\begin{align*}
	-\langle \bar{\bfzeta}_{_{\bfx_i}}^k, \bar{\bfx}_i^k-\bfx_i^{k} \rangle-\langle \bar{\bfzeta}_{_{\bfy_j}}^k, \bar{\bfy}_j^{k}-\bfy_j^{k}\rangle
	\le& \frac{1}{2a^{k-1}}\|\bar{\bfzeta}_{_{\bfx_i}}^k\|^2+\frac{a^{k-1}}{2}\|\bar{\bfx}_i^k-\bfx_i^{k}\|^2\\
	&+\frac{1}{2b^{k-1}}\|\bar{\bfzeta}_{_{\bfy_j}}^k \|^2+\frac{b^{k-1}}{2}\|\bar{\bfy}_j^{k}-\bfy_j^{k}\|^2.
\end{align*}
Since $\{\bfz^k\}$ is bounded and $J$ is blockwise semiconvex, it follows from Lemma \ref{subconvergence} that, for the critical point $\hat{\bfz}\in\IS(\bfz^0)\subset \textnormal{crit}(J)$, there exist $\delta_1>0$ and $\delta_2>0$ such that $\bfz^k\in\B(\hat{\bfz})$,  $\bfxi_{\bfx_i}^k\in\B^{\delta_1}(\mathbf{0})$, $\bfxi_{\bfy_j}^k\in\B^{\delta_2}(\mathbf{0})$ for all $k\ge 0$. 
By \eqref{xixi} and \eqref{xiyj}, we have 
\begin{align*}
	\frac{1}{2a^{k-1}}\|\bar{\bfzeta}_{_{\bfx_i}}^k\|^2&\le \frac{1}{a^{k-1}}\left((\tau_i^{k-1}+\mu_i^{k-1})^2 \|\bfx^{k}_i-\bfx_i^{k-1}\|^2+\sum\limits_{l_1=i+1}^s\eta^2_i\|\bfx^{k}_{l_1}-\bfx^{k-1}_{l_1}\|^2
	+\eta^2_i\|\bfy^k-\bfy^{k-1}\|^2\right),\\
	\frac{1}{2b^{k-1}}\|\bar{\bfzeta}_{_{\bfy_j}}^k\|^2&\le  \frac{1}{b^{k-1}}\left((\sigma_j^{k-1}+\nu_j^{k-1})^2 \|\bfy^{k}_i-\bfy_j^{k-1}\|^2+\sum\limits_{l_2=j+1}^t(\eta'_j)^2\|\bfy^{k}_{l_2}-\bfy^{k-1}_{l_2}\|^2\right).
\end{align*}
Hence, we deduce from \eqref{Jine} that 
\begin{align*}
	&J(\bfz^{k+1})+\frac{1}{2}\sum\limits_{i=1}^s (\tau_i^k-\mu_i^k)\|\bfx_i^{k+1}-\bfx_i^k\|^2+\frac{1}{2}\sum\limits_{j=1}^t(\sigma_j^k-\nu_j^k)\|\bfy_j^{k+1}-\bfy_j^k\|^2\nonumber\\
	\le&  J(\bfz^k)+\frac{1}{a^{k-1}}(\tau_i^{k-1}+\mu_i^{k-1})^2\|\bfx_i^k-\bfx_i^{k-1}\|^2 +\frac{\eta^2_i}{a^{k-1}}\sum\limits_{l_1=i+1}^s\|\bfx^{k}_{l_1}-\bfx^{k-1}_{l_1}\|^2\nonumber\\
	&+ \frac{1}{b^{k-1}}(\sigma_j^{k-1}+\nu_j^{k-1})^2\|\bfy_j^k-\bfy_j^{k-1}\|^2+\frac{(\eta'_j)^2}{b^{k-1}}\sum\limits_{l_2=j+1}^t\|\bfy^{k}_{l_2}-\bfy^{k-1}_{l_2}\|^2+\frac{\eta^2_i}{a^{k-1}}\|\bfy^{k}-\bfy^{k-1}\|^2.\nonumber\\
	& -\frac{1}{2}(\tilde{\tau}_i^k-2\tilde{\mu}_i^k-\alpha_i-a^{k-1})\|\bar{\bfx}_i^k-\bfx_i^k\|^2-\frac{1}{2}(\tilde{\sigma}_j^k-2\tilde{\nu}_j^k-\beta_j-b^{k-1})\|\bar{\bfy}_j^{k}-\bfy_j^{k}\|^2.
\end{align*}
Consequently, it follows from the definition of $R^k$ in \eqref{Laydef} that
\begin{align}\label{inequf}
	&R^{k+1} \nonumber\\
	\le& R^k-\frac{1}{2}\sum\limits_{l_1=1}^{i-1}(\tau_{l_1}^{k-1}-\mu_{l_1}^{k-1})\|\bfx_{l_1}^{k}-\bfx_{l_1}^{k-1}\|^2-\left(\frac{\tau_i^{k-1}-\mu_i^{k-1}}{2}-\frac{1}{a^{k-1}}(\tau_i^{k-1}+\mu_i^{k-1})^2\right)\|\bfx_{i}^k-\bfx_{i}^{k-1}\|^2\nonumber\\
	&-\sum\limits_{l_1=i+1}^{s}\left(\frac{\tau_{l_1}^{k-1}-\mu_{l_1}^{k-1}}{2}-\frac{\eta_i^2}{a^{k-1}}\right)\|\bfx_{l_1}^{k}-\bfx_{l_1}^{k-1}\|^2-\sum\limits_{l_2=1}^{j-1}\left(\frac{\sigma_{l_2}^{k-1}-\nu_{l_2}^{k-1}}{2}-\frac{\eta_i^2}{a^{k-1}}\right)\|\bfy_{l_2}^{k}-\bfy^{k-1}_{l_2}\|^2\nonumber\\
	&-\left(\frac{\sigma_j^{k-1}-\nu_j^{k-1}}{2}-\frac{1}{b^{k-1}}(\sigma_j^{k-1}+\nu_j^{k-1})^2-\frac{\eta_i^2}{a^{k-1}} \right)\|\bfy_{j}^{k}-\bfy_{j}^{k-1}\|\nonumber\\
	&-\sum\limits_{l_2=j+1}^t\left(\frac{\sigma_{l_2}^{k-1}-\nu_{l_2}^{k-1}}{2}-\frac{(\eta'_{l_2})^2}{b^{k-1}}-\frac{\eta^2_i}{a^{k-1}} \right)\|\bfy_{l_2}^k-\bfy_{l_2}^{k-1}\|^2\nonumber\\
	&-\frac{1}{2}(\tilde{\tau}_i^k-2\tilde{\mu}_i^k-\alpha_i-a^{k-1})\|\bar{\bfx}_i^k-\bfx_i^k\|^2-\frac{1}{2}(\tilde{\sigma}_j^k-2\tilde{\nu}_j^k-\beta_j-b^{k-1})\|\bar{\bfy}_j^{k}-\bfy_j^{k}\|^2.
\end{align}

For all $k\ge 0$, choosing $(a^k,b^k):=(20\mu_i^k,20\nu_j^k)$ and $(\tilde{\tau}_i^k,\tilde{\sigma}_j^k):=(10\tau_i^k,10\sigma_j^k)$ in \eqref{res2} such that
\begin{equation*}
	\begin{aligned}
		&\frac{1}{2}(\tau_i^{k}-\mu_i^{k})\ge \frac{1}{a^k}(\tau_i^{k}+\mu_i^{k})^2\ge \frac{\bar{\eta}}{a^k},\\
		&\frac{1}{2}(\sigma_j^{k}-\nu_j^{k})\ge\frac{1}{b^k}(\sigma_j^{k}+\nu_j^{k})^2+\frac{\bar{\eta}^2}{a^k}\ge \frac{(\bar{\eta}')^2}{b^k}+\frac{\bar{\eta}^2}{a^k},\\
		&\tilde{\tau}_i^k-2\tilde{\mu}_i^{k}-\alpha_i-a^{k-1}>0,\quad
		\tilde{\sigma}_j^k-2\tilde{\nu}_j^{k}-\beta_j-b^{k-1}>0,
	\end{aligned}
\end{equation*}
where $\bar{\eta}$ and $\bar{\eta}'$ are denoted in \eqref{upperB}.
%
Hence, we have 
\begin{subequations}\label{ineq1}
	\begin{align}              
		&\frac{a^k}{4}-\mu_i^k\geq\tau_i^k\geq\bar{\eta}-\mu_i^k,\quad \tau_i^k> \frac{1}{10}(a^{k-1}+\alpha_i+2\tilde{\mu}_i^k);\\
		&\frac{b^k}{4}-\nu_j^k\geq\sigma_j^k\geq\bar{\eta}'-\nu_j^k,\quad \sigma_j^k> \frac{1}{10}(b^{k-1}+\beta_j+2\tilde{\nu}_j^k).
	\end{align}
\end{subequations}
Furthermore, from $a^{k-1}=20\mu_i^{k-1}$, $b=20\nu_j^{k-1}$ and \eqref{ineq1}, we obtain \eqref{Nstep}. It follows from
\eqref{equ:deilta_rho} and \eqref{Nstep} that $\delta_i^k>0$ and $\rho_j^k>0$. Finally, by the definition of $R^k$ in \eqref{Laydef} and inequality \eqref{inequf} we obtain the desired result, which completes the proof.
\end{proof}

\begin{remark}
Actually, 
the sets about stepsizes in \eqref{Nstep} are nonempty if the moduli $(\underline{\mu}_i,\bar{\mu}_i)$, $(\underline{\nu}_j,\bar{\nu}_j)$  defined in \eqref{Ldef}, $(\eta_i,\eta'_j)$ defined in Assumption \ref{assum:Aub}, and  $(\alpha_i,\beta_j)$ defined in Lemma \ref{Laydes} satisfy
\begin{equation}\label{consfy}
	\begin{aligned}
		&(5-\sqrt{5})\bar{\mu}_i<\bar{\eta}<5\underline{\mu}_i,\quad\quad \bar{\eta}^2<6\underline{\mu}_i\underline{\nu}_j,\\
		&0<\alpha_i<40\underline{\mu}_i-22\bar{\mu}_i,\quad\quad 	0<\beta_j<40\underline{\nu}_j-22\bar{\nu}_j,\\
		&5\bar{\nu}_j-\underline{\nu}_j-5\underline{\nu}_j\sqrt{\frac{1}{5}-\frac{\eta_i^2}{30\underline{\nu}_j\underline{\mu}_i}}<\bar{\eta}'<5\underline{\nu}_j.
	\end{aligned}
\end{equation}
\end{remark}

\begin{theorem}
Suppose Assumptions \ref{assum1}, \ref{assum2} (i) and \ref{assum:Aub} hold and $H$ is blockwise semiconvex with muludi $\alpha_i$ and $\beta_j$ for all $i=1,\ldots,s$ and $j=1,\cdots,t$. Let $\{\bfz^k\}_{k=0}^\infty$ be the bounded sequence generated by Algorithm \ref{algo1} and $\{ \bar{\bfz}^k\}_{k=0}^\infty$ be the sequence defined in \eqref{res2}. If \eqref{consfy} holds and the stepsizes satisfy \eqref{Nstep}, for any iterative point $\bfz^\Theta=(\bfx^\Theta,\bfy^\Theta)$ from Algorithm \ref{algo1}, there exists $\hat{\delta}>0$ such that 
\begin{align*}
	\textnormal{dist}(0,\G(\bfz^\Theta))^2\le \frac{J(\bfz^0)-\underline{J}}{\lambda K},\;\;\hbox{with}\;\; \lambda:=\frac{\hat{\delta}}{(s+t)^2}.
\end{align*}
\end{theorem}

\begin{proof}
Let $\hat{\delta}^k:=\min\{\delta_i^k,\rho_j^k\mid i=1\ldots,s,\; j=1,\ldots,t\}>0$. From Lemma \ref{Laydes}, we get that
\begin{align*}
	R^{k+1}\le& R^k-\delta^1_i \|\bar{\bfx}^k_i-\bfx_i^k\|^2-\delta^2_j\|\bar{\bfy}^k_j-\bfy_j^k\|^2\nonumber\\
	\le & R^k-\hat{\delta}^k\|\bar{\bfx}^k_i-\bfx_i^k\|^2-\hat{\delta}^k\|\bar{\bfy}^k_j-\bfy_j^k\|^2,
\end{align*}
which indicates that
\begin{align}\label{Rdes22}
	\hat{\delta}^k\|\bar{\bfx}^k_i-\bfx_i^k\|^2+\hat{\delta}^k\|\bar{\bfy}^k_j-\bfy_j^k\|^2\le R^k-R^{k+1}.
\end{align}
Summing \eqref{Rdes22} over $i=1,\ldots,s$ and $j=1,\ldots,t$, we obtain 
\begin{align}\label{Laydes2}
	\frac{\hat{\delta}^k}{s+t}\|\bfz^k-\bar{\bfz}^{k}\|^2\le \sum\limits_{i=1}^s\hat{\delta}^k\|\bar{\bfx}^k_i-\bfx_i^k\|^2+\sum\limits_{j=1}^t \hat{\delta}^k\|\bar{\bfy}^k_j-\bfy_j^k\|^2 \le (s+t)(R^k-R^{k+1}).
\end{align}
Likewise, setting $\hat{\delta}:=\min\{\hat{\delta}^k\mid k=0,\ldots,K-1\}>0$ and summing \eqref{Laydes2} over $k=0,\ldots,K-1$ yield
\begin{align}\label{Laydes4}
	\sum\limits_{k=0}^{K-1} \textnormal{dist}({\bf 0},\G(\bfz^k))^2\le \frac{(s+t)^2}{\hat{\delta}}(R^0-R^K)\le \frac{(s+t)^2}{\hat{\delta}}(R^0-\underline{J}).
\end{align} 
By taking $\bfz^{-1}:=\bfz^{0}$, it follows from \eqref{Laydef} that $R^0=J(\bfz^0)$. For any iterative point $\bfz^\Theta$ from Algorithm \ref{algo1}, the equality \eqref{Laydes4} reduces to the desired result, which completes the proof.	
\end{proof}

\section{Numerical experiments}\label{sec:num}
In this section, we shall conduct some numerical simulations of \eqref{geP} on synthetic and real data to demonstrate the performance of eASAP. All codes for the upcoming numerical experiments are written in MATLAB and implemented on a Lenovo portable computer with Intel Core (TM) CPU 4800 MHZ and 16G memory.

We now present briefly the tensors for the upcoming numerical simulations. The interested reader can  refer to, e.g.,  \cite{Kolda2009,vervliet2019exploiting,qi2021triple,nie2018complete}, for more details about tensor and tensor decomposition.
For an $N$-order tensor $ \IT\in\IR^{I_1\times\cdots\times I_N}$, the CANDECOMP/PARAFAC (CP)  decomposition of $\IT$ is defined by 
\begin{align*}
	\IT=\lr \bfA_1,\ldots,\bfA_N\rr=\sum_{i=1}^r \bfA_1(:,i)\circ \cdots\circ \bfA_N(:,i),  \end{align*}
where $r$ is the CP rank of $\IT$,  $\bfA_n\in\IR^{I_n\times r}$ $(n=1,\ldots,N)$ are factor matrices, and $\circ$ denotes the outer product of vectors. The mode-$n$ unfolding of $\IT$, denoted by ${\bm T}_{(n)}$, is a  $I_n$-by-$J_n$ matrix ($J_n=I_1\times\cdots\times I_{n-1}\times I_{n+1}\times \cdots\times I_{N}$)  satisfying ${\bm T}_{(n)}(j,i_n)=\IT(i_1,\ldots,i_N)$, where $j=1+\sum_{k=1,k\ne n}^N(i_k-1)\bar{J}_k$ and $\bar{J}_k=\prod_{m=1,m\ne n}^{k-1}I_m$. An equivalent reformulation of CP decomposition reads (see, e.g., \cite{Kolda2009}) $ {\bm T}_{(n)}={\bm A}_{n}{\bm H}_{(n)}^\top$, where ${\bm H}_{(n)}=\bfA_N\odot\cdots\odot \bfA_{n+1}\odot\bfA_{n-1}\odot\cdots\odot\bfA_{1}\in\IR^{J_n\times r}$ with  $\odot$ denoting the Khatri-Rao product of matrices. Let $\|\IT\|_F:=(\sum_{{i_1}=1}^{I_1}\cdots\sum_{i_N=1}^{I_N} t^2_{i_1,\ldots,i_N})^{1/2}$ denote the Frobenius norm of tensor $\IT$. Accordingly, the minimization of CP decomposition is 
\begin{align}\label{minF}
	\min\limits_{\{\bfA_n\}_{n=1}^N}\;f(\bfA_1,\ldots,\bfA_N)=\frac{1}{2}\Big\|\IT-\lr \bfA_1,\ldots,\bfA_N\rr\Big\|_F^2.
\end{align}
The objective of \eqref{minF} is continuously differentiable, and 
\begin{align*}
	\nabla_{\bfA_n} f= \bfA_n\bfH_{(n)}^\top \bfH_{(n)}-{\bm T}_{(n)}\bfH_{(n)},\quad \hbox{for}\quad n=1,\ldots,N.
\end{align*} 

\subsection{Synthetic data}

We first consider the Laplacian stochastic coupling model \eqref{optcp} for multimodal data fusion. Specifically,   
let $\IY$ and $\IY'$ denote the required fusion tensors. By assuming that  coupling occurs between $\bfA_3$ and $\bfB_1$. Accordingly, the model can be formulated as 
\begin{align}\label{Lap}
	\min\limits_{\bfA,\bfB}\quad \frac{1}{2}\Big\|\IY-\lr \bfA_1,\bfA_2,\bfA_3\rr \Big\|_F^2+\frac{1}{2}\Big\|\IY'-\lr \bfB_1,\bfB_2,\bfB_3\rr\Big\|_F^2+\mu\|\hbox{vec}(\bfA_3-\bfB_1)\|_1,		
\end{align}
where $\mu$ is a trade-off parameter. It falls into the abstract model \eqref{geP} with 
\begin{align*}
	F(\bfA):=\frac{1}{2}\Big\|\IY-\lr \bfA_1,\bfA_2,\bfA_3\rr\Big\|_F^2,\;  G(\bfB):=\frac{1}{2}\Big\|\IY'-\lr\bfB_1,\bfB_2,\bfB_3\rr \Big\|_F^2, \;  H(\bfA,\bfB):=\mu\|\hbox{vec}(\bfA_3-\bfB_1)\|_1,
\end{align*}
where $\bfA:=[\bfA_1;\bfA_2;\bfA_3]$, and $\bfB:=[\bfB_1;\bfB_2;\bfB_3]$.

We now synthesize the noisy tensors $\IY\in\IR^{30\times 40\times 50}$ and $\IY'\in\IR^{50\times 60\times 70}$. Firstly, we generate the ideal tensors $\IX\in\IR^{30\times 40\times 50}$ and $\IX'\in\IR^{50\times 60\times 70}$. Let $\hbox{rank}_{cp}({\bf\IX})=\hbox{rank}_{cp}({\bf \IX'})=5$. The ideal factor matrices $\bfA_1\in\IR^{30\times 5}$, $\bfA_2\in\IR^{40\times 5}$, $\bfA_3\in\IR^{50\times 5}$, $\bfB_2\in\IR^{60\times 5}$, $\bfB_3\in\IR^{70\times 5}$ can be generated by MATLAB syntax \texttt{rand} with related dimension and $\bfB_1=\bfA_3+{\bm \Gamma}_{i,j}$ with $\gamma_{i,j}\sim\hbox{Laplace}(0,0.1)$. Furthermore, let $\IX=\lr \bfA_1,\bfA_2,\bfA_3\rr$ and $\IX'=\lr \bfB_1,\bfB_2,\bfB_3\rr$. Both of them are ground truth of \eqref{Lap}. Then 
adding noise tensor $\IIN$ with entries drawn from a standard normal distribution as follows,
\begin{align*}
	\IY = \IX+10^{-s/20}\frac{\| \IX\|_F}{\| \IIN\|_F}{\IIN},
\end{align*} 
where $s$ denotes the signal-to-noise ratio (SNR). Let $s=14$ dB, we can obtain tensors $\IY,\;\IY'$ directly.

Throughout this numerical simulation, we take the trade-off parameter $\mu=0.01$ for \eqref{Lap} and random initial points by MATLAB syntax \texttt{rand}/\texttt{randn} (i.e., $\bfA_3^0$, $\bfB_2^0$, $\bfB_3^0$ generated by \texttt{rand} and $\bfA_1^0$, $\bfA_2^0$, $\bfB_1^0$ generated by \texttt{randn}) for all test methods. We compare Algorithm \ref{algo1} with ASAP in  \cite{nikolova2019alternating} (i.e., the recursion \eqref{asap} with two blocks $\bfA$, $\bfB$) and accelerated ASAP with extrapolation in \cite{yang2022some}. The stepsizes are taken as $(\tau_i^k,\sigma_j^k)=(\hbox{trace}({\bm H}_{\bfA_i^k})+k$, $\hbox{trace}({\bm H}_{\bfB_j^k})+k)$ for Algorithm \ref{algo1}, $(\tau^k,\sigma^k)=(\hbox{trace}({\bm H}_{\bfA^k}),\hbox{trace}({\bm H}_{\bfB^k}))$ for ASAP and accelerated-ASAP. Herein,  the matrices for calculating stepsizes are listed  as follows
\begin{gather*}
	\bfH_{\bfA_1^k}:= \bfA_3^{k}\odot \bfA_2^k,\quad \bfH_{\bfA_2^k}=\bfA_{3}^k\odot \bfA_1^{k+1},\quad\bfH_{\bfA_3^k}=\bfA^{k+1}_2\odot \bfA^{k+1}_1,\\
	\bfH_{\bfB_1^k}:= \bfB_3^{k}\odot \bfB_2^k,\quad \bfH_{\bfB_2^k}=\bfB_{3}^k\odot \bfB_1^{k+1},\quad \bfH_{B_3^k}=\bfB^{k+1}_2\odot \bfB^{k+1}_1,\\
	\bfH_{\bfA^k}=\begin{pmatrix}
		\bfA_3^{k}\odot \bfA_2^k\\
		\bfA_3^{k}\odot \bfA_1^k\\
		\bfA_2^{k}\odot \bfA_3^k
	\end{pmatrix},\quad \bfH_{\bfB^k}=\begin{pmatrix}
		\bfB_3^{k}\odot \bfB_2^k\\
		\bfB_3^{k}\odot \bfB_1^k\\
		\bfB_2^{k}\odot \bfB_3^k
	\end{pmatrix}.
\end{gather*}
Moreover, the extrapolation parameter for accelerated-ASAP is given by $\alpha_i^k=\beta_j^k=1-(t_{k-1}-1)/t_k$ for $i=1,2,3$ and $j=1,2,3$, where  $t_k:=(1+\sqrt{\smash[b]{1+4t_{k-1}^2}})/2$ with $t_{-1}=t_0=1$. Particularly, let $f(\bfx):=\|\bfx-{\bm b}\|_{1}=\sum_{l=1}^n|\bfx_l-{\bm b}_l|$ and the closed-form proximity of $f$ involving the $\bfA_3$- or $\bfB_1$- subproblem is
\begin{align*}
	\hbox{prox}_{tf}(\bfx)=\left( \hbox{soft}_{[-t,t]}(\bfx_l) \right)_{1\le l\le n},\quad \hbox{with}\quad  \hbox{soft}_{[-t,t]}(\bfx_l)=\left\{
	\begin{aligned}
		&\bfx_l+t,\quad \bfx_l<{\bm b_l}-t;\\
		&{\bm b}_l,\quad    {\bm b}_l-t\le  \bfx_l\le {\bm b}_l+t;\\
		& \bfx_l-t,\quad  \hbox{otherwise}.
	\end{aligned} 	
	\right.
\end{align*}

\begin{figure}[t]
	\centering
	\setlength\tabcolsep{2pt}
	\begin{tabular}{ccccc}
		\multicolumn{2}{c}{\includegraphics[width=0.45\textwidth]{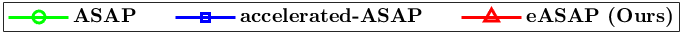}}\\
		\includegraphics[width=0.4\textwidth,height=0.28\textwidth]{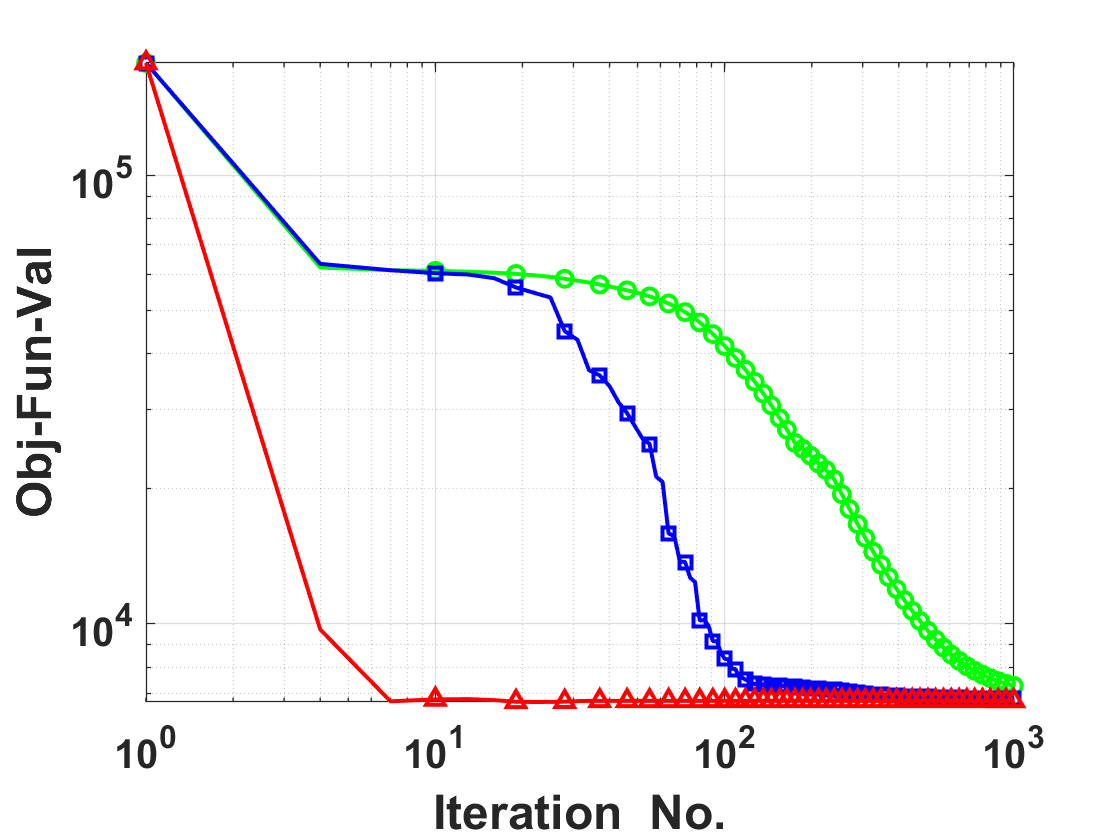}&
		\includegraphics[width=0.4\textwidth,height=0.28\textwidth]{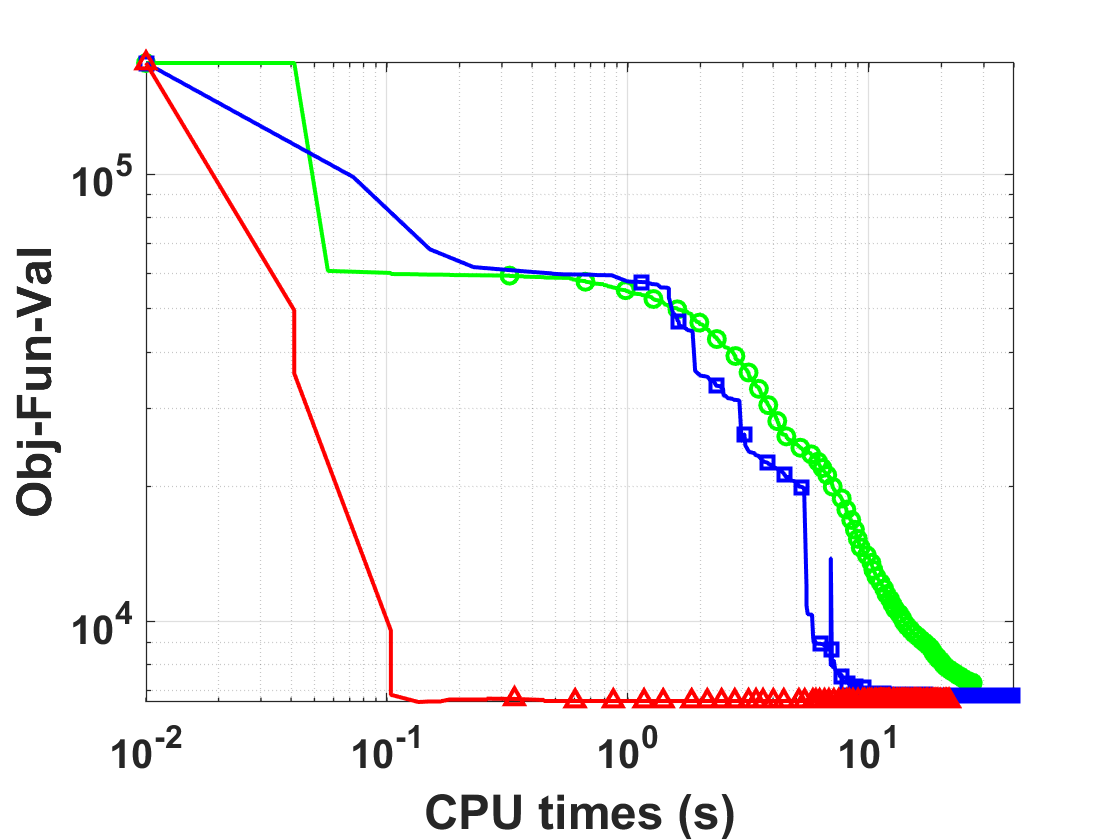}
	\end{tabular}
	\caption{The objective function values with respect to  iterations (left) and CPU time (right) for solving problem \eqref{Lap}. }\label{Figdecay}
\end{figure}

\begin{figure}[t]  
	\centering
	\setlength\tabcolsep{2pt}
	\begin{tabular}{ccccc}
		\multicolumn{2}{c}{\includegraphics[width=0.45\textwidth]{fig/legend_l1}}\\
		\includegraphics[width=0.4\textwidth,height=0.28\textwidth]{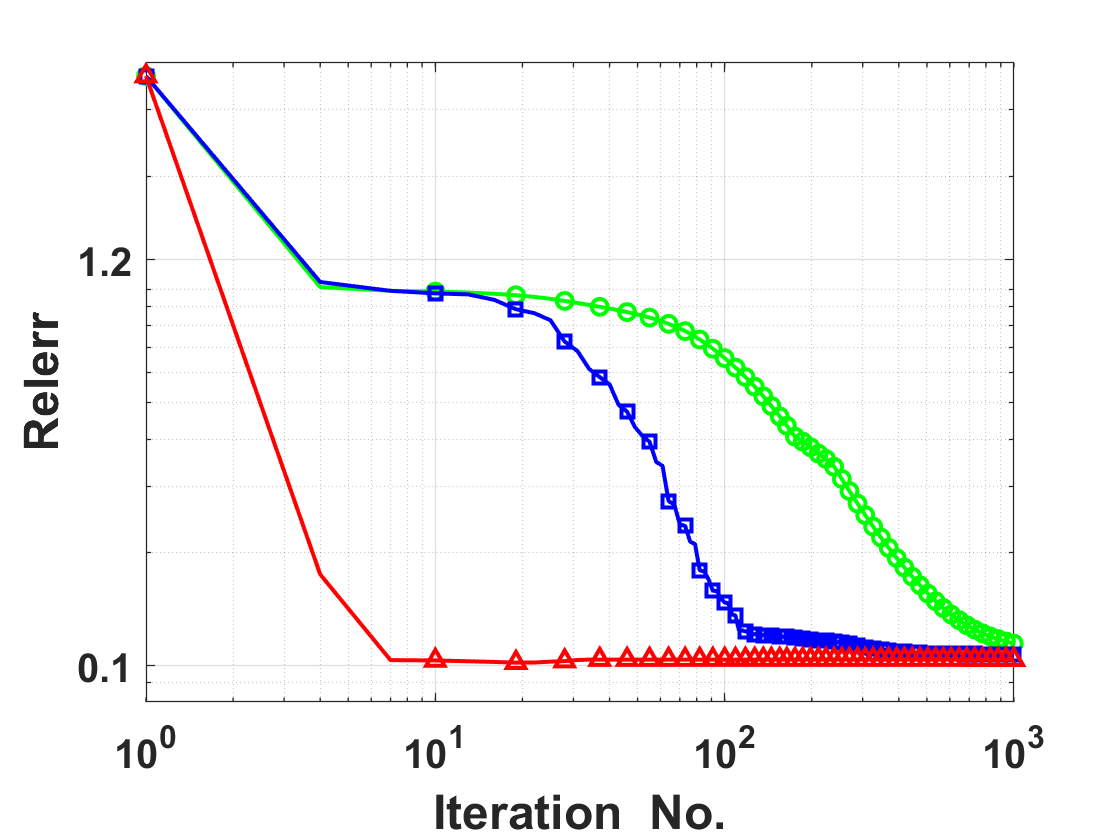}&
		\includegraphics[width=0.4\textwidth,height=0.28\textwidth]{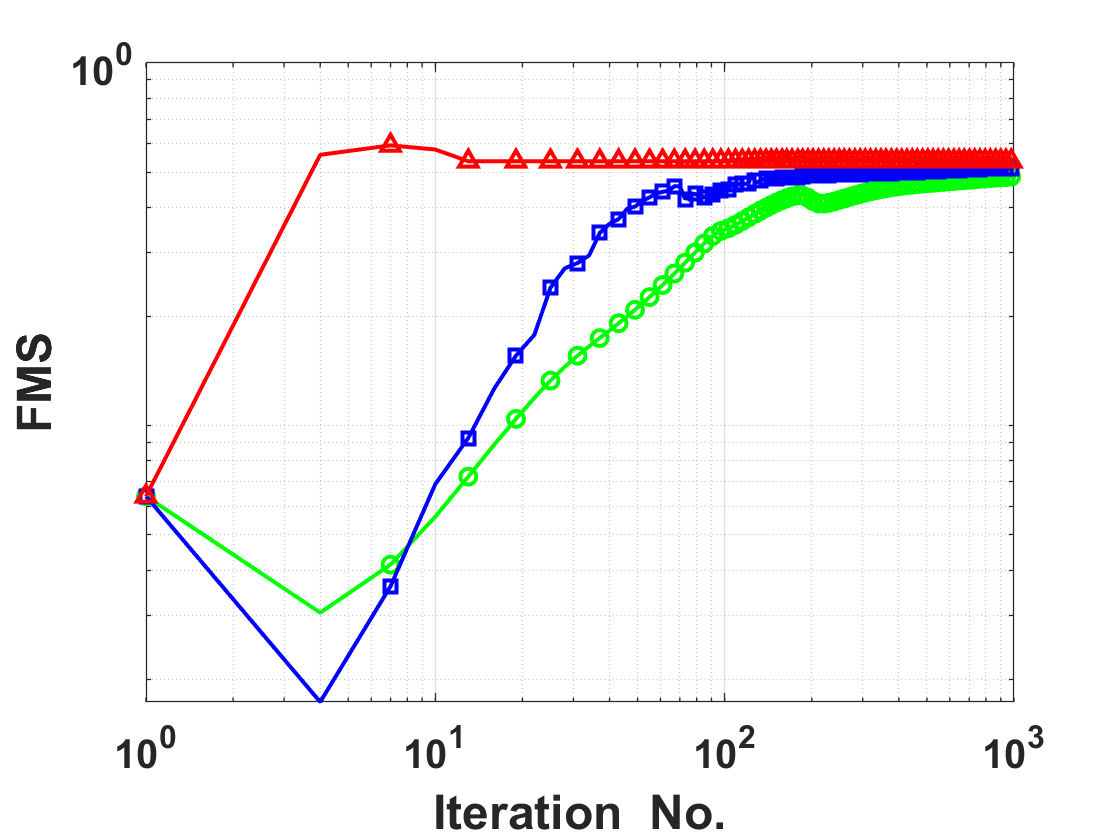}
	\end{tabular}
	\caption{The relative error (Relerr) and factor match score (FMS) with respect to iterations for solving problem \eqref{Lap}.}\label{Relerdecay}
\end{figure}

Figure \ref{Figdecay} displays some evolutions of objective function values in \eqref{Lap} with respect to iterations and computing time in seconds. Those evolutions illustrate that Algorithm \ref{algo1} renders fast decays of objective function values (always with little perceptual modifications after 10 iterations), which demonstrate that Algorithm \ref{algo1} outperforms accelerated-ASAP and ASAP to reach stable evolutions. To quantify the numerical  performances of test methods, we adopt the measurements relative error (``Relerr'') and factor match score (``FMS'') as in  \cite{xu2013block,schenker2020flexible}
\begin{align*}
	&\hbox{Relerr}:=\frac{1}{2}\left(\big\|\IY-\lr \bfA_1,\bfA_2,\bfA_3\rr \big\|_F^2/\|\IY\|_F^2+\big\|\IY'-\lr \bfB_1,\bfB_2,\bfB_3\rr\big\|_F^2/\|\IY'\|_F^2\right),\\
	&\hbox{FMS}:=\frac{1}{2}\left(\frac{1}{5}\sum\limits_{r=1}^5 \prod_{i=1}^3\frac{\langle \bfA_i(:,r),\bfA_i^{\hbox{\tiny true}}(:,r) \rangle }{\|\bfA_i(:,r)\|\|\bfA^{\hbox{\tiny true}}_i(:,r)\|}+\frac{1}{5}\sum\limits_{r=1}^5 \prod_{i=1}^3\frac{\langle \bfB_i(:,r),\bfB_i^{\hbox{\tiny true}}(:,r) \rangle }{\|\bfB_i(:,r)\|\|\bfB^{\hbox{\tiny true}}_i(:,r)\|}\right).
\end{align*}
Higher FMS value and lower Relerr value indicate a more preferred reconstruction performance. Figure \ref{Relerdecay} displays the evolutions of Relerr and FMS with respect to iterations. Visually, the Gauss-Seidel algorithmic framework of eASAP facilitates more accurate numerical solution at the initial stage of iterations and our eASAP also outperforms the other two algorithms in solution quality.

\subsection{Real image data}
For the numerical experiments on real data, we focus on the hyperspectral super-resolution problem. Concretely, it refers to fusing a hyperspectral image (HSI) and multispectral image (MSI) to produce a super-resolution image (SRI) with good spatial and spectral resolutions. This task is illustrated in Figure \ref{hyperspectral_task}.
\begin{figure}[htbp]	
	\centering
	\includegraphics[width=4.3cm]{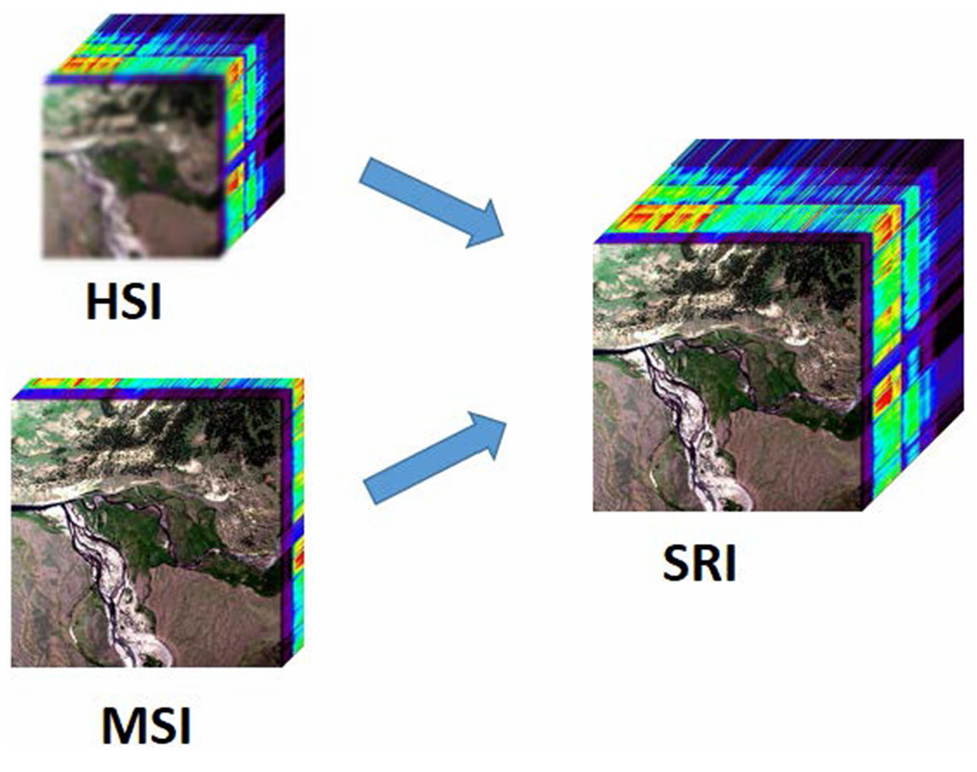}
	\caption{Illustration of the hyperspectral super-resolution task \cite{kanatsoulis2018hyperspectral}.}
	\label{hyperspectral_task}
\end{figure}
Based on coupled tensor CP decomposition, Charilaos et al. \cite{kanatsoulis2018hyperspectral} developed a model for hyperspectral  super-resolution as follows
\begin{align}\label{ste}
	\min\limits_{\bfA,\bfB,\bfC}\; \Big\|\IY_h-\lr \bfP_1 \bfA,\bfP_2\bfB,\bfC\rr\Big\|_F^2+\lambda\Big\|\IY_m-\lr \bfA,\bfB,\bfP_m\bfC\rr\Big\|_F^2,
\end{align}
where $\IY_h\in\IR^{I_h\times J_h\times K}$ and $\IY_m\in\IR^{I\times J\times K_m}$ are the given HSI and  MSI, respectively. $\bfP_1$, $\bfP_2$, $\bfP_m$ are the known degradation operators. The fusion SRI is obtained by $\IY_{s}=\lr \bfA,\bfB,\bfC\rr$. In \cite{kanatsoulis2018hyperspectral}, the alternating minimization method is adopted to solve \eqref{ste}. However, there is no guaranteed convergence for this method and it is heavily dependent on the initial  point. 
Hence, we modify  \eqref{ste} to a joint Gauss coupling model for hyperspectral super-resolution problem. More concretely, the modified model is 
\begin{align}\label{hyper}
	\min\limits_{\bfA,\bfB} \Big\|\IY_h-\lr \bfA_1,\bfA_2,\bfA_3\rr\Big\|_F^2+\lambda\Big\|\IY_m-\lr \bfB_1,\bfB_2,\bfB_3\rr\Big\|_F^2+\mu\big\|[\bfA;\bfB]
	\big\|_{\bm Q}^2
	,
\end{align}
where 
\begin{align*}
	{\bm Q}:=\begin{pmatrix}
		\bfP {\bm\Sigma}_1 \bfP^\top&-\bfP{\bm\Sigma}_1\\
		-{\bm \Sigma}_1\bfP^\top &{\bm\Sigma}_1+{\bm \Sigma}_2
	\end{pmatrix},\quad  \bfP:=\begin{pmatrix}
		\bfP_1 & &\\
		&\bfP_2&\\
		& & \bfP_m
	\end{pmatrix},\quad {\bf\Sigma}_1:=\sigma_1{\bm I},\quad {\bf\Sigma}_2:=\sigma_2{\bm I}.
\end{align*}
As a comprehensive description, we expand the last term in \eqref{hyper} by 
\begin{align*}
	\mu\big\|[\bfA;\bfB]
	\big\|_{\bm Q}^2=&
	\sigma_1\|\bfA_1-\bfP_1\bfB_1\|_F^2+\sigma_1\|\bfA_2-\bfP_2\bfB_2\|_F^2+\sigma_1\|\bfB_3-\bfP_m\bfA_3\|_F^2\\
	& +\sigma_2\|\bfB_1\|_F^2+\sigma_2\|\bfB_2\|_F^2+\sigma_2\|\bfA_3\|_F^2,
\end{align*}
where $\bfA_1\in\IR^{I_h\times R}$, $\bfA_2\in\IR^{J_h\times R}$ and $\bfA_3\in\IR^{K\times R}$; and $\bfB_1\in\IR^{I\times R}$, $\bfB_2\in\IR^{J\times R}$ and $\bfB_3\in\IR^{K_m\times R}$. Obviously, it falls into the abstract model \eqref{geP} with 
\begin{align*}
	F(\bfA):=\Big\|\IY_h-\lr \bfA_1,\bfA_2,\bfA_3\rr\Big\|_F^2,\quad 
	G(\bfB):=\lambda\Big\|\IY_m-\lr \bfB_1,\bfB_2,\bfB_3\rr\Big\|_F^2,\quad
	H(\bfA,\bfB):=\mu\big\|[\bfA;\bfB]
	\big\|_{\bm Q}^2,
\end{align*}
where $\bfA:=\big[\bfA_1;\bfA_2;\bfA_3\big]$,  $\bfB:=\big[\bfB_1,\bfB_2,\bfB_3\big]$.

\begin{figure}[h]
	\centering
	\setlength\tabcolsep{2pt} 
	\begin{tabular}{ccc}
		\includegraphics[height=3.3cm,width=3.3cm]{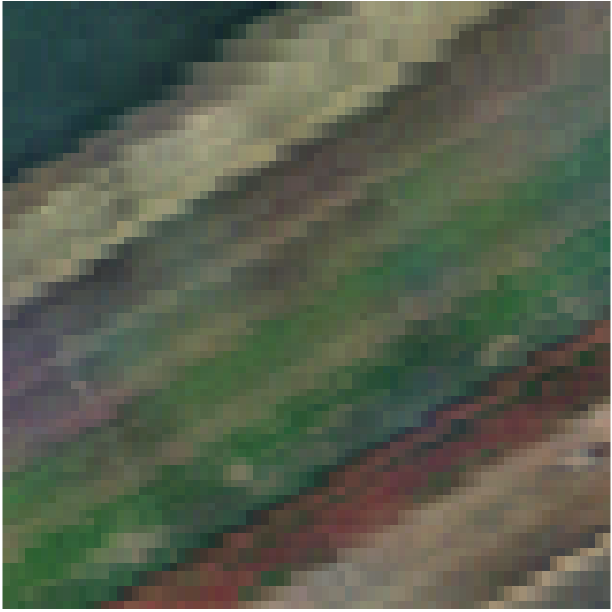}&
		\includegraphics[height=3.3cm,width=3.3cm]{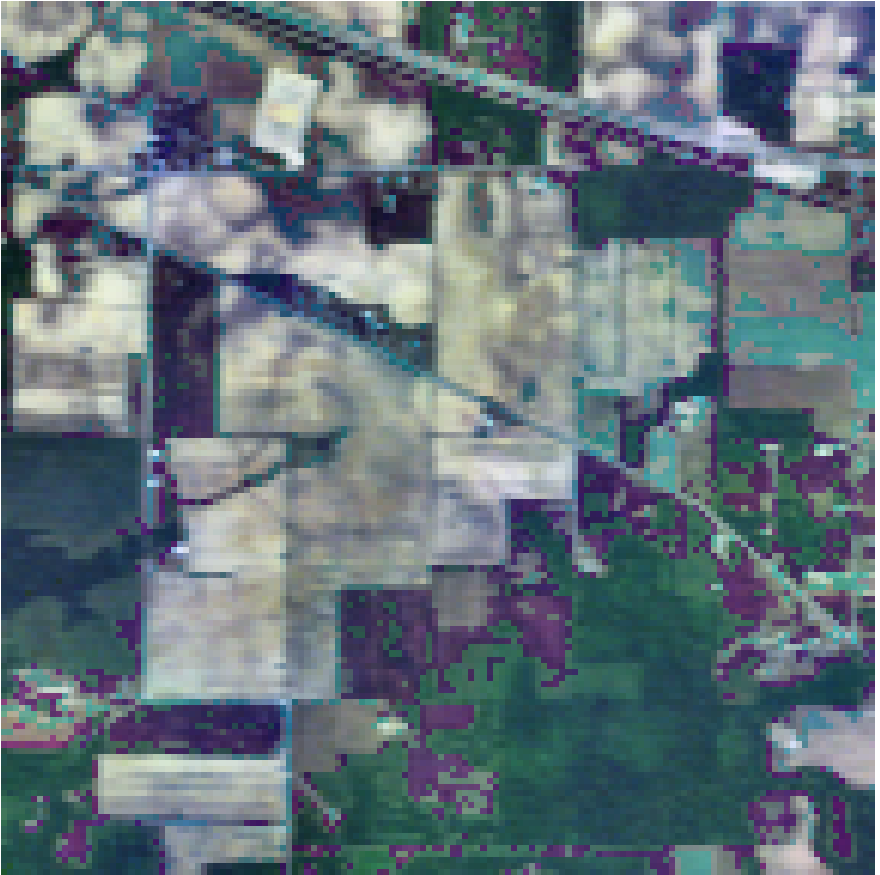}&    		\includegraphics[height=3.3cm,width=3.3cm]{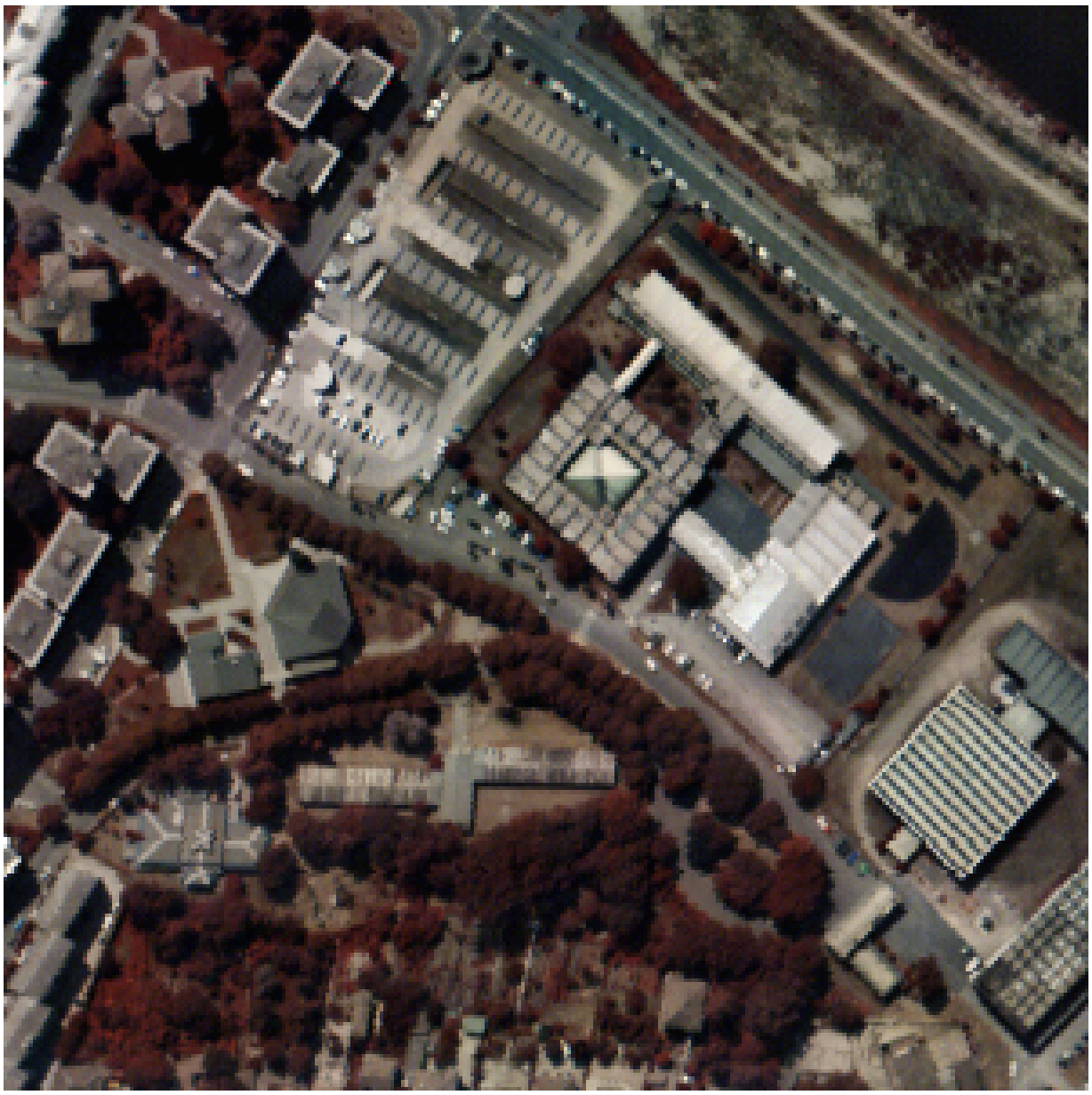}\\
		Salinas & Indian Pines & Pavia Centre
	\end{tabular}
	\caption{Testing hyperspectral images. (a) $80\times 80\times 204$ subscene of Salinas datasets. (b) $144\times 144\times 220$ subscene of Indian Pines dataset. (c) $300\times 300\times 102$ subscene of Pavia Centre dataset.}\label{hyperimg}
\end{figure}

The testing hyperspectral images\footnote{\url{https://www.ehu.eus/ccwintco/index.php/Hyperspectral_Remote_Sensing_Scenes}\label{web} }   are shown in Figure \ref{hyperimg}. We follow the convention in \cite{kanatsoulis2018hyperspectral} that these hyperspectral images act as target SRIs. Hence, the recovery performance can be measured. The degradations from SRI to HSI and MSI are as follows: (i) SRI is first blurred by a $9\times 9$ Gaussian kernel and then downsampled every 4 pixels along each spatial dimension, then degradation matrices $\bfP_1$, $\bfP_2$ and HSI are obtained; (ii) According to spectral degradation sensors LANDSAT\footnote{\url{https://landsat.gsfc.nasa.gov}} and QuickBird\footnote{\url{https://www.satimagingcorp.com/satellite-sensors/quickbird/}}, we can form the spectral degradation matrix $\bfP_m$ and MSI; (iii) Zero-mean white Gauss noise is added to HSI and MSI with SNR being 20 dB and 30 dB, respectively.

To evaluate the quality of recovered SRIs, we adopt several metrics: reconstruction  signal-to-noise (R-SNR), structural similarity (SSIM), cross correlation (CC), root mean square error (RMSE), spectral angle mapper (SAM). The interested reader can refer to, e.g.,  \cite{kanatsoulis2018hyperspectral,wei2015fast,ding2020hyperspectral}, for definitions. Additionally, higher R-SNR, SSIM, CC, and lower RMSE, SAM indicate better reconstruction performance. 

We now compare the numerical performance of Algorithm \ref{algo1} with FUSE  \cite{wei2015fast} and STEREO  \cite{kanatsoulis2018hyperspectral} on solving \eqref{hyper}. Note that  STEREO is used to solve minimization \eqref{ste}, and FUSE contributes to solving the Sylvester equation. The step sizes in Algorithm \ref{algo1} are taken as 
\begin{align*}
	\tau^k_i=\hbox{trace}(\bfH_{\bfA^k_i})\;\;\hbox{for}\;\; i=1,2,3,\quad\hbox{and}\quad
	\sigma^k_j=\hbox{trace}(\bfH_{\bfB^k_j})\;\;\hbox{for}\;\; j=1,2,3,
\end{align*}
where 
\begin{align*}
	&\bfH_{\bfA_1^k}:= \bfA_3^{k}\odot \bfA_2^k,\quad \bfH_{\bfA_2^k}=\bfA_{3}^k\odot \bfA_1^{k+1},\quad \bfH_{\bfA_3^k}=\bfA^{k+1}_2\odot \bfA^{k+1}_1,\\
	&\bfH_{\bfB_1^k}:= \bfB_3^{k}\odot \bfB_2^k,\quad \bfH_{\bfB_2^k}=\bfB_{3}^k\odot \bfB_1^{k+1},\quad \bfH_{\bfB_3^k}=\bfB^{k+1}_2\odot \bfB^{k+1}_1.
\end{align*}
The initial point  $(\bfA^0,\bfB^0)$ is chosen by the initialization technique of STEREO (see e.g., \cite{kanatsoulis2018hyperspectral} for more details). The model parameters $\lambda$, $\sigma_1$, $\sigma_2$ and $\hbox{rank}_{cp}$ for \eqref{hyper};  $\lambda$, and $\hbox{rank}_{cp}$ for \eqref{ste}; the number of endmembers (model rank) $F$ for FUSE; and all data scales are listed in Table \ref{parameters}. 

Figure \ref{SaRs} exhibits the 32-th band of the estimated SRIs, corresponding residual images $\IY_s-\hat{\IY}_s$, and SAM maps on Salines data. Those figures illustrate that  Algorithm \ref{algo1} has small residues across all pixels, while other algorithm's residuals maps are less smooth. Meanwhile, the SAM map of our algorithm is relative closer to the ideal one, which is displayed in the last column. The results on Indian Pines data and Pavia Centra data are displayed in Figures \ref{InRs} and \ref{PaRs}, respectively. More details for numerical comparisons under the five aforementioned metrics can be seen in Table \ref{results}. Therein, numbers in bold indicate the best performance. In conclusion, our proposed Algorithm \ref{algo1} outperforms the baseline FUSE entirely and performs better than STEREO.
\begin{table}[htbp] 
	\setlength\tabcolsep{3.5pt}
	\renewcommand{\arraystretch}{1}
	\begin{center}
		\caption{Data scale and model settings for solving problem \eqref{hyper}.}
		\label{parameters}
		\scalebox{0.9}{
			\begin{tabular}{ccccccccccc}
				\toprule[1pt]
				&\multicolumn{4}{c}{Data Scale} &\multicolumn{3}{c}{Ours} 
				&\multicolumn{2}{c}{STEREO}
				&\multicolumn{1}{c}{FUSE}
				\\ 
				\cmidrule(r){2-5} \cmidrule(r){6-8} \cmidrule(r){9-10}
				& $\IY_h$ & $\IY_m$ & $\bfP_1/\bfP_2$ & $\bfP_m$ & $\lambda$ & $(\sigma_1,\sigma_2)$ & $\hbox{rank}_{cp}$ & $\lambda$ & $\hbox{rank}_{cp}$ & F \\ 
				\toprule
				Salines & $20\times 20\times 204$ & $80\times 80\times 6$ & $20\times 80$ & $6\times 204$ & 0.1 & $(1,1)$ & 30 & 1 &30 & 6\\
				Indian Pines &  $36\times 36\times 220$ & $144\times 144\times 6$ & $36\times144$&$6\times 220$ & 10 & $(10,100)$ & 80 & 1 & 80 & 16\\
				Pavia Centre & $75\times 75\times 102$ & $300\times 300\times 4$ & $75\times 300$ &$4\times 102$ & 10 & $(1,1)$ & 300 &1 & 300 & 9
				\\
				\bottomrule[1pt]
		\end{tabular}}
	\end{center}
\end{table}
\begin{table}[htbp]
	\setlength\tabcolsep{3.5pt}
	\renewcommand{\arraystretch}{1.1}
	\caption{Performance of all test algorithms on three real image data for solving problem \eqref{hyper}.}
	\label{results}
	\begin{center}
		\scalebox{0.97}{
			\begin{tabular}{cccccccccc}
				\toprule[1pt]
				&\multicolumn{3}{c}{Salines} &\multicolumn{3}{c}{Indian Pines} 
				&\multicolumn{3}{c}{Pavia Centre}
				\\ 
				\cmidrule(r){2-4} \cmidrule(r){5-7} \cmidrule(r){8-10}
				Method (Ideal)& FUSE & STEREO & Ours & FUSE & STEREO & Ours & FUSE & STEREO & Ours\\ 
				\toprule
				R-SNR ($\infty$) & 23.658 & 24.536 & {\bf 25.606} &21.777 & 23.564 & {\bf 24.485} &21.455 &21.877 &{\bf 22.450} 
				\\
				SSIM (1) & 0.8937 & 0.9132 & {\bf 0.9359} & 0.2789 &0.3361 & {\bf 0.3777}&0.8304 &0.8329 &{\bf 0.8481}
				\\
				CC (1) & 0.8008 & 0.8075 &{\bf 0.8308} &0.5872 &0.6172 &{\bf 0.6448} &0.9835&0.9850 &{\bf  0.9981}
				\\
				RMSE (0) & 0.0131 & 0.0119 &{\bf 0.0105} &0.0241 &0.0196 &{\bf 0.0136}&0.0131 &0.0124 &{\bf 0.0907}
				\\
				SAM (0) &0.0619 &0.0557 &{\bf 0.0492} & 0.0749 &0.0636 &{\bf 0.0511}&0.1228&0.1197 &{\bf 0.0977}
				\\\bottomrule[1pt]
		\end{tabular}}
	\end{center}
\end{table}

\begin{figure}[htbp]
	\centering
	\setlength\tabcolsep{2pt}
	\begin{tabular}{ccccc}
		\includegraphics[height=2.3cm,width=2.3cm]{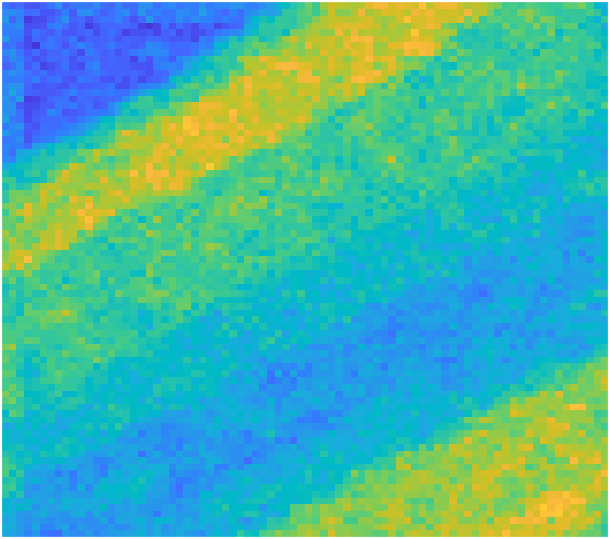}&			\includegraphics[height=2.3cm,width=2.3cm]{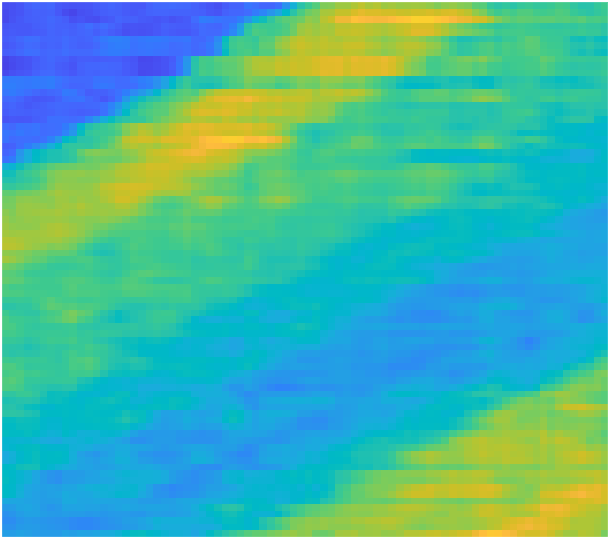}&			\includegraphics[height=2.3cm,width=2.3cm]{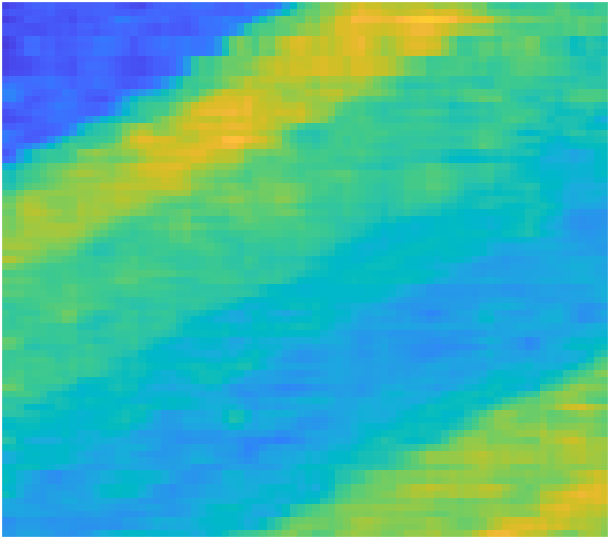}&		\includegraphics[height=2.3cm,width=2.3cm]{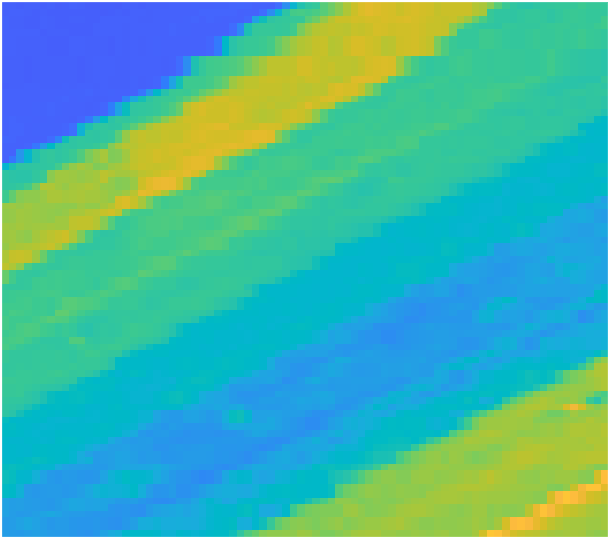}&    		\includegraphics[height=2.3cm,width=0.6cm]{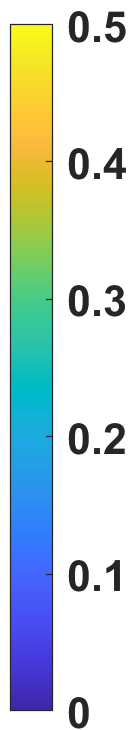}\\
		\includegraphics[height=2.3cm,width=2.3cm]{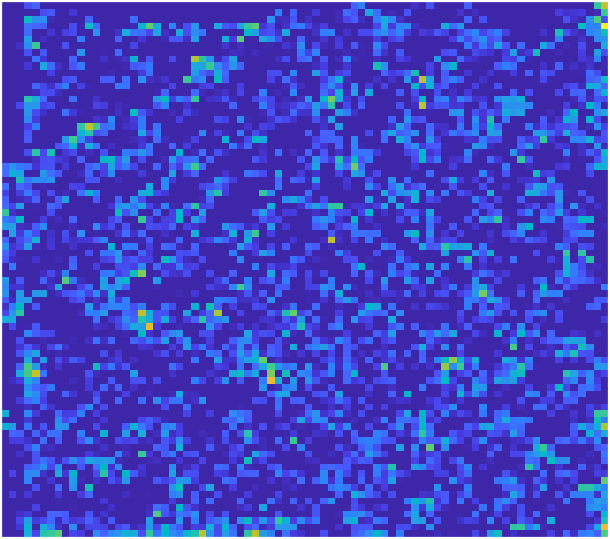}&			\includegraphics[height=2.3cm,width=2.3cm]{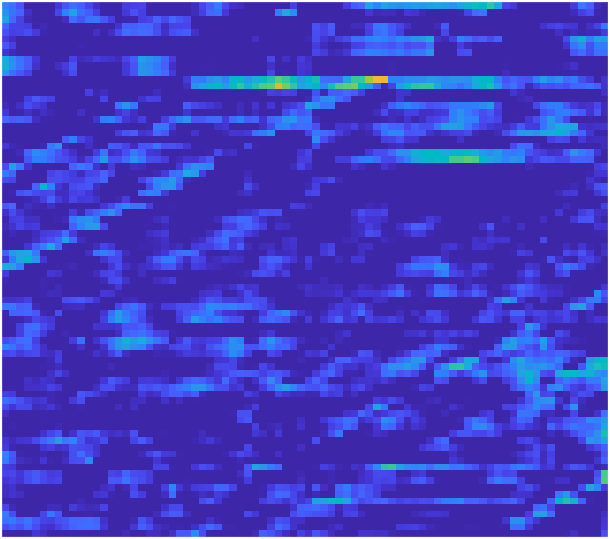}&			\includegraphics[height=2.3cm,width=2.3cm]{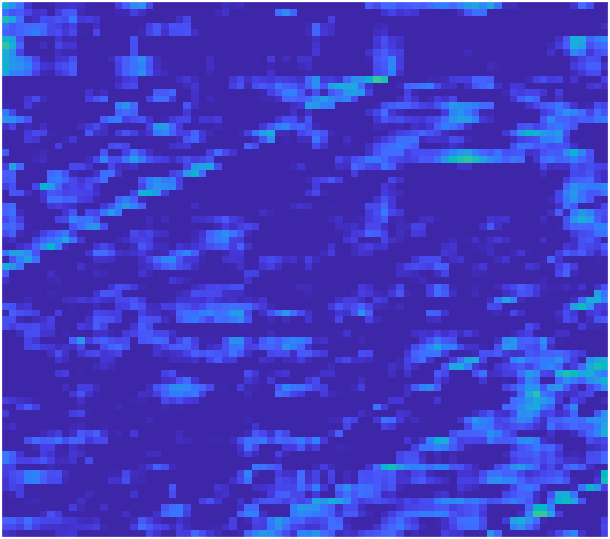}&  		\includegraphics[height=2.3cm,width=2.3cm]{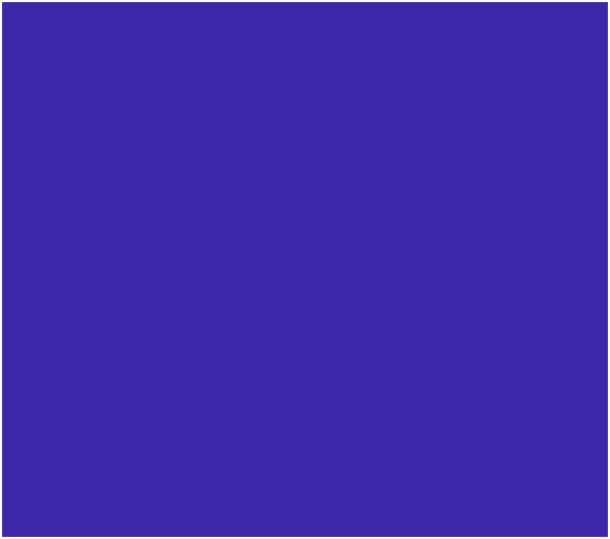}&  		\includegraphics[height=2.3cm,width=0.6cm]{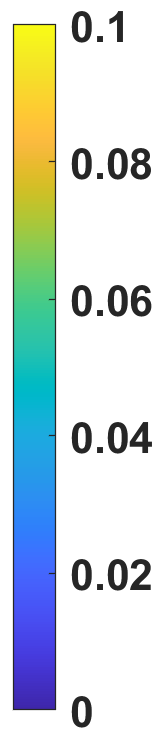}\\
		\includegraphics[height=2.3cm,width=2.3cm]{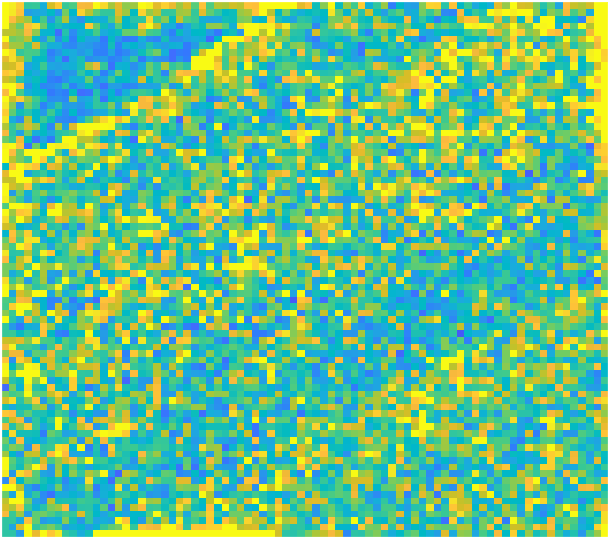}&			\includegraphics[height=2.3cm,width=2.3cm]{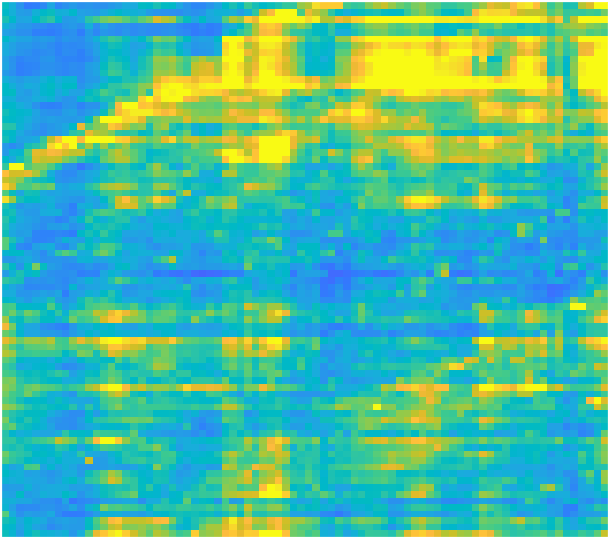}&			\includegraphics[height=2.3cm,width=2.3cm]{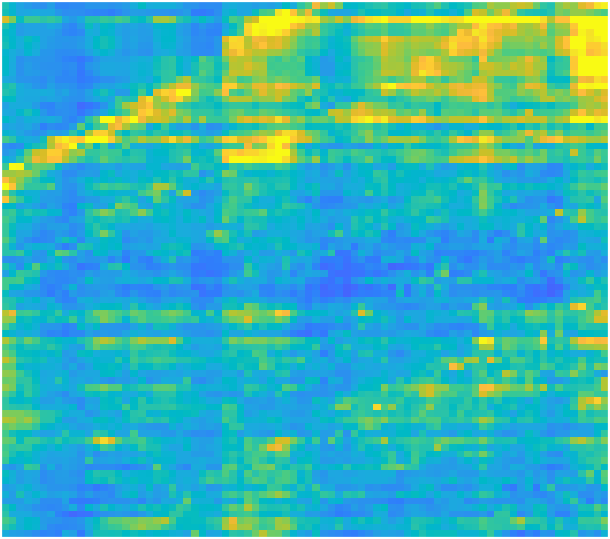}&		\includegraphics[height=2.3cm,width=2.3cm]{fig/Saline_res_SRI}&    		\includegraphics[height=2.3cm,width=0.6cm]{fig/Saline_res_legend}\\
		FUSE & STEREO & Ours & SRI
	\end{tabular}	
	\caption{Results of Salinas reconstructions by  solving \eqref{hyper}. The first row: the 32-th band of recovered SRIs. The second row: the 32-th band of corresponding residual images. The last row: the SAM maps.}	\label{SaRs}
\end{figure}

\begin{figure}[h!]
	\centering
	\setlength\tabcolsep{2pt}
	\begin{tabular}{ccccccc}
		\includegraphics[height=2.3cm,width=2.3cm]{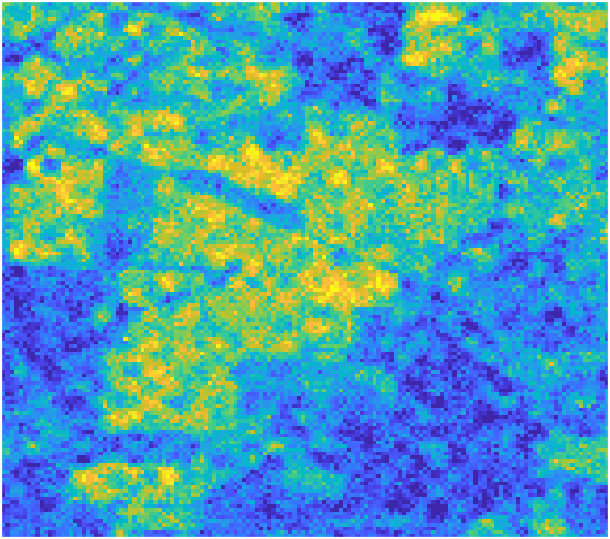}&
		\includegraphics[height=2.3cm,width=2.3cm]{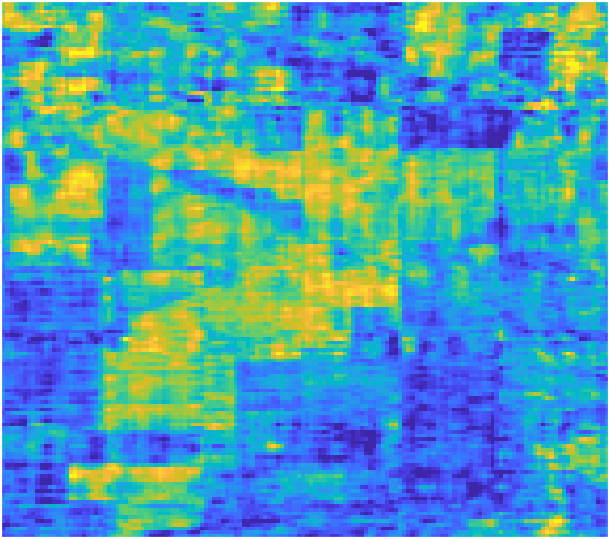}&
		\includegraphics[height=2.3cm,width=2.3cm]{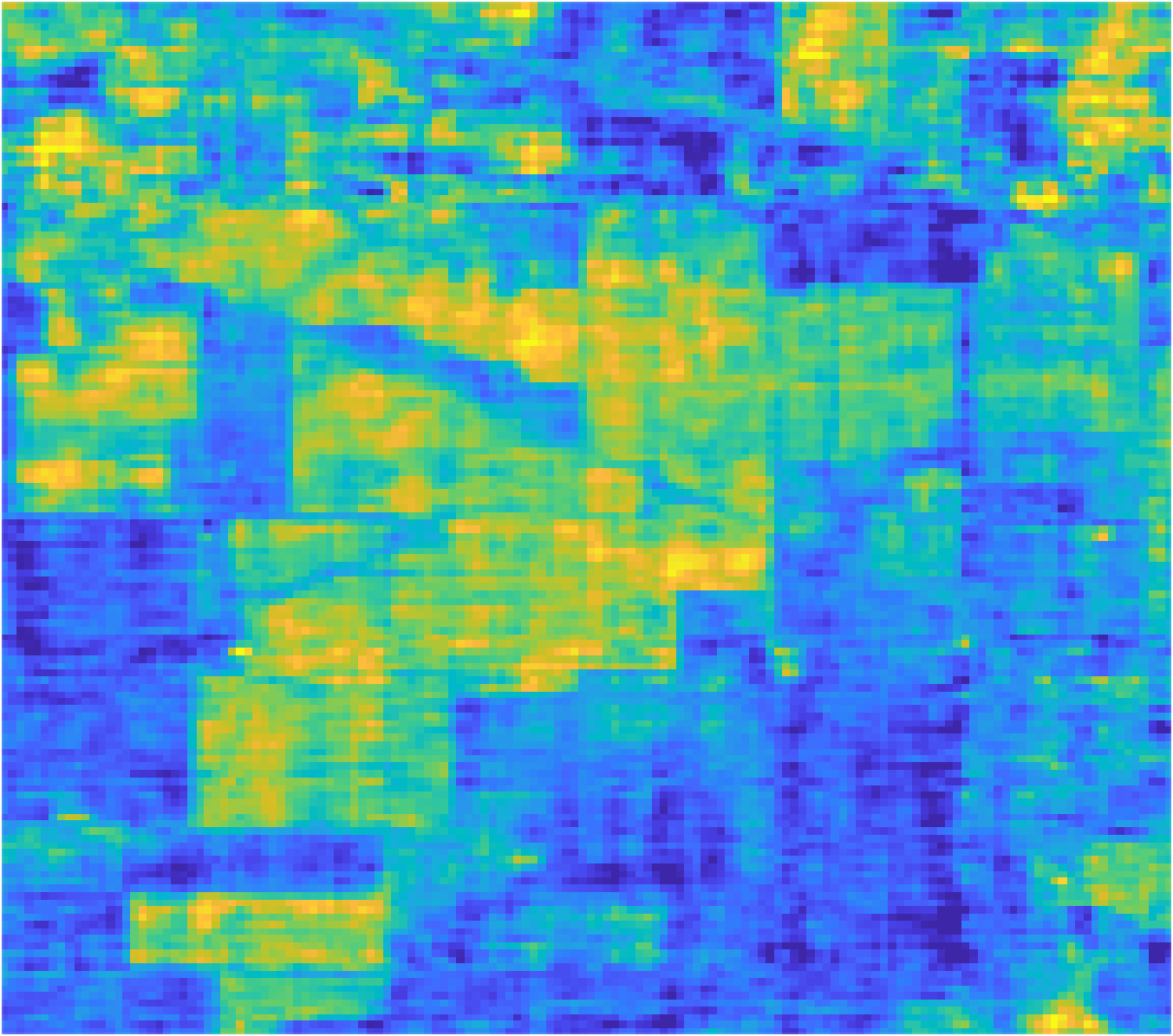}&
		\includegraphics[height=2.3cm,width=2.3cm]{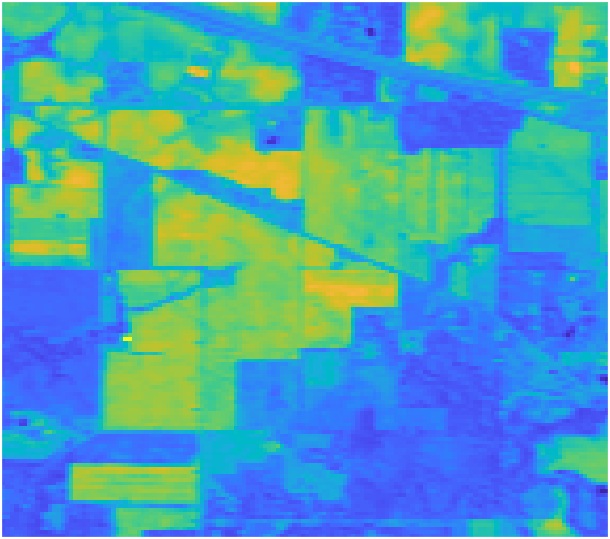}&
		\includegraphics[height=2.3cm,width=0.6cm]{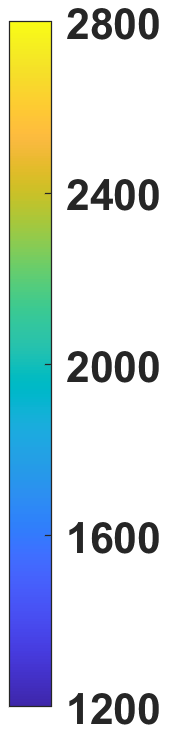}\\	
		
		\includegraphics[height=2.3cm,width=2.3cm]{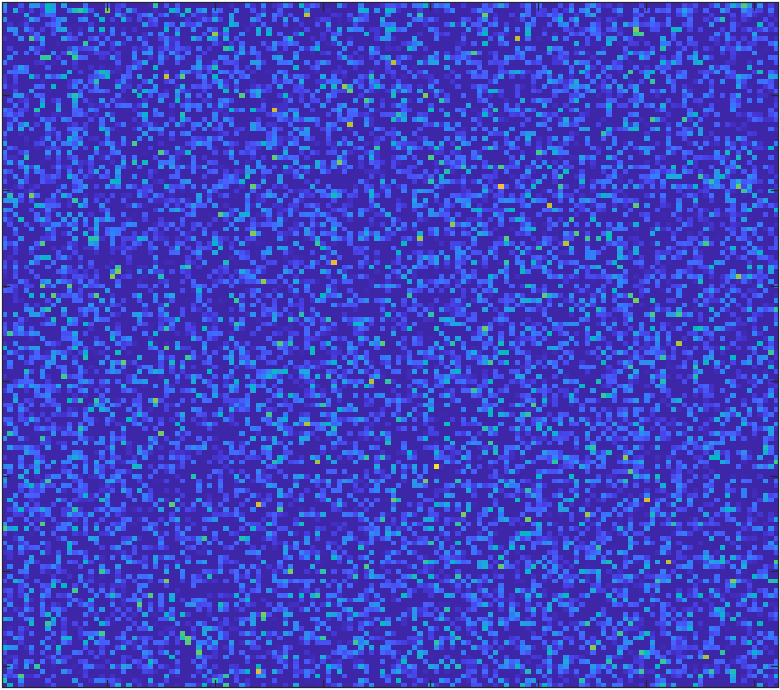}&			\includegraphics[height=2.3cm,width=2.3cm]{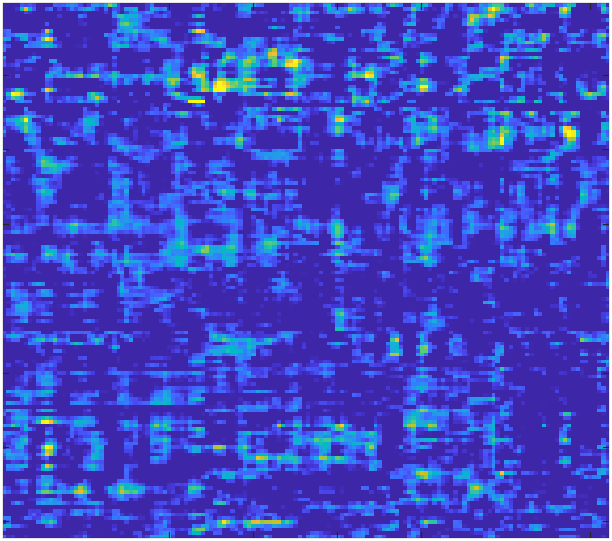}&				\includegraphics[height=2.3cm,width=2.3cm]{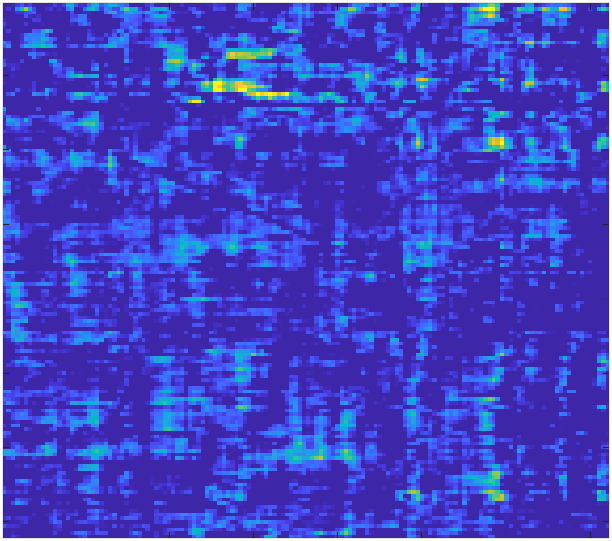}&	\includegraphics[height=2.3cm,width=2.3cm]{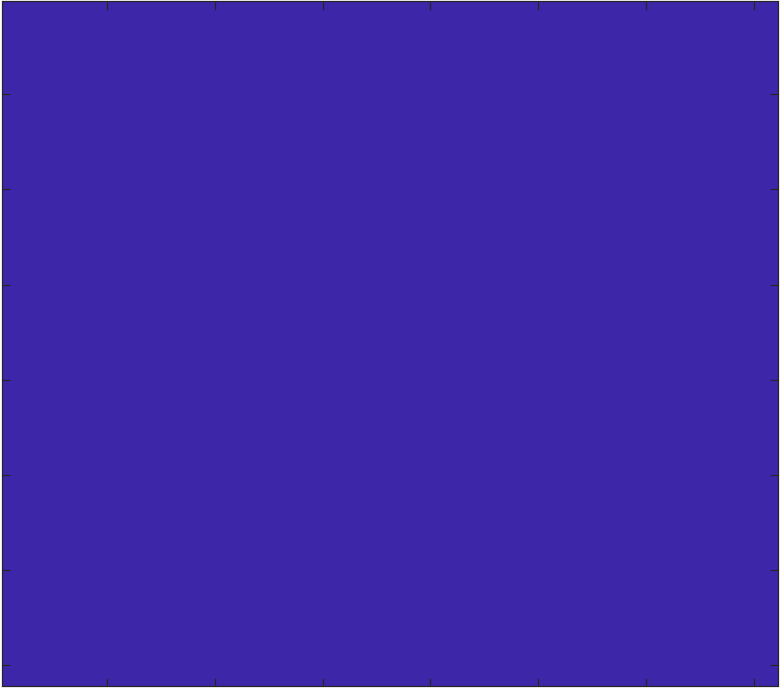}&
		\includegraphics[height=2.3cm,width=0.6cm]{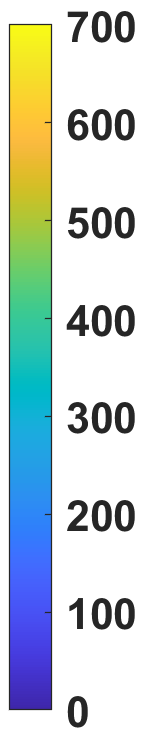}\\
		\includegraphics[height=2.3cm,width=2.3cm]{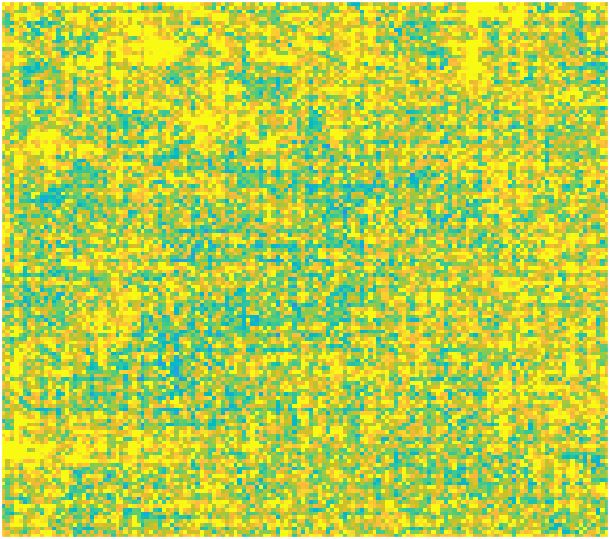}&
		\includegraphics[height=2.3cm,width=2.3cm]{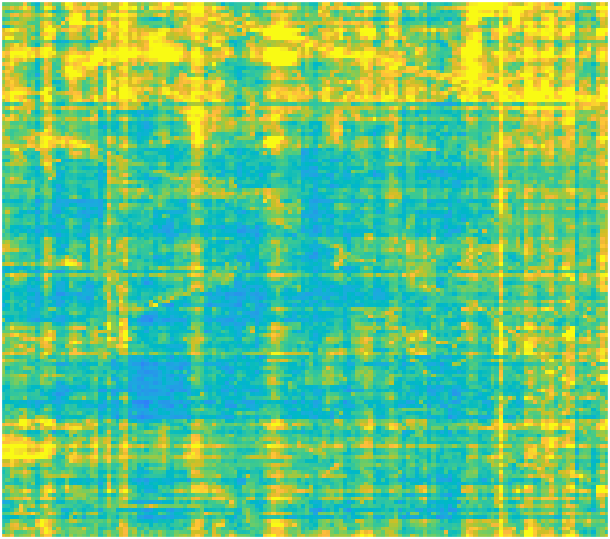}&			\includegraphics[height=2.3cm,width=2.3cm]{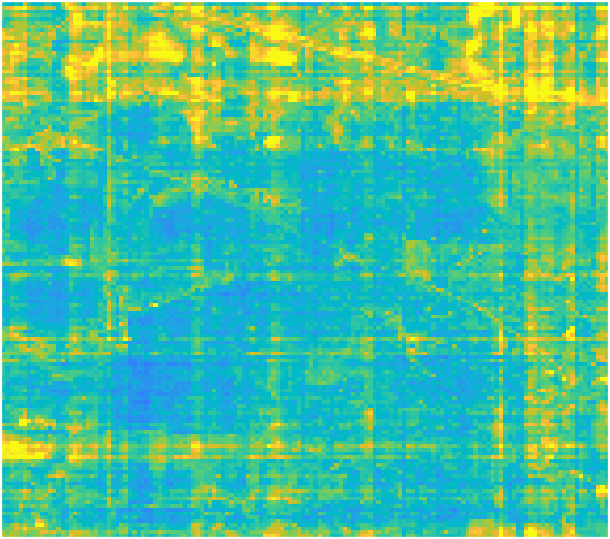}&			\includegraphics[height=2.3cm,width=2.3cm]{fig/Indina_res_SRI}&		\includegraphics[height=2.3cm,width=0.6cm]{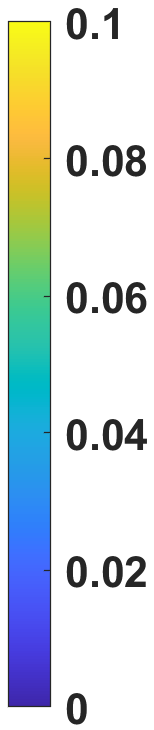}\\
		FUSE & STEREO & Ours & SRI
	\end{tabular}
	\caption{Results of Indian Pines reconstructions by  solving \eqref{hyper}. The first row: the 125-th band of recovered SRIs. The second row: the 125-th band of corresponding residual images. The last row: the SAM maps.}\label{InRs} 
\end{figure}
\begin{figure}[h!]
	\centering
	\setlength\tabcolsep{2pt}
	\begin{tabular}{cccccc}			\includegraphics[height=2.3cm,width=2.3cm]{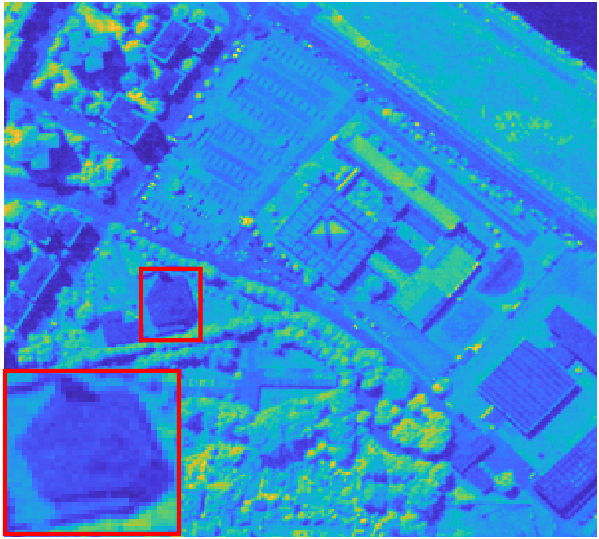}&			\includegraphics[height=2.3cm,width=2.3cm]{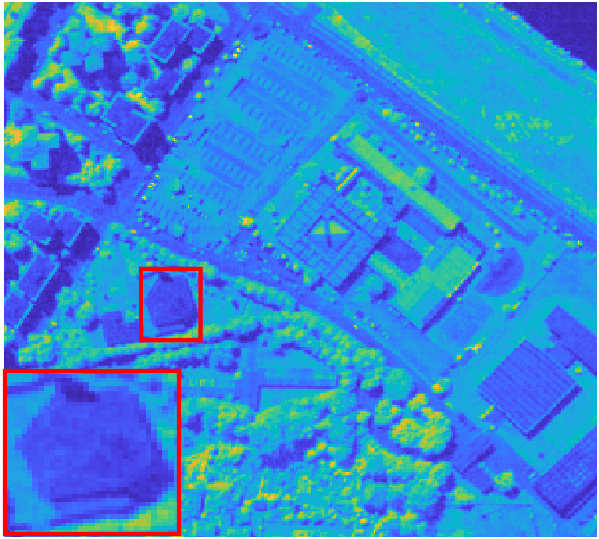}&
		\includegraphics[height=2.3cm,width=2.3cm]{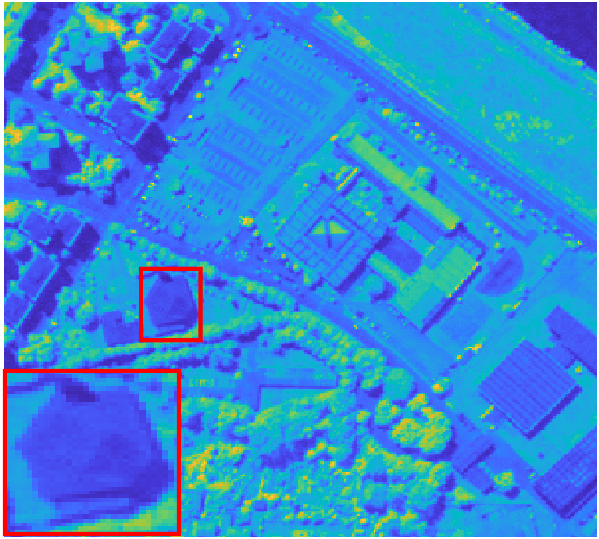} &
		\includegraphics[height=2.3cm,width=2.3cm]{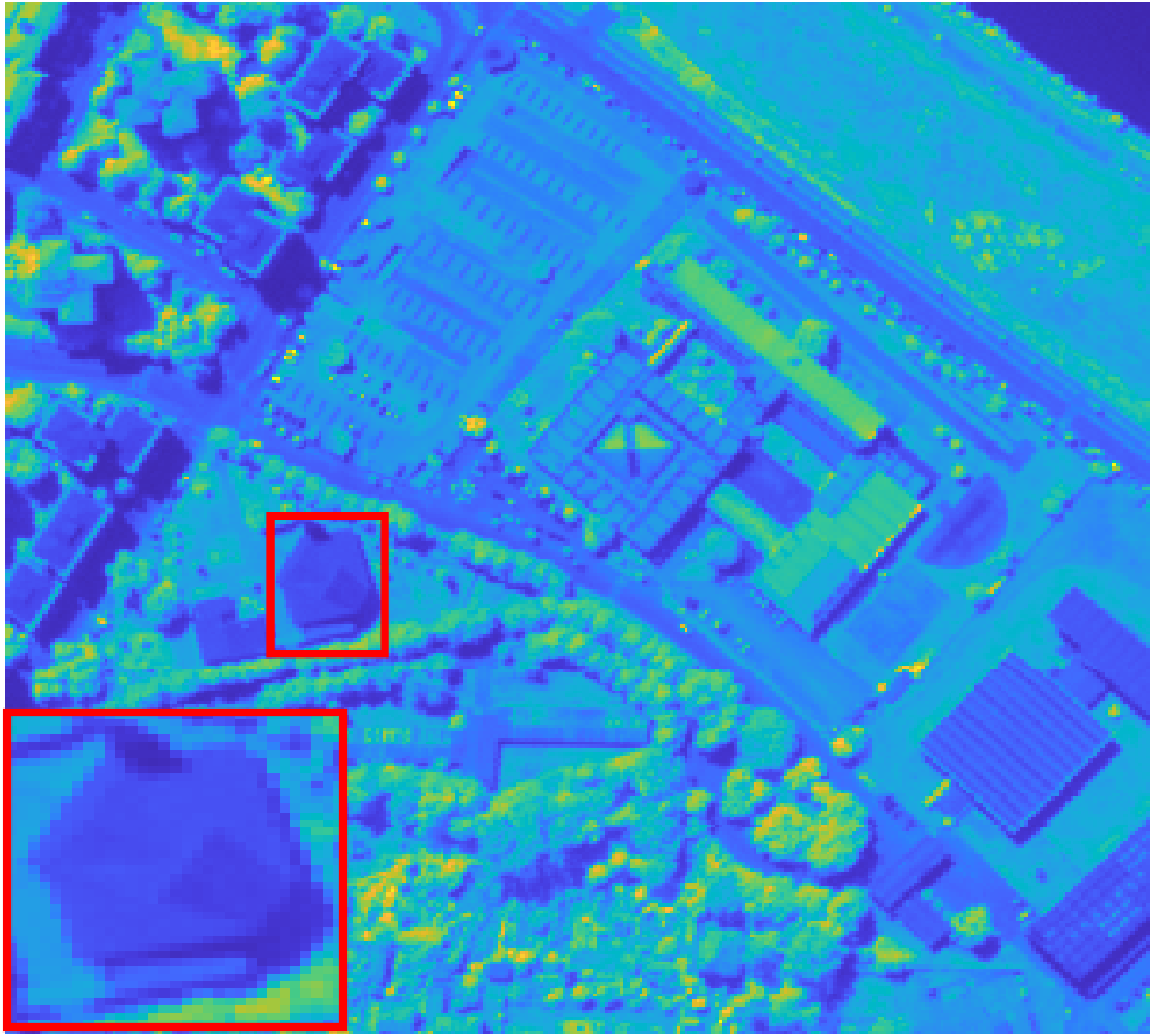}&
		\includegraphics[height=2.3cm,width=0.6cm]{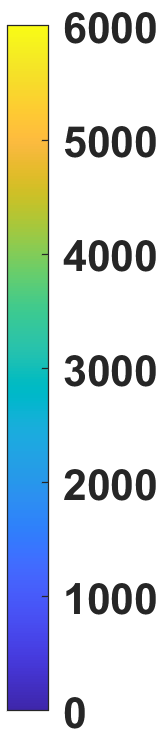}\\
		\includegraphics[height=2.3cm,width=2.3cm]{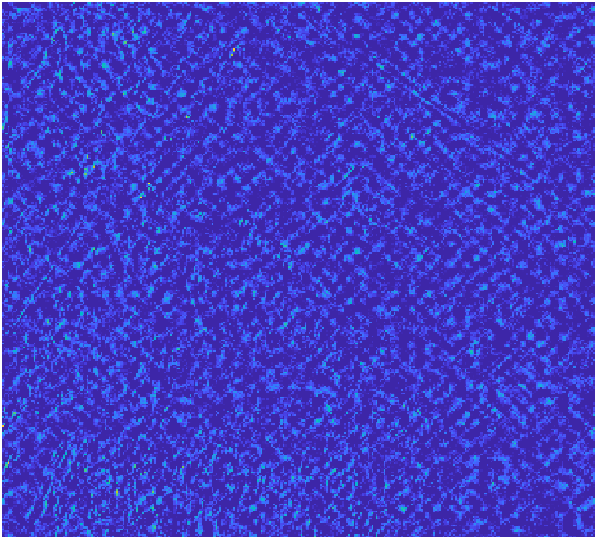}&
		\includegraphics[height=2.3cm,width=2.3cm]{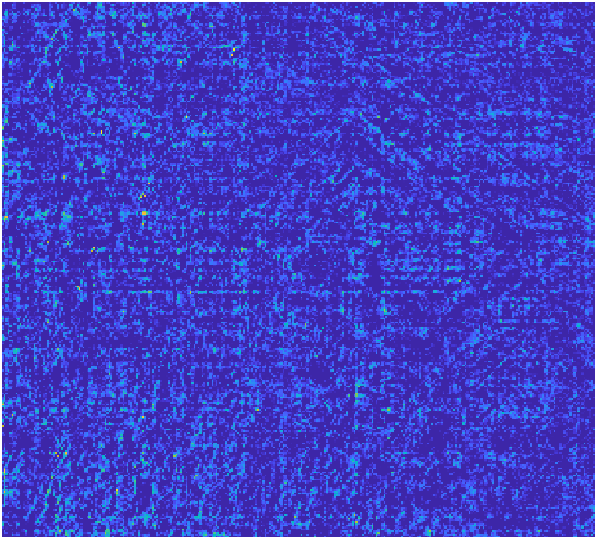}&
		\includegraphics[height=2.3cm,width=2.3cm]{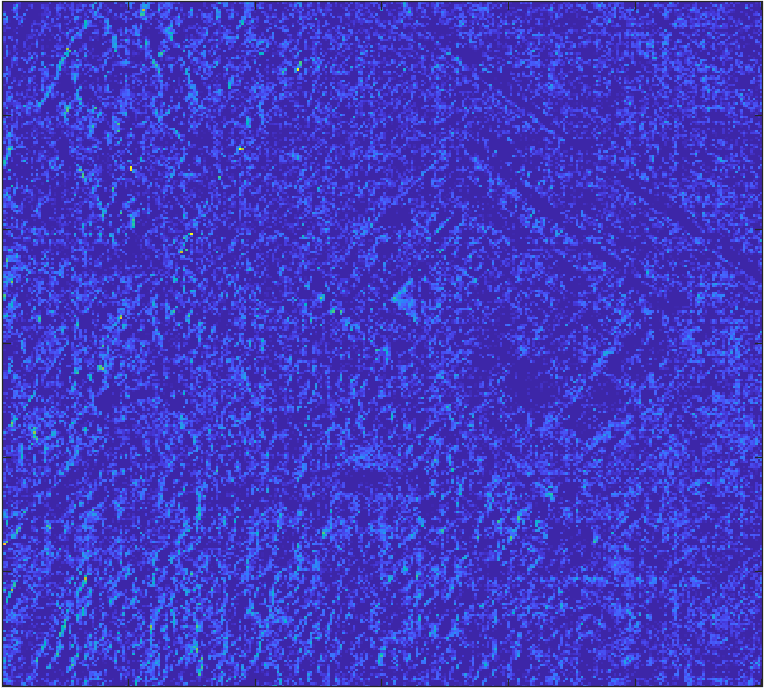}&
		\includegraphics[height=2.3cm,width=2.3cm]{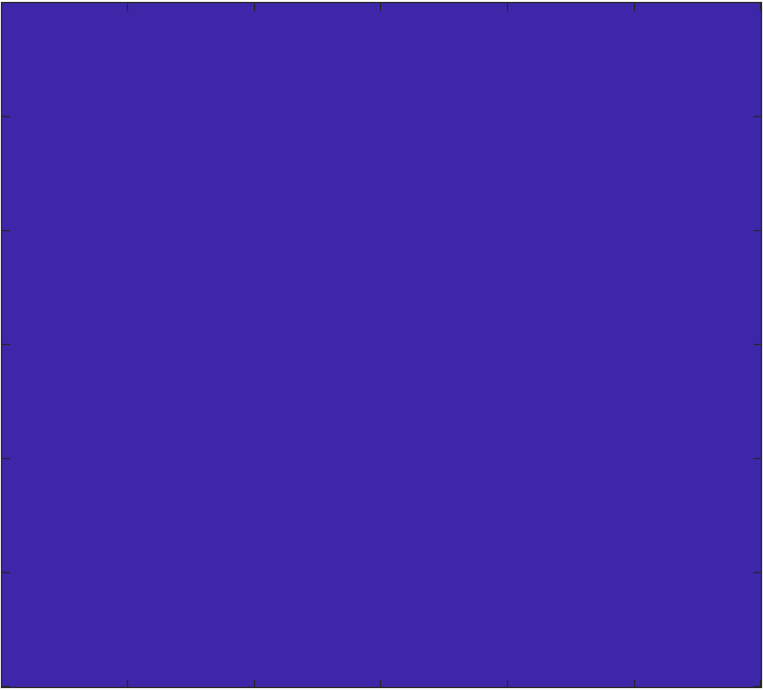}&
		\includegraphics[height=2.3cm,width=0.6cm]{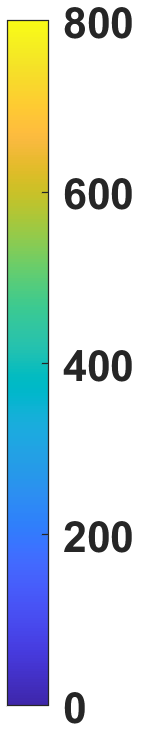}\\
		\includegraphics[height=2.3cm,width=2.3cm]{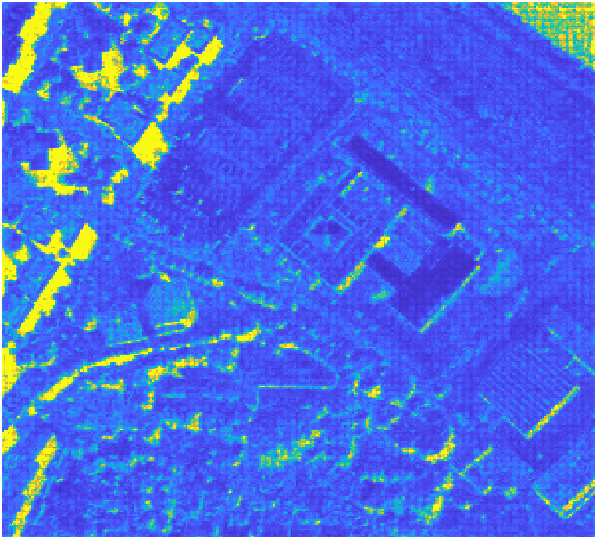}&
		\includegraphics[height=2.3cm,width=2.3cm]{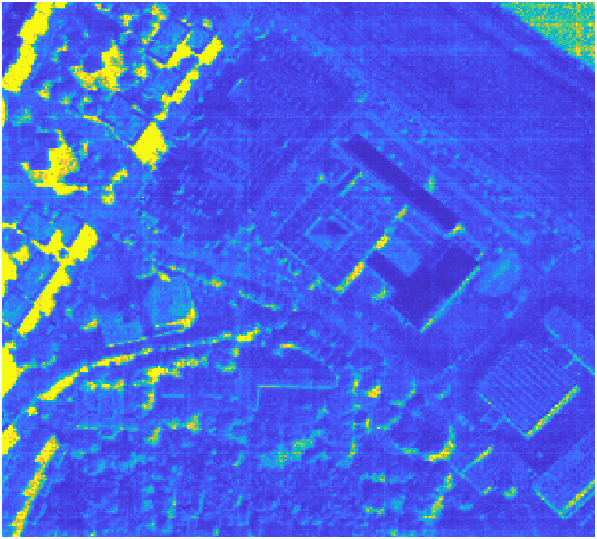}&		
		\includegraphics[height=2.3cm,width=2.3cm]{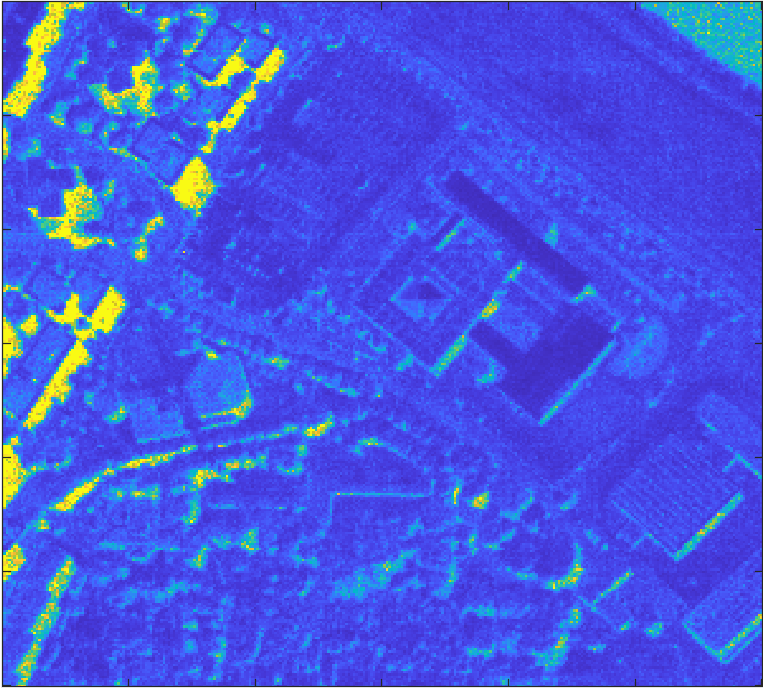}&
		\includegraphics[height=2.3cm,width=2.3cm]{fig/Pavia_res_SRI}&			\includegraphics[height=2.3cm,width=0.6cm]{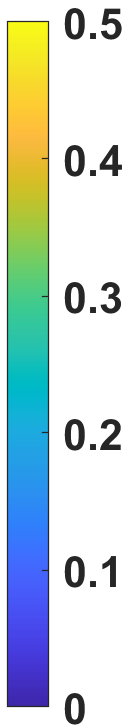}\\
		FUSE & STEREO & Ours & SRI
	\end{tabular}
	\caption{Results of Pavia Centre reconstructions by  solving \eqref{hyper}. The first row: the 100-th band of recovered SRIs. The second row: the 100-th band of corresponding residual images. The last row: the SAM maps.}\label{PaRs}	
\end{figure}

\newpage
\section{Concluding remarks}\label{sect6}
In this paper, we developed an extended ASAP (eASAP) for solving nonconvex nonsmooth optimization problem with the multiblock nonseparable structure. Under some mild assumptions, we analyzed the convergence (rate) of eASAP. Concretely, by the blockwise restricted prox-regularity of $H$ in Assumption \ref{assum2}, we proved that any limit point of the sequence generated by eASAP is a critical point of \eqref{geP}. Furthermore, we established  the global convergence  when the objective fulfills  Assumption \ref{assum:Aub} and the K{\L} property. Finally, we built upon the sublinear convergence rate based on the error function. Our novel convergence analysis covers a variety of nonconvex nonsmooth nonseparable coupling functions, which further extends the range of the model. Besides, as a peculiarity,   eASAP can be reduced to  ASAP when $s=t=1$.  Numerical simulations on multimodal data fusion demonstrate the compelling performance of the proposed method.

\bibliographystyle{is-abbrv}
\bibliography{mybib0601.bib}

\begin{thebibliography}{10}
\ifx \showCODEN  \undefined \def \showCODEN #1{CODEN #1}  \fi
\ifx \showISBN   \undefined \def \showISBN  #1{ISBN #1}   \fi
\ifx \showISSN   \undefined \def \showISSN  #1{ISSN #1}   \fi
\ifx \showLCCN   \undefined \def \showLCCN  #1{LCCN #1}   \fi
\ifx \showPRICE  \undefined \def \showPRICE #1{#1}        \fi
\ifx \showURL    \undefined \def \showURL {URL }          \fi
\ifx \path       \undefined \input path.sty               \fi
\ifx \ifshowURL \undefined
     \newif \ifshowURL
     \showURLtrue
\fi

\bibitem{attouch2010proximal}
H.~Attouch, J.~Bolte, P.~Redont, and A.~Soubeyran.
\newblock Proximal alternating minimization and projection methods for
  nonconvex problems: An approach based on the {K}urdyka-{{\L}}ojasiewicz
  inequality.
\newblock {\em Math. Oper. Res.}, 35\penalty0 (2):\penalty0 438--457, 2010.

\bibitem{auslender1992asymptotic}
A.~Auslender.
\newblock Asymptotic properties of the {F}enchel dual functional and
  applications to decomposition problems.
\newblock {\em J. Optim. Theory Appl.}, 73:\penalty0 427--449, 1992.

\bibitem{beck2016alternating}
A.~Beck, S.~Sabach, and M.~Teboulle.
\newblock An alternating semiproximal method for nonconvex regularized
  structured total least squares problems.
\newblock {\em SIAM J. Matrix Anal. Appl.}, 37\penalty0 (3):\penalty0
  1129--1150, 2016.

\bibitem{bertsekas2015parallel}
D.~Bertsekas and J.~Tsitsiklis.
\newblock {\em Parallel and distributed computation: Numerical methods}.
\newblock Athena Scientific, Nashua, NH, 2003.

\bibitem{bertsekas1997nonlinear}
D.~P. Bertsekas.
\newblock {\em Nonlinear programming}.
\newblock Athena Scientific, Belmont, MA, 1995.

\bibitem{bolte2010characterizations}
J.~Bolte, A.~Daniilidis, O.~Ley, and L.~Mazet.
\newblock Characterizations of {{\L}}ojasiewicz inequalities: Subgradient
  flows, talweg, convexity.
\newblock {\em Trans. Amer. Math. Soc.}, 362\penalty0 (6):\penalty0 3319--3363,
  2010.

\bibitem{bolte2014proximal}
J.~Bolte, S.~Sabach, and M.~Teboulle.
\newblock Proximal alternating linearized minimization for nonconvex and
  nonsmooth problems.
\newblock {\em Math. Program.}, 146\penalty0 (1-2):\penalty0 459--494, 2014.

\bibitem{bot2019proximal}
R.~I. Bo{\c{t}}, E.~R. Csetnek, and D.-K. Nguyen.
\newblock A proximal minimization algorithm for structured nonconvex and
  nonsmooth problems.
\newblock {\em SIAM J. Optim.}, 29\penalty0 (2):\penalty0 1300--1328, 2019.

\bibitem{boct2020proximal}
R.~I. Bo{\c{t}} and D.-K. Nguyen.
\newblock The proximal alternating direction method of multipliers in the
  nonconvex setting: {C}onvergence analysis and rates.
\newblock {\em Math. Oper. Res.}, 45\penalty0 (2):\penalty0 682--712, 2020.

\bibitem{chatzichristos2022early}
C.~Chatzichristos, E.~Kofidis, W.~Van~Paesschen, L.~De~Lathauwer,
  S.~Theodoridis, and S.~Van~Huffel.
\newblock Early soft and flexible fusion of electroencephalography and
  functional magnetic resonance imaging via double coupled matrix tensor
  factorization for multisubject group analysis.
\newblock {\em Hum Brain Mapp.}, 43\penalty0 (4):\penalty0 1231--1255, 2022.

\bibitem{chung2010efficient}
J.~Chung and J.~G. Nagy.
\newblock An efficient iterative approach for large-scale separable nonlinear
  inverse problems.
\newblock {\em SIAM J. Sci. Comput.}, 31\penalty0 (6):\penalty0 4654--4674,
  2010.

\bibitem{ding2020hyperspectral}
M.~Ding, X.~Fu, T.-Z. Huang, J.~Wang, and X.-L. Zhao.
\newblock Hyperspectral super-resolution via interpretable block-term tensor
  modeling.
\newblock {\em IEEE J. Sel. Top Signal Process}, 15\penalty0 (3):\penalty0
  641--656, 2020.

\bibitem{driggs2021stochastic}
D.~Driggs, J.~Tang, J.~Liang, M.~Davies, and C.-B. Schonlieb.
\newblock A stochastic proximal alternating minimization for nonsmooth and
  nonconvex optimization.
\newblock {\em SIAM J. Imaging Sci.}, 14\penalty0 (4):\penalty0 1932--1970,
  2021.

\bibitem{farias2016exploring}
R.~C. Farias, J.~E. Cohen, and P.~Comon.
\newblock Exploring multimodal data fusion through joint decompositions with
  flexible couplings.
\newblock {\em IEEE Trans. Signal Process.}, 64\penalty0 (18):\penalty0
  4830--4844, 2016.

\bibitem{gao2023alternating}
X.~Gao, X.~Cai, X.~Wang, and D.~Han.
\newblock An alternating structure-adapted bregman proximal gradient descent
  algorithm for constrained nonconvex nonsmooth optimization problems and its
  inertial variant.
\newblock {\em J. Global Optim.}, pages 1--24, 2023.

\bibitem{hu2023convergence}
Y.~Hu and X.~Liu.
\newblock The convergence properties of infeasible inexact proximal alternating
  linearized minimization.
\newblock {\em Sci. China Math.}, pages 1--26, 2023.

\bibitem{jia2024stochastic}
Z.~Jia, W.~Zhang, X.~Cai, and D.~Han.
\newblock Stochastic alternating structure-adapted proximal gradient descent
  method with variance reduction for nonconvex nonsmooth optimization.
\newblock {\em Math. Comp.}, 93\penalty0 (348):\penalty0 1677--1714, 2024.

\bibitem{jourani2012c1}
A.~Jourani, L.~Thibault, and D.~Zagrodny.
\newblock C\textsuperscript{1,$\omega$ ({\textperiodcentered})}-regularity and
  {L}ipschitz-like properties of subdifferential.
\newblock {\em Proc. London Math. Soc.}, 105\penalty0 (1):\penalty0 189--223,
  2012.

\bibitem{kanatsoulis2018hyperspectral}
C.~I. Kanatsoulis, X.~Fu, N.~D. Sidiropoulos, and W.-K. Ma.
\newblock Hyperspectral super-resolution: A coupled tensor factorization
  approach.
\newblock {\em IEEE Trans. Image Process.}, 66\penalty0 (24):\penalty0
  6503--6517, 2018.

\bibitem{Kolda2009}
T.~G. Kolda and B.~W. Bader.
\newblock Tensor decompositions and applications.
\newblock {\em SIAM Rev.}, 51\penalty0 (3):\penalty0 455--500, 2009.

\bibitem{levy1997characterizing}
A.~Levy and R.~Poliquin.
\newblock Characterizing the single-valuedness of multifunctions.
\newblock {\em Set-Valued Var. Anal.}, 5:\penalty0 351--364, 1997.

\bibitem{levy2000stability}
A.~B. Levy, R.~A. Poliquin, and R.~T. Rockafellar.
\newblock Stability of locally optimal solutions.
\newblock {\em SIAM J. Optim.}, 10\penalty0 (2):\penalty0 580--604, 2000.

\bibitem{GuoPong2015}
G.~Li and T.~K. Pong.
\newblock Global convergence of splitting methods for nonconvex composite
  optimization.
\newblock {\em SIAM J. Optim.}, 25\penalty0 (4):\penalty0 2434--2460, 2015.

\bibitem{li2018calculus}
G.~Li and T.~K. Pong.
\newblock Calculus of the exponent of {K}urdyka--{{\L}}ojasiewicz inequality
  and its applications to linear convergence of first-order methods.
\newblock {\em Found. Comput. Math.}, 18\penalty0 (5):\penalty0 1199--1232,
  2018.

\bibitem{li2015accelerated}
H.~Li and Z.~Lin.
\newblock Accelerated proximal gradient methods for nonconvex programming.
\newblock {\em Adv. Neural Inf. Process. Syst.}, 28, 2015.

\bibitem{li2018fusing}
S.~Li, R.~Dian, L.~Fang, and J.~M. Bioucas-Dias.
\newblock Fusing hyperspectral and multispectral images via coupled sparse
  tensor factorization.
\newblock {\em IEEE Trans. Image Process.}, 27\penalty0 (8):\penalty0
  4118--4130, 2018.

\bibitem{mordukhovich2005subgradient}
B.~S. Mordukhovich and N.~M. Nam.
\newblock Subgradient of distance functions with applications to {L}ipschitzian
  stability.
\newblock {\em Math. Program.}, 104\penalty0 (2):\penalty0 635--668, 2005.

\bibitem{nie2018complete}
J.~Nie, Z.~Yang, and X.~Zhang.
\newblock A complete semidefinite algorithm for detecting copositive matrices
  and tensors.
\newblock {\em SIAM J. Optim.}, 28\penalty0 (4):\penalty0 2902--2921, 2018.

\bibitem{nikolova2019alternating}
M.~Nikolova and P.~Tan.
\newblock Alternating structure-adapted proximal gradient descent for nonconvex
  nonsmooth block-regularized problems.
\newblock {\em SIAM J. Optim.}, 29\penalty0 (3):\penalty0 2053--2078, 2019.

\bibitem{color2015}
F.~Pierre, J.-F. Aujol, A.~Bugeau, N.~Papadakis, and V.-T. Ta.
\newblock Luminance-chrominance model for image colorization.
\newblock {\em SIAM J. Imaging Sci.}, 8\penalty0 (1):\penalty0 536--563, 2015.

\bibitem{pock2016inertial}
T.~Pock and S.~Sabach.
\newblock Inertial proximal alternating linearized minimization (i{PALM}) for
  nonconvex and nonsmooth problems.
\newblock {\em SIAM J. Imaging Sci.}, 9\penalty0 (4):\penalty0 1756--1787,
  2016.

\bibitem{qi2021triple}
L.~Qi, Y.~Chen, M.~Bakshi, and X.~Zhang.
\newblock Triple decomposition and tensor recovery of third order tensors.
\newblock {\em SIAM J. Matrix Anal. Appl.}, 42\penalty0 (1):\penalty0 299--329,
  2021.

\bibitem{roald2022admm}
M.~Roald, C.~Schenker, V.~D. Calhoun, T.~Adali, R.~Bro, J.~E. Cohen, and
  E.~Acar.
\newblock An {AO-ADMM} approach to constraining {PARAFAC2} on all modes.
\newblock {\em SIAM J. Math. Data Sci.}, 4\penalty0 (3):\penalty0 1191--1222,
  2022.

\bibitem{Varia-Ana}
R.~T. Rockafellar and R.~J. Wets.
\newblock {\em Variational Analysis}.
\newblock Springer, Berlin, 1998.

\bibitem{schenker2020flexible}
C.~Schenker, J.~E. Cohen, and E.~Acar.
\newblock A flexible optimization framework for regularized matrix-tensor
  factorizations with linear couplings.
\newblock {\em IEEE J. Sel. Top Signal Process}, 15\penalty0 (3):\penalty0
  506--521, 2020.

\bibitem{tan2019inertial}
P.~Tan, F.~Pierre, and M.~Nikolova.
\newblock Inertial alternating generalized forward--backward splitting for
  image colorization.
\newblock {\em J. Math. Imaging Vision}, 61:\penalty0 672--690, 2019.

\bibitem{vervliet2019exploiting}
N.~Vervliet, O.~Debals, and L.~De~Lathauwer.
\newblock Exploiting efficient representations in large-scale tensor
  decompositions.
\newblock {\em SIAM J. Sci. Comput.}, 41\penalty0 (2):\penalty0 A789--A815,
  2019.

\bibitem{wang2019global}
Y.~Wang, W.~Yin, and J.~Zeng.
\newblock Global convergence of {ADMM} in nonconvex nonsmooth optimization.
\newblock {\em J. Sci. Comput.}, 78:\penalty0 29--63, 2019.

\bibitem{wei2015fast}
Q.~Wei, N.~Dobigeon, and J.-Y. Tourneret.
\newblock Fast fusion of multi-band images based on solving a {S}ylvester
  equation.
\newblock {\em IEEE Trans. Image Process.}, 24\penalty0 (11):\penalty0
  4109--4121, 2015.

\bibitem{xu2013block}
Y.~Xu and W.~Yin.
\newblock A block coordinate descent method for regularized multiconvex
  optimization with applications to nonnegative tensor factorization and
  completion.
\newblock {\em SIAM J. Imaging Sci.}, 6\penalty0 (3):\penalty0 1758--1789,
  2013.

\bibitem{yang2017alternating}
L.~Yang, T.~K. Pong, and X.~Chen.
\newblock Alternating direction method of multipliers for a class of nonconvex
  and nonsmooth problems with applications to background/foreground extraction.
\newblock {\em SIAM J. Imaging Sci.}, 10\penalty0 (1):\penalty0 74--110, 2017.

\bibitem{yang2022some}
X.~Yang and L.~Xu.
\newblock Some accelerated alternating proximal gradient algorithms for a class
  of nonconvex nonsmooth problems.
\newblock {\em J. Global Optim.}, pages 1--26, 2022.

\end{thebibliography}
\end{document}